\documentclass[english,11pt]{article}
\usepackage{authblk}
\usepackage[T1]{fontenc}
\usepackage[latin9]{inputenc}
\usepackage{amscd}
\usepackage{amsthm,amsmath,amsfonts,amssymb}
\usepackage{verbatim}
\usepackage{lmodern}
\usepackage{graphicx}
\usepackage{enumerate}
\usepackage{color}
\makeatletter
\numberwithin{equation}{section}
\numberwithin{figure}{section}
\theoremstyle{plain}
\newtheorem{thm}{\protect\theoremname}
  \theoremstyle{plain}
  \newtheorem{lemma}[thm]{\protect\lemmaname}
    \newtheorem{prop}[thm]{\protect\propname}

\newtheorem*{thm*}{Theorem}
\newtheorem{definition}[thm]{Definition}

\numberwithin{thm}{section}

\newtheorem{cor}[thm]{Corollary}

\theoremstyle{remark}
\newtheorem*{rem}{Remark}
\makeatother

\usepackage{babel}
\providecommand{\propname}{Proposition}

\providecommand{\lemmaname}{Lemma}
\providecommand{\theoremname}{Theorem}

\newcommand{\imag}{\operatorname{Im} \,}
\newcommand{\real}{\operatorname{Re} \,}
\renewcommand{\Im}{\imag}
\renewcommand{\Re}{\real}

\newcommand{\ee}{\epsilon}

\newcommand{\DD}{\mathbb{D}}

\newcommand{\ZZ}{\mathbb{Z}}

\newcommand{\E}{\mathbf{E}}

\newcommand{\Prob}{\mathbf{P}}

\newcommand{\hcap}{\operatorname{hcap}}
\newcommand{\diam}{\operatorname{diam}}

\newcommand{\dist}{\operatorname{dist}}

\newcommand{\ball}{\mathcal{B}}

\newcommand{\Z}{{\mathbb Z}}

\newcommand{\R}{{\mathbb{R}}}
\newcommand{\C}{{\mathbb C}}
\newcommand{\rad}{{\rm rad}}

\newcommand{\SLE}{\text{\tiny SLE}}
\newcommand {\G} {{\mathcal G}}

\newcommand{\LERW}{\text{\tiny LERW}}

\def \Half {{\mathbb H}}

\def \Disk {{\mathbb D}}
\def \F {{\cal F}}

\newcommand {{\wind}} {{\rm wind}}

\newcommand {{\Sine}}{S}

\let \setminus \smallsetminus
\let \le \leqslant
\let \leq \leqslant
\let \ge \geqslant
\let \geq \geqslant
\let \epsilon \varepsilon
\let \phi \varphi

\newcommand{\whoknows}  {{\mathcal A}}
\newcommand{\soup}  {{\mathcal J}}
\newcommand{\paths} {{\mathcal K}}
\newcommand{\saws}{\mathcal{W}}
\newcommand {\eset}{{\emptyset}}
\newcommand {\Square} {{\mathcal S}}
\newcommand {\Cont}{{\rm Cont}}

\newcommand {\maxsle} {{J_\SLE}}
\newcommand {\maxlerw} {{J_\LERW}}
\newcommand{\maxsletwo} {{J_\SLE^2}}
\newcommand {\maxlerwtwo}{{J_\LERW^2}}
\newcommand {\opensle}  {O^{\SLE}}
\newcommand {\openlerw}  {O^{\LERW}}

\AtEndDocument{\bigskip{\footnotesize%
  (Lawler) \textsc{Department of Mathematics, University of Chicago} \par  
  \textit{E-mail address}: \texttt{lawler@math.uchicago.edu} \par
  
  \addvspace{\medskipamount}
  (Viklund) \textsc{Department of Matematics, KTH Royal Institute of Technology} \par
  \textit{E-mail address}: \texttt{fredrik.viklund@math.kth.se}
}}

\title{Convergence of Loop-Erased Random Walk in the Natural Parametrization}
\author{Gregory F. Lawler}
\affil{University of Chicago}

\author{Fredrik Viklund}
\affil{KTH Royal Institute of Technology}

\begin{document}

\maketitle
\begin{abstract}
Loop-erased random walk, abbreviated LERW, is one of the most well-studied critical lattice models. It is the self-avoiding random walk one gets after erasing the loops from a simple random walk in order or alternatively by considering the branches in a uniformly chosen spanning tree. 
This paper proves that planar LERW parametrized by renormalized length converges in the lattice size scaling limit to SLE$_2$ parametrized by $5/4$-dimensional Minkowski content. In doing this we also provide a method for proving similar convergence results for other models converging to SLE. Besides the main theorem, several of our results about LERW are of independent interest: for example, two-point estimates, estimates on maximal content, and a ``separation lemma''. 
\end{abstract}
\tableofcontents
\section{Introduction}
\subsection{Introduction}
The hypothesis that critical two-dimensional lattice models should have conformally invariant scaling limits was formulated in the physics community in the 1980s. Starting with \cite{BPZ} conformal field theory (CFT) was developed to exploit conformal invariance and subsequently applied to many lattice models, producing predictions of, e.g., critical exponents and correlation functions. 

The Schramm-Loewner evolution (SLE) processes \cite{Schramm_LERW} provide a precise mathematical approach by describing scaling limits of random cluster interfaces and self-avoiding walks in the lattice models. To date, convergence in a sense described below, and in particular conformal invariance, has been established in several cases: loop-erased random walk (LERW), the uniform spanning tree, critical percolation, the Ising model, and the discrete Gaussian free field \cite{LSW04, smirnov-perc, smirnov-ising, SS}. The uniform measure on self-avoiding walks is strongly believed to belong to this collection of models, but whether it actually does remains one of the most interesting and apparently difficult open problems in probability. Once a convergence result is established one can use SLE computations to rigorously derive properties such as critical exponents or dimensions of the discrete interfaces, see, e.g., \cite{smirnov_werner, lsw_arm_exponents, Masson}. Some field theoretic statements may also be interpreted and given probabilistic and geometric meaning, see, e.g., \cite{friedrich-werner, kang-makarov} and the references in the latter.

SLE curves are constructed using Loewner's differential equation. It gives the dynamics of a family of Riemann maps from a reference domain onto a continuously decreasing family of simply connected domains. Under favorable circumstances, as in the case of SLE, there is a non-crossing, continuous curve such that one gets the decreasing domains by taking the complements of the growing curve. This \emph{Loewner curve} comes equipped with a particular  parametrization by capacity which it inherits from the Loewner equation. Studying SLE in this parametrization is practical for many problems and we have information about, e.g., sharp Holder exponents, continuity properties, and finer multifractal relations  \cite{RS, lind-holder, JVL, JVL2, JVRW}. Before the present paper all SLE convergence results we know of consider a discrete curve reparametrized by capacity, and proves convergence in that parametrization. This is sufficient to study many properties of  discrete models converging to SLE. 

However, information is lost when reparametrizing the discrete curve. A more detailed analysis (see \cite{benoist-dumaz-werner} for an example) is possible by considering the discrete process parametrized by length, in what is sometimes called its \emph{natural parametrization}. By this we mean that the curve traverses each lattice edge in the same amount of time. Since the limiting trace is fractal, one needs to rescale so that whole curve in a smooth
bounded domain  is traversed in time of order $1$. It then seems reasonable to expect that the discrete curve in its natural parametrization converges to SLE equipped with a different parmatrization than capacity. Indeed, this is widely believed to be true in all the cases where convergence to SLE is known. 

The SLE curve with parameter $\kappa \in (0,8)$ is a random fractal of almost sure dimension $d=1+\kappa/8$. With the length rescaling of the discrete curve in mind, we are looking for a parametrization $\gamma(t)$ such that $r\gamma[0,t]$ equals $\gamma[0,r^dt]$ in distribution. That is, one in which it takes about time $O(r^d)$ for the curve to travel distance $r$. (Compare this with the discrete interpretation of dimension.)  It would be natural to try  to parametrize by $d$-dimensional Hausdorff content, but it turns out that this does not work: the Hausdorff content is $0$ almost surely \cite{rezaei-hausdorff}.  What does work is to parametrize by $d$-dimensional Minkowski content, so that 
\[
\lim_{\ee \to 0+} \ee^{2-d} \text{Area}\left(\left\{z: \, \dist(z, \gamma[0,t]) \le \ee \right\} \right) = t
\]
holds for all $t \ge 0$. To make sense of this requires work, see \cite{LR}. The resulting parametrization is also called the natural parametrization of SLE$_\kappa$.  The first construction \cite{lawler_sheffield} did not use Minkowski content, but went via the Doob-Meyer decomposition of a supermartingale obtained by integrating the SLE Green's function. The ``natural time'' was defined as the increasing part in this decomposition. Both approaches are important for this paper.  

Our main theorem is that LERW parametrized by renormalized length converges to SLE$_2$ parametrized by Minkowski content. Let us give a rough statement.  Fix an analytic simply connected domain $D$ with distinct boundary points $a,b$ and for $N=1,2,\ldots,$ a lattice spacing $N^{-1}$. We take $D_{N}$ to be an appropriate simply connected component of $(N^{-1} \ZZ^2) \cap D$ with boundary edges $a_N, b_N$ approximating $a,b$. We will measure distance between curves using a metric on parametrized curves defined as follows: If $\gamma^1:[s_1,t_1] \rightarrow \C$ and $\gamma^2:[s_2,t_2]
 \rightarrow \C$ are continuous curves, then
\begin{equation} \label{metric}     \rho(\gamma^1,\gamma^2) = \inf\left[\sup_{s_1 \leq t \leq t_1}
   \left|\alpha(t) - t \right| +  \sup_{s_1 \leq t \leq t_1}
   \left|\gamma^2\left(\alpha(t) \right) -\gamma^1( t) \right|\right] ,\end{equation}
 where the infimum is over all increasing homeomorphisms
 $\alpha:[s_1,t_1] \rightarrow [s_2,t_2]$.
\begin{thm}\label{main-thm-rough}
There is a universal (but presumed lattice dependent) constant $\check{c}$ and an explicit sequence $\ee_N \to 0+$ as $N \to \infty$ such that the following holds.  For each $N$, let $\eta(t), t\in [0, T_\eta],$ be LERW in $D_N$ from $a_N$ to $b_N$ viewed as a continuous curve parametrized so that each edge is traversed in time $\check{c}N^{-5/4}$. Let $\gamma(t), t \in [0, T_\gamma],$ be chordal SLE$_2$ in $D$ from $a$ to $b$ parametrized by $5/4$-dimensional Minkowski content. There is a coupling of $\eta$ and $\gamma$ such that
\[
\Prob\left\{\rho \left(\eta,\gamma \right) > \ee_N \right\} < \ee_N.
\]
In particular, $\eta$ converges to $\gamma$ weakly with respect to the metric $\rho$.
\end{thm}
See Section~\ref{sect:main_complete} and Theorem~\ref{thm:main_complete} in particular for a complete statement, but we mention here that we obtain an estimate on the convergence rate and one may take $\ee_N = c \, \left(\log N \right)^{-1/60}$ in the theorem, where $c$ is a constant depending on the domain configuration. The recent paper \cite{benoist-dumaz-werner} gives an already worked out application of Theorem~\ref{main-thm-rough}. See also \cite{AMK, BCK, kennedy_montecarlo} for additional discussions of discrete models and their relations to SLE in the natural parametrization. There is a version of Theorem~\ref{main-thm-rough} for radial LERW. The proof requires some extra work which is done in \cite{LV_lerw_radial_note}.

The starting point of the proof is the main result of \cite{BLV}:  the renormalized probability that LERW uses a fixed interior edge converges towards the SLE$_2$ Green's function. (See \cite{BLV} and the references therein for a discussion of the LERW growth exponent and related work.) It is important that this result holds for general domains and that we have estimates on the convergence rate. With these facts in hand, the next step is to revisit the convergence in the capacity parametrization \cite{LSW04}. We need to work with a slightly different coupling than the ones previously constructed and we need to be careful about certain measurability properties. We carry out the needed work in a separate paper \cite{LV_lerw_chordal_note}. There we give  proofs using the Green's function as martingale observable and derive quantitative bounds on error terms. It is convenient to work with a discrete difference version of Loewner's equation and we discuss this and develop the required estimates in \cite{LV_lerw_chordal_note}. 

Given these results we thus have a coupling of LERW with SLE$_2$ in which with large probability the Loewner chains and paths are uniformly close when parametrized by capacity. The main  goal of this paper is to show that in this coupling, uniformly as the capacity of the paths is varied, the renormalized length of the LERW is nearly the same as the Minkowski content of the SLE, except for an event of small probability. In order to do this we consider martingales given by taking conditional expectations of the total number of steps and the total content of the LERW and SLE, respectively, given the growing coupled curves sampled at mesoscopic capacity increments. The idea is to look at the Doob-Meyer decompositions of the martingales and use the fact that the Green's functions are very close in order to show that the supermartingale parts are close. From this it is possible to deduce that the increasing parts, that is, the naural times, must also be close. A significant complication is to control the contribution of regions in the complement of the curves where the result of \cite{BLV} gives only trivial information. We handle
 this by discretizing and at each step restricting attention to ``open'' squares for which certain geometric estimates hold that allow us to estimate using \cite{BLV}. The contribution of ``closed'' squares is shown to be negligible.

Although many of our estimates are specific to LERW, our general method of proof is not. We do not see any obstructions for it to work for other models as well, if (and this is a big if!) the analogs of the Green's function convergence and  second moment estimates for the discrete model
are available. 
 
\subsection{Overview of the proof of Theorem~\ref{main-thm-rough}}  \label{sect:overview}
Let us now be more precise about what is needed to carry out this
idea and where in the paper it is done.  We will give more detailed definitions in Section~\ref{sect:preliminaries}.
\begin{itemize}
\item{We fix an analytic simply connected domain $D$ with distinct boundary points $a',b'$.}
\item{For any lattice spacing $N^{-1}$ we approximate $(D,a',b')$ by a triple $(A,a,b)$ where $A = A(D,N) \subset \Z^2=\Z + i\Z$ is a
simply connected lattice set with boundary edges $a,b$ near $Na',Nb'$.  We often identify edges with their midpoints.}
\item{
We identify each $\zeta \in \Z^2 $ with the closed square
$\Square_\zeta$ of side length $1$ centered at $\zeta$. We let $D_A$ be the simply connected complex domain generated by $A$ by taking the (interior of the) union of the squares corresponding the points of $A$. Note that $N^{-1}D_A$ approximates the domain $D$. We will sometimes slightly abuse language and refer to $D_A$ as a ``union of squares domain''.
}
\item{We write $\check a = N^{-1}a, \check b = N^{-1}b$, and  $\check D = \check{D}_A=N^{-1}D_A$ for the quantities scaled by
$N^{-1}$. As $N \to \infty$, $\check D$ converges to $D$ in the Carath\'eodory sense, and it is not hard to estimate the convergence rate. Indeed (see Lemma
\ref{mar5.lemma1}) there exists a conformal transformation $\psi: \check D\rightarrow D$
with $\psi(0) = 0, \psi'(0) > 0, $
\[   |\psi(z) - z| \leq \frac{c \, \log N}{N} , \qquad z \in \check D, \]
\[    |\psi'(z) - 1| \leq \frac{c }{N \, \dist(z, \partial \check D)},
\qquad z \in \check D, \qquad \dist(z, \partial D) \geq  \frac c N.\]}
\item{We fix a conformal transformation
\[   F:D_A \rightarrow \Half, \quad F(a) = 0, \quad F(b) = \infty.\]
Note that this map is defined only up to a final scaling.  We will consider the paths
only up to the time that their half plane capacity reaches $1$.  This half plane capacity is defined in terms of the image under $F$ and so depends on the scaling. We will be able
to consider the entire path in $D$ by varying the initial $F$. }
\item{
We choose a mesoscopic scale $h  = N^{-2u/3}$ where $u>0$ is the exponent (also denoted by $u$) in the error term
in the main estimate from \cite{BLV}, see \eqref{nov3.1}.  We choose the scale  this coarse so that
this error does not contribute significantly in estimates. Let $n_{0}= \lfloor h^{-1} \rfloor$. }
\item{ We
 grow a LERW in $A$ from $a$ to $b$ which we denote by $\eta$. We write $\Prob_{A,a,b}$ for the associated probability measure. We stop the path each
time its capacity has increased by $h$, and write $\eta^n$ for path stopped after $n$ mesoscopic increments.   By removing the vertices of $\eta^n$ from $A$ (taking an appropriate connected component if needed), we have a 
sequence of configurations $(A_0,a_0,b), (A_1,a_1,b), (A_2,a_2,b), \ldots$ with
$A_0 \supset A_1 \supset  \cdots$ and we let 
$D_n = D_{A_n}$. By mesoscopic capacity increment, we mean the
half-plane capacity of $\Half \setminus F(D_n)$
so that  $\hcap\left[\Half \setminus F(D_n)\right]
 \approx h \, n$.  }
 \item{
 We let $g_n: \Half \setminus  F(D_n) \rightarrow
 \Half$ be the conformal transformation with $g_n(z) = z + o(1),
 z \rightarrow \infty$; let $F_n = g_n \circ F$ for
 $g_n = g^n \circ \cdots \circ g^1$
 where $g^n$ is the corresponding transformation $g^n: F_{n-1}(D_{n-1}
 \setminus D_n) \rightarrow \Half$ normalized at infinity.  Let \[U_n
 = F_n(a_n), \quad \xi_n = U_n- U_{n-1}\]  so that $U_n$ is a discrete ``driving term'' for the LERW.}
 
\item{ Let $\Prob_n = \Prob_{A_n,a_n,b}$.   In an accompanying paper
\cite{LV_lerw_chordal_note} we use the LERW Green's function and Loewner
difference estimates to couple the LERW with an SLE$_2$.
To be more precise, we  
 find a standard Brownian motion $W_t$ and a sequence of stopping times
 $0 = \tau_0 < \tau_1 < \tau_2 < \cdots$ such that  except for an event
 of small probability,
  \[                 \max_{n \le n_0}|U_n - W_{\tau_n}| \leq c \, h^{1/5}. \] }
 
    \item{
 Given the Brownian motion, there is a corresponding SLE$_2$ path
 in $\Half$, that is, there is a  simple curve $ \gamma:[0,\infty)
 \rightarrow \Half$ and conformal maps $g_t^\SLE: \Half \setminus \gamma_t
 \rightarrow \Half$ satisfying
 \[     \partial_t g_t^\SLE(z) = \frac{1}{g_t^\SLE(z) - W_t}. \]
 Here we write $\gamma_t =  \gamma[0,t]$ for the trace in $\Half$ and we have parametrized
 the curve so that $\hcap[\gamma_t] = t$.
 We obtain the SLE in $\check D$ by $\check \gamma(t) =
   N^{-1}F^{-1} \left[\gamma(t) \right]$; here, we have retained the capacity 
   parametrization.  
   }
   \item{We let \[\phi_n^{\LERW}(z) = \left(g_n \circ F \right)(Nz)-U_n, \quad  \phi_{\tau_n}^{\SLE}(z) = \left(g_{\tau_n}^\SLE \circ F \right) (Nz) - W_{\tau_n}\] and let $\G_n$ be the $\sigma$-algebra of the coupling, that is, the $\sigma$-algebra
generated by the discrete LERW domains $A_k, k \le n,$ and the Brownian motion $W_s, 0 \leq s \leq \tau_n$.  We are careful in
our construction to make sure that $\{W_t- W_{\tau_n}: t \geq \tau_n\}$
is independent of $\G_n$.  If $\Im \phi^{\SLE}_{\tau_{n}}(z) \geq h^{1/20}$, then with large probability the two uniformizing maps are close:
\[   \max_{n \le n_{0}}\left|\phi_n^\LERW(z) - \phi_{\tau_n}^\SLE(z) \right| \leq ch^{1/30}.\]}
  
\item{Let $T$ and $T_n$ be the number of steps of $\eta$ and $\eta^n,$ respectively, and let $\check{T} = c_{*}^{-1} N^{-5/4}T$ and $\check{T}_n = c_*^{-1} N^{-5/4}T_n$ be the scaled quantities. Here $c_{*}$ is the constant appearing in \eqref{nov3.1} below. Similarly let $\check{\Theta}$ be the $5/4$-dimensional Minkowski content of  $\check{\gamma}_\infty$ and let $\check{\Theta}_t$ times the $5/4$-dimensional Minkowski content of $\check{\gamma}_t$. }

\item{We consider two discrete time $\mathcal{G}_n$-martingales:
\[
M^{\LERW}_n = \E\left[\check{T} \mid \mathcal{G}_n \right] = \check{T}_n + c_*^{-1} N^{-5/4}\sum_{z \in A_n}\Prob_n\left\{z \in \eta\right\}.
\]
and
\[
M^{\SLE}_n =  \E\left[\check{\Theta} \mid \mathcal{G}_n \right] = \check{\Theta}_{\tau_n} + \int_{\check  D \setminus \check  \gamma_{\tau_n 
   }}
 G_{\check  D \setminus \check  \gamma_{\tau_n 
   }}(z;\check  \gamma(\tau_n),
        \check  b) \, dA(z). 
\]
 Here \[ G_{\check  D \setminus \check  \gamma_{\tau_n 
   }}\left(z;\check  \gamma(\tau_n),
        \check  b \right)  = \tilde{c}\, r_n(z)^{-3/4} \sin^3\left[\arg \phi_{\tau_n}^\SLE(z) \right]\] is the Euclidean
 Green's function for SLE$_2$ in $\check{D}\setminus \check{\gamma}_{\tau_n}$ from $\check\gamma(\tau_n)$ to $b$ and we are writing $r_n(z)$ for the conformal radius of $\check{D}\setminus \check{\gamma}_{\tau_n}$ seen from $z$. The value of the constant $\tilde{c} \in (0,\infty)$ is unknown.}

 \item{
 We form the difference of the two $\G_n$-martingales:
\begin{equation}   \label{nov4.1}
  M_n = M_n^\SLE -
     M_n^{\LERW}.
   \end{equation}
   We can then write $M_n = B_n +  Y_n$,
 where
 \[
 B_n = \Check{\Theta}_{\tau_n} - \check{T}_n,
 \]
  
   and
 \[   Y_n =    \int_{\check  D \setminus \check  \gamma_{\tau_n 
   }}
 G_{\check  D \setminus \check  \gamma_{\tau_n 
   }}(z;\check  \gamma(\tau_n),
        \check  b) \, dA(z)
 - c_*^{-1} \, N^{-5/4} \, \sum_{\zeta \in A_n} \Prob_{n}\{\zeta \in \eta\}.\] Notice that $B_n$
 is a difference of two increasing processes so it is a process
 of bounded variation, and
 $ Y_n$ is a difference of two supermartingales.}
 
 \item{The main result of \cite{BLV} tells us that there are constants $c_* \in (0, \infty)$ and $u > 0$ such that
\begin{equation}  \label{nov3.1}
  \Prob_n\{\zeta \in \eta\} = c_* \, G_{D_n}(\zeta; a_n, b)\left(1 + O\left(N^{- u} \right) \right),
  \end{equation}
at least if the interior point $\zeta$ is not too close to $\partial
A_n$.  So after rescaling and integrating this relation, taking regularity properties into account, we expect $Y_n$ to be uniformly small.
}

 \item{In Section~\ref{sect:main-proof} we will (roughly speaking) use estimates for the coupling and \eqref{nov3.1} to find a $\delta>0$ so that if $\ee_N = \left(\log N \right)^{-\delta}$ then there is a ``large'' stopping time $\tau$ such that \begin{itemize}
 \item{$|Y_{n}| \le \ee_N$ for all $n < \tau$,}
 \item{ $\E \left[Y_\tau^2 \right] \le \ee_N$,}
 \item{ and  $\left|B_n' -B_{n-1}'\right| \le \ee_N$ for all $n \le \tau$, where $B_n'$ is a predictable version of $B_n$.} 
 \end{itemize}} 
 \item{Given this, an argument using the $L^2$-maximum principle shows that $\max_{n \le \tau}|B_{n}'|$ is bounded terms of $\ee_N$, with large probability. In Section~\ref{bvsec} we use this to bound $\max_{n \le \tau}|B_n|$, and this is the estimate we want.}
 
         \end{itemize}

 A substantial complication in this approach is that the Loewner difference
 equation only shows that for suitable $\ee > 0$ the uniformizing LERW and SLE maps $\phi_n^\LERW$ and $\phi_{\tau_n}^\SLE$
 are uniformly close for $z \in D$ with $\Im \left[ \phi_{\tau_n}^\SLE(\zeta) \right] \geq h^\ee$.
We need to also control the contribution of points 
for which $\Im\left[\phi_{\tau_n}^\SLE(\zeta)\right]$ is small. Moreover, the error in the precise version of \eqref{nov3.1} depends on the geometry of the domain seen from $\zeta$. In fact, the curves may \emph{a priori} both create large regions of ``bad'' points, but we will show that the proportion of bad points that are subsequently visited goes to zero and so do not actually contribute.  We will achieve this by showing that, roughly speaking, all
such points satisfy at least one of the following conditions for each $n$, and estimate the contribution differently depending on which. Here we summarize the definition for SLE, see Section~\ref{sect:open} the slightly different definition for LERW and further discussion. We set $\lambda = h^{1/100}$ and describe the two conditions. \begin{enumerate}

\item[I.]{We have $\Im\left[ \phi_{\tau_n}^\SLE(\zeta) \right] \ge \lambda$ and there exists $j \leq n$, such that
  \[S_j(\zeta)  \leq \left(\log N \right)^{-2/5}, \quad \text{ where } S_j(\zeta)
   = \sin \left[\arg  \phi_{\tau_n}^\SLE(\zeta) \right].\] Roughly speaking, this means the the path ``screens'' $\zeta$, e.g., by almost closing a bubble around it, but the distance between the curve and $\zeta$ may still be large. }
   \item[II.]{We have $\Im \left[\phi_{\tau_n}^\SLE(\zeta) \right] < \lambda$ and the distance at time $\tau_n$
from $\zeta$ to the curve is less than $\left(\log N \right)^{-5}$
but the tip of the curve is at least distance
$\left(\log N \right)^{-1}$ from $\zeta$, so that the curve ``got close to $\zeta$ and then away''. 
}
 
\end{enumerate}

A square $\Square_\zeta$ becomes ``closed'' at time $n$ (and stays closed forever) if either of the conditions I or II hold for $\zeta$ at time $n$. A square is ``open'' at a given time if it is not closed. The idea is to do the argument as sketched above but instead redefining the processes $M_n^{\SLE},
M_n^{\LERW},M_n,B_n,Y_n$ to be the corresponding
quantities referring to the amount of 
natural time spent in open squares, that is, time spent before the
square has become closed. For this to work we have to show that it is enough to consider the open squares and this part of the argument is given in Section~\ref{badsec}. The proof of Theorem~\ref{main-thm-rough} is completed in Section~\ref{sect:main-proof}, assuming some statements that are proved in later sections.

 The proof of Theorem~\ref{main-thm-rough}   requires sharp one and two-point
estimates for both SLE and LERW.  For SLE they have been developed
in several recent papers, and the sharp one-point estimate for
LERW is \eqref{nov3.1}.
In Section~\ref{LERWsec} we have collected the needed estimates about LERW.
We have separated them from the main argument because they have independent interest and because this section can be read independently.  
  The two-point estimates for LERW
need both the sharp one-point estimate and an 
  appropriate separation lemma that states that that
two-sided LERW conditioned to reach a ball about the origin have
a good chance of having the endpoints at the first visits from
the two directions ``separated''.  We leave the exact
statements for Section \ref{LERWsec}.  This section, which comprises
almost half of this paper, does not use any facts about SLE.

\subsection{Acknowledgments}
Lawler was supported by National Science Foundation grant DMS-1513036. Viklund was supported by the Knut and Alice Wallenberg Foundation, the Swedish Research Council, the Gustafsson Foundation, and National Science Foundation grant DMS-1308476. We also wish to thank the Isaac Newton Institute for Mathematical Sciences where part of this work was carried out.

\section{Preliminaries}\label{sect:preliminaries}
\subsection{Discrete set-up and loop-erased random walk}\label{sect:set-up} 
 
 Here we will give precise definitions of our discrete quantities.

\begin{itemize}
\item If $A$ is a finite subset of $\Z^2$, we let $\partial_e A$ denote
the edge boundary of $A$, that is, the set of edges  of
$\Z^2$ with exactly one endpoint in $A$.  We will specify
elements of $\partial_e A$ by $a$, the midpoint of the edge.  Note
 that
$a$ specifies the edge uniquely up to the orientation.  We will write\
$a_-,a_+$ for the endpoints of the edge in $\Z^2\setminus A$ and $A$, respectively.
Note   that 
\[  a_-, b_- \in \partial A:= \{z \in \Z^2 \setminus A: \dist(z, A) = 1\}, \]
\[     a_+, b_+ \in  \partial_iA := \{z \in A: \dist(z, \partial A) = 1\}. \]
 We also write the edge as
$e_a = [a_-,a_+], e_b = [b_-,b_+]$ for the edges oriented from the outside to the
inside.

\item Let $\whoknows$
denote the set of triples $(A, a, b)$ where
  $A$  is  a finite, simply connected subset of $\Z^2$  
containing the origin,  and $  a,   b$ are 
elements
of $\partial_e A$ with $a_- \neq b_-$.
  We allow $a_+ = b_+$.

\item  let $\Square= \{x+iy \in \C: |x|,|y| \leq 1/2\}$
be the closed square of side length one centered at the origin and
$\Square_z = z + \Square$. If $(A,a,b) \in \whoknows$, let
 $D_A$ be the corresponding simply connected domain defined
 as the interior of
 \[               \bigcup_{z \in A} \Square_z . \]
  This is a simply connected Jordan domain whose boundary is a subset of the edge set of the dual graph of $ \Z^2$. Note that $a,b \in \partial D_A$. We refer to $D_A$ as a ``union of squares'' domain.

\item{Let
$F = F_{A,a,b}$ denote a conformal map from
$D_A$ onto $\Half$ with $F( a ) = 0, F( b ) = \infty$.  This
map is defined only up to a dilation; later we will fix
a particular choice of $F$.
Note that $F$ and $F^{-1}$
  extend continuously to the boundary of the domain (with the appropriate
  definition of continuity at infinity).
   }
   
\item  For $z \in D_A$, we define 
\[   \theta_A(z;a,b) = \arg F(z) , \;\;\;   \Sine_{A,a,b}(z) = \sin \theta_A(z;a,b),
\]
which are
independent of the choice of $F$, since $F$ is unique up to scaling.  
Also for $z \in \Half$, we write 
\[   \Sine(z) =  \sin[\arg(z)].\]

\item We write $r_A(z)=r_{D_A}(z)$ for the conformal radius of $D_A$
with respect to $z$.  It can be computed from $F$ by
\[        r_A(z) = 2\, \frac{\Im F(z) }{|F'(z)|}, \]
which is independent of the choice of $F$.

\item{A walk $\omega = [ \omega_0, \ldots, \omega_n]$ is a sequence of nearest neighbors in $\Z^2$. The length $|\omega| = n$ is by definition the number of traversed edges.}
\item{If $z,w \in A$, we write 
  $\paths_A(z,w)$ for the
  set of walks $\omega$  starting at $z$, ending
  at $w$, and otherwise staying in $A$. }
  \item{  
  The simple random walk measure  $p$ assigns to each walk measure $p(\omega) = 4^{-|\omega|}$. The total measure
  of $\paths_A(z,w)$ equals $G_A(z,w)$, the simple random
  walk Green's function.}
  \item{
  If $a,b \in \partial_e A$, there is an obvious bijection between
  $\paths_A(a_+,b_+)$ and $\paths_A(a,b)$,  the set of 
  walks starting with edge $e_a$, ending with 
  $e_b^R$ and otherwise staying in $A$.
    (Here and
  throughout this section we write $\omega^R$ for the
  reversal of the path $\omega$, that is, if
  $\omega = [\omega_0,\omega_1,\ldots,\omega_k]$, then
  $\omega^R = [\omega_k,\omega_{k-1},\ldots,\omega_0]$.)
We sometimes write $\omega: a \rightarrow b$ for walks in $\paths_A(a,b)$ with the condition to stay in $A$ implicit.  } 
  
 \item{We
  write $H_{\partial A}(a,b)$ for the total random walk measure of  $\paths_A(a,b)$.
It is easy to see that 
$   H_{\partial A}(a,b) =   G_A(a_+,b_+)/16$, 
 The factor of $1/16 = (1/4)^2$ comes from the $p$-measure of the edges $e_a,e_b$. $H_{\partial A}(a,b)$ is called the boundary Poisson kernel.}
 
\item{A self-avoiding walk (SAW) is a walk visiting each point at most once. We write $\saws_A(z,w) \subset \mathcal{K}_A(z,w)$ for the set of  SAWs from $z$ to $w$ staying in $A$.
  We will write $\omega$ for
general nearest neighbor paths and reserve $\eta$ for SAWs.  We   write $\saws_A(a,b)$ similarly when $a,b$ are boundary edges.  }

 \item{The loop-erasing procedure takes a walk and outputs a SAW, the \emph{loop-erasure} of $\omega$. Suppose a walk $\omega = [\omega_0, \ldots, \omega_n]$ is given.
 \begin{itemize}
 \item{If $\omega$ is self-avoiding, set $\text{LE}[\omega] = \omega$.}
 \item{Otherwise, define $s_0 = \max\{ j \le n: \omega_j = \omega_0\}$ and let $\text{LE}[\omega]_0 =\omega_{s_0}$. }
 \item{For $i \ge 0$, if $s_i < n$, define $s_{i+1} = \max\{ j \le n: \, \omega_j = \omega_{s_i}\}$ and set $\text{LE}[\omega]_{i+1} = \omega_{s_i+1}$.}
 \end{itemize}
Note that if $e_a \oplus \omega
\oplus e_b^R \in \paths_A(a,b)$, then
LE$[e_a \oplus \omega
\oplus e_b^R] =  e_a \oplus \text{LE}[\omega]
\oplus e_b^R $.
 }
 \item  The loop-erasing procedure induces a natural measure on SAWs as follows. We define $\hat P_{A,a,b}$, the ``loop-erased'' measure, on $\saws_A(a,b)$ by
 \[     \hat P_{A,a,b}(\eta) = \sum_{\omega \in \paths_A(a,b): \; \text{LE}(
 \omega) = \eta}
p(\omega). \]
Note that  $\hat P_{A,a,b}[\saws_A(z,w)] = H_{\partial A}
(a,b)$.  Let \[\Prob_{A,a,b} = \frac{\hat P_{A,a,b}}{
H_{\partial A}
(a,b)}\] denote the probability measure obtained by normalization. This is the probability law of loop-erased random walk (LERW) in $A$ from $a$ to $b$.

\end{itemize}

We state the main result from \cite{BLV}.  

\begin{lemma}  There exists $\hat c > 0$ and $u  >0$ such that the following holds. Suppose $(A,a,b) \in \whoknows$ and that $\zeta \in A$ is such that $S_{A,a,b}(\zeta) \ge r_A(\zeta)^{-u}$ ,  then
\begin{equation}
 \label{BLV1}   \Prob_{A,a,b}\{ \zeta \in \eta\} = \hat c \, \, r_A(\zeta)^{-3/4}\Sine_{A,a,b}^3(\zeta)
 \, \left[1 +O\left(r_A(\zeta)^{-u}\Sine_{A,a,b}^{-1}(\zeta) \right)\right].
 \end{equation}
 \end{lemma}
 	
 We do not have an explicit bound on $u$ except $u>0$. We will fix a value of $u$ such that \eqref{BLV1} 
holds for the remainder of the paper.
For our purpose it is more useful to write \eqref{BLV1} in terms of the
Euclidean Green's function of SLE$_2$, see Section~\ref{SLE:defs} for the definition. For now we recall that in this case
\[    G_{D_A}(\zeta;a,b) = \tilde c \,  r_A(\zeta)^{-3/4}
 \, S_{A,a,b}^3(\zeta), \]
for some universal (but unknown)
$\tilde c >0$.  Therefore,  we may  rewrite \eqref{BLV1}
 as
\begin{equation}  \label{BLV2}
   \Prob_{A,a,b}\{ \zeta \in \eta\} =  c_* \,
     G_{D_A} (\zeta;a,b) \,\left[1
       + O\left(r_A(\zeta)^{-u} \right) \;  \Sine_{A,a,b}^{-1}(\zeta)
       \right], 
       \end{equation}
  where $c_* =   \hat c/\tilde c$.
   
  \subsection{SLE and Minkowski content}\label{SLE:defs}
Chordal SLE$_\kappa$ in $\Half$ is defined by first solving the Loewner equation
\[
\partial_t g_t(z) = \frac{2/\kappa}{g_t(z) - B_t}, \quad g_0(z)=z,
\]
with $B_t$ a standard Brownian motion. For each $t \ge 0$, $g_t(z)$ is a conformal map from a simply connected domain $H_t$ onto $\Half$ normalized so that $g_t(z) = z + (2/\kappa)
t/z + O(1/|z|^2)$ as $z \to \infty$. The family $(g_t(z))$ is called the SLE$_\kappa$ Loewner chain. The SLE$_\kappa$ path is the continuous curve defined by \[\gamma(t)=\lim_{y \to 0+} g^{-1}_t(iy+B_t).\] The curve generates the Loewner chain in the sense that $H_t$ is the unbounded component of $\Half \setminus \gamma_t$, where $\gamma_t = \gamma[0,t]$. As $t \to \infty$, this curve connects $0$ with $\infty$ in $\Half$. The compact set which is disconnected from $\infty$ by $\gamma_t$ is called the SLE$_\kappa$ hull (in general, a hull is a compact set such that $\Half \setminus K$ is unbounded and simply connected) and is denoted $K_t$. If $\kappa \le 4$, then $\gamma$ is simple so that $K_t = \gamma_t$.

Given a hull $K$  there is a Riemann map $g : \Half \setminus K \to \Half$ such that $g(z) = z + o(1)$ as $z \to \infty$. We define the half-plane capacity of $K$ by
\[
\hcap[K]  = \lim_{|z| \to \infty} z \left( g(z)-z \right).
\]
If $\gamma$ is parametrized so that $\hcap[K_t] = (2/\kappa)   t  $, we say that $\gamma$ is parametrized by capacity.

Given a simply connected domain $D$ with marked boundary points (prime ends in general) $a,b$, we define SLE$_\kappa$ in $D$ from $a$ to $b$ by conformal invariance. That is, we choose a conformal map $\phi:D \to \Half$ such that $\phi(a) = 0, \, \phi(b) = \infty$ and consider the image of $\gamma$ under $\phi^{-1}$. The map $\phi$ is only unique up to scaling, but allowing for a linear reparametrization the law of $\gamma$ is scale invariant.

The Green's function for SLE$_\kappa$ in $\Half$ is the function defined by \[G_{\Half}(z; 0,\infty)=\lim_{\ee \to 0+} \ee^{d-2} \Prob \left\{ \dist(z,\gamma_\infty) \le \ee \right\}, \quad d=1+\frac{\kappa}{8}.\]
(We suppress the $\kappa$-dependence in writing $G_{D}$.) We have the formula
\[
G_{\Half}(z; 0,\infty)=\tilde{c} \, r_{\Half}(z)^{d-2}\sin^{\beta}\left( \arg z \right), \quad \beta=\frac 8 \kappa-1,
\] 
where $\tilde{c} \in (0, \infty)$ is a $\kappa$-dependent but unknown constant. Note that $r_{\Half}(z) =2 \Im z$. (Replacing Euclidean distance by conformal radius on the left-hand side in the definition results in the same formula with a computable constant.) 
Using conformal covariance we can see that
\[
G_D(z; a,b) = \tilde{c}\, r_D(z)^{d-2} \sin^{\beta}\left( \arg \phi(z) \right),
\]
where $\phi : D \to \Half$ is as in the previous paragraph. 

Besides the capacity parametrization, the \emph{natural parametrization} of SLE is important for this paper. Let us review a few facts about it, see \cite{LR} for proofs and further discussion.
The simplest definition to state is in terms of $d$-dimensional Minkowski content: given $\gamma_t$, we can define
\[
\Theta_t = \Cont_d\left( \gamma_t\right) = \lim_{\ee \to 0+}\ee^{d-2} \text{Area}\left\{ z: \, \dist(z, \gamma_t) \le \ee \right\}.
\]
Then almost surely this limit exists for all $t$ and $t \mapsto \Theta_t$ is Holder continuous. Setting \[s(t) = \inf \left\{s \ge 0 : \Theta_s = t   \right\},\] we may reparametrize $\gamma$ by Minkowski content, that is, consider $t \mapsto \gamma \circ s(t)$, which can be seen to be almost surely Holder continuous.  This is SLE$_\kappa$ in the natural parametrization (or SLE$_\kappa$ parametrized by natural time). When we do not specify the dimension $d$, e.g., by simply writing $\Cont(\cdot)$ we are assuming $\kappa=2$ and $d=5/4$.
Suppose $D$ is a bounded  simply connected domain with (say) analytic boundary. An important property of the Minkowski content is that if $\gamma$ is SLE$_\kappa$ in $D$ from $a$ to $b$, and $V \subset D$, then
\[
\E \left[ \Cont_d(\gamma_\infty \cap V) \mid \gamma_t \right] = \Cont_d(\gamma_t \cap V) + \int_{V \setminus \gamma_t} G_{D\setminus \gamma_t}(z; \gamma(t), b)\, dA(z).
\]
In particular, $\E[\Theta_\infty] = \int_D G_D(z; a,b) \, dA(z) <\infty$ and
\[
\E \left[\Theta_\infty \mid \gamma_t \right] = \Theta_t + \int_{D \setminus \gamma_t} G_{D \setminus \gamma_t}(z; \gamma(t), b) \, dA(z)
\]
is a martingale, and the two terms on the right hand side form its Doob-Meyer decomposition into an increasing process and a supermartingale, respectively.  

In several places, sometimes without explicit reference, we will use the one-point estimate for SLE. We state one version here, see Section~2.2 of \cite{LR} for this and other versions.

\begin{lemma}\label{lem:SLE-one-point}
Suppose $0 < \kappa < 8$. There exist positive constants $c_*, \alpha$ such that the following holds. Let $\gamma$ be SLE$_\kappa$ in $D$ from $a$ to $b$, where $D$ is a simply connected domain with distinct boundary points (prime ends) $a,b$. Then for all $z \in D$ with $\dist(z,\partial D) \ge 2\ee$,
\[
\Prob\left\{\gamma \cap \ball(z,\ee) \neq \emptyset \right\} = c_*\ee^{2-d}G_D(z;a,b)\left[1+O(\ee^\alpha)\right],
\]
where $G_D(z;a,b)$ is the Green's function for SLE$_\kappa$ from $a$ to $b$ in $D$.
\end{lemma}

  \subsection{Complete statement of main result}\label{sect:main_complete}
  
  We will now give a complete statement of our main result. In order to do so, we will have to scale
 the lattice path.

\begin{itemize}
 
 \item  Given $\eta \in \saws_A(a,b)$, of the form
 \[  \eta = [\eta_0 = a_-, \eta_1 = a_+,
 \ldots ,\eta_{n} = b_+, \eta_{n+1} = b_-] , \]
 we write $\eta(t)$ for the curve obtained by going from $a$ to $b$
 along $\eta$ at speed one.  More precisely, $\eta(t), 0 \leq t \leq n$, 
 is defined by
 $\eta(0) = a; \eta(n) = b$; 
 \[ \eta\left(j- \frac 12\right) = \eta_j, \;\;\;\;j=1,\ldots, n;\]
and $\eta(t)$ is defined for other $t$ by
  linear interpolation. 
 
 \item If $\eta(t), 0 \leq t \leq n$ is a curve as above
 and $N > 0$, we let $\eta^N(t)$ denote the scaled map
 \[     \eta^N(t) = N^{-1} \, \eta(t \,c_* \,  N^{5/4}) ,
 \;\;\;\;\; 0 \leq t \leq   \frac{n}{c_* \, N^{5/4}}.\]
 Here $c_*$ is the constant from \eqref{BLV2}.
 
 \item  We write $\Prob_{A,a,b}^N$ for the probability measure
 obtained from $\Prob_{A,a,b}$ by considering the curves scaled
 as above.

\item  If $\eta = [\eta_0,\ldots,\eta_k] \in \saws_A(a,b)$,
let $\eta^j = [\eta_0,\ldots,\eta_j]$, $A_j
 = A \setminus \eta^j$ and $a_j = [\eta_{j} + \eta_{j+1}]/2$
 so that $a_j \in \partial_e A_j$.
  The tuples $(A_j, a_j, b), \, j=0,1,\ldots$ form a sequence of decreasing discrete domains with two marked boundary edges.  We write $D_j = D_{A_j}$.
  Note that $D_j$ is obtained from $D_A$ by removing the
  $j$ squares associated to first $j$ steps of $\eta$ plus
  any squares that have become disconnected from $0$.

\item{We will assume that we have a bounded analytic
simply connected domain $D$ containing the origin with
analytic boundary and two distinct boundary points
$a',b'$.   We will consider lattice approximations of $D$.  
The lattice scaling will be $N^{-1}$.  We will define some
scaled quantities, but the dependence on $N$ will be implicit.}
\item{The assumption that $D$ is analytic is of course not necessary -- we will make it for convenience, but remark that by an approximation argument our main result can be extended to more general domains, assuming local analyticity at $a',b'$.} 
 \item{ If $N > 0$, let $A  = A_{N,D}$ be the connected
 component containing the origin of the set of $\zeta \in \Z^2$
 with $\Square_\zeta \subset ND$. Let $D_A$ be the corresponding
 domain obtained by taking the interior of the union of the $\Square_\zeta$. Let $\check D_A = N^{-1} \, D_A$.  If $a \in \partial_e
 A$, we write $\check a $ for the midpoint of the edge $a/N$. We sometimes identify an edge with its midpoint.}

\end{itemize}

\begin{itemize}

\item We consider the metric \eqref{metric} on continuous curves and write $\wp_\rho$ for the corresponding Prokhorov metric
on probability measures on curves.

\item  If $D$ is a domain as above and $a,b$
are distinct boundary points, then $\mu_D(a,b)$ denotes
the probability measure given by SLE$_2$ with the natural
parametrization.  (In other papers of the first author, 
the notation   measure $\mu_D(a,b)$
refers to SLE with total mass of the partition function
and the probability measure and the corresponding probability
measure 
  denoted by $\mu_D^\#(a,b)$.  However, since we only
  need to use the probability measure in this paper, we
  choose the simpler notation.)    

\end{itemize}

\begin{thm}\label{thm:main_complete}  Let $D$ be a bounded analytic domain containing
the origin with distinct
boundary points $a',b'$.  For each $N$, let $A_N = A_{N,D}$
and let $a_N,b_N \in \partial_e A_N$ with 
\[   \check a :=  a_N/N \stackrel{N \to \infty}{\longrightarrow} a', \;\;\;\;
   \check b:=  b_N/N \stackrel{N \to \infty}{\longrightarrow} b'.\]
   
 Then,
 \[   \lim_{N \to \infty} \Prob^N_{A_N,a_N,b_N} =
    \mu_D(a',b'), \]
 where the convergence is with respect to the Prokhorov
 metric as above.
 \end{thm}

 We start by making some reductions.  It is not
 difficult (see Corollary \ref{mar5.cor3}) to show that
 \[  \lim_{N \rightarrow \infty}
     \wp\left[\mu_{\check D}(\check a,\check b) ,
     \mu_D(a',b')\right] = 0.\]
Hence, it suffices to show that
$    \lim_{N \to \infty} \wp\left[\nu_N ,\tilde \nu_N
     \right]=0,$
 where \[\nu_N^{\LERW} = \Prob^N_{A_N,a_N,b_N}, \qquad 
  \nu_N^{\SLE} =  \mu_{\check D}(\check a, \check b).\]

 In order to compare
 $\nu_N^{\LERW},  \nu_N^{\SLE}$  we consider the paths parametrized
 by half-plane capacity.  This capacity is defined by
 first taking $F:D_A \rightarrow \Half$ with
 $F_N(a_N) = 0, F_N(b_N) = \infty$ and measuring
 capacities of the images under $F$.  The map $F$
 is unique up to a final dilation.
 
 For every $k < \infty$, we can consider the measures
 $\nu_N^{\LERW}, \nu_N^{\SLE}$ on paths  stopped when the capacity
 reaches $k$.  Since the total capacity of the curves 
 is infinite, this truncation is well defined and does not
 give the entire curve.  In order to get convergence
 in the Prokhorov metric we need two facts.  The first:
 \begin{itemize}
 \item  The curves parametrized by capacity are very close
 in supremum norm except for small probability.
 \end{itemize}
 We prove this in the separate paper \cite{LV_lerw_chordal_note}, but we will give the needed statements below.
 This result has been proved previously for convergence
 in capacity parametrization using a slightly different coupling, see \cite{LSW04, zhan}. The second is the one that we focus on:
 \begin{itemize}
 \item  If we reparametrize by length (using normalized
 number of steps for the LERW and Minkowski content for
 the SLE), the reparametrizations are very close in
 supremum norm except for small probability.
 \end{itemize}
 Let $\nu_{N,k}^{\LERW},  \nu_{N,k}^{\SLE}$ denote the corresponding
 measures on curves 
 parametrized by length but {\em truncated when their
 capacity reaches $k$}.  We will show that
 $\wp \left[\nu_{N,k}^{\LERW},  \nu_{N,k}^{\SLE} \right] $ is small.
 We also need to show for LERW and for SLE that as
 $k \rightarrow \infty$, both
 $\wp \left[\nu_{N,k}^{\LERW},\nu_N^{\LERW} \right]$ and $\wp\left[ \nu_{N,k}^{\SLE},
 \tilde \nu_k^{\SLE} \right]$ are bounded by $\epsilon _k$ for some $\epsilon_k
 \rightarrow 0$ (independent of $N$).
 This is discussed in Section \ref{metricsec}.
 
 Finally, rather than take a particular $F$ and showing
 the estimates for paths truncated at capacity $k$, we
 will start with any $F$ and truncate at capacity $1$.
 Note that paths truncated at capacity $k$ for a given
 $F$ are the same as those truncated at capacity $1$
 for  the map $z \mapsto k^{-1/2} \, F(z).$
 
 This is the main theorem of this paper and 
 precise statements can be 
found in   \eqref{feb15.1}
and \eqref{feb15.2}.

 We prove a theorem on about the chordal version of LERW connecting two boundary
 points but one can derive from this a corresponding result
 about LERW from a boundary point to an interior point.
The details can be found in \cite{LV_lerw_radial_note}, but
 we state the result here.

 \begin{thm}\label{main-thm-rough-radial}
There is a universal constant $\check{c}$ and an explicit sequence $\ee_N \to 0+$ as $N \to \infty$ such that the following holds.  For each $N$, let $\eta(t), t\in [0, T_\eta],$ be LERW in $A_N$ from $a_N$ to the origin viewed as a continuous curve parametrized so that each edge is traversed in time $\check{c}N^{-5/4}$. Let $\gamma(t), t \in [0, T_\gamma],$ be radial SLE$_2$ in $D$ from $a$ to $0$ parametrized by $5/4$-dimensional Minkowski content. There is a coupling of $\eta$ and $\gamma$ such that
\[
\Prob\left\{\rho \left(\eta,\gamma \right) > \ee_N \right\} < \ee_N.
\]
In particular, $\eta$ converges to $\gamma$ weakly with respect to the metric $\rho$.
\end{thm}

There are two approaches to this last theorem.  One would be to
redo the work in this paper in the radial setting.  In fact, the  
proof of convergence of LERW with respect to capacity
parametrization in  \cite{LSW04} was for the radial case,
and a version for the chordal case was first done in \cite{zhan} by following
the same outline.   However, for our result this
would take a fair amount of work; in particular, the radial analogue
of the estimate \eqref{BLV2} would need to be proved.  Fortunately,
we now know that one can go between the chordal and radial
results using Radon-Nikodym derivatives and this is the appraoch
we use in \cite{LV_lerw_radial_note}.

\section{Deterministic estimates and coupling}\label{sect:deterministic}

The first step to the proof is to construct the coupling between SLE$_2$
and LERW.  Our argument is similar to that in \cite{LSW04}  although there
are some differences.  First, we use a difference equation form of the Loewner theory.
    This is useful because our
  domains $D_n=D_{A_n}$ are derived from $D_A$ by cutting out
squares rather than a curve.
We could   work with slit domains and translate results, but we feel the difference equation approach produces cleaner arguments in the present setting. 
The other change is that we use   the LERW Green's
function rather than the discrete Poisson kernel as our observable.

 It is important that these results are
redone in our context, but because the proofs are similar to those
in previous coupling,    we will only state the important results here.  Complete proofs and additional discussion
  can be found in
 \cite{LV_lerw_chordal_note}.

\subsection{Loewner difference equation}
Suppose
$\gamma:(0,\infty) \rightarrow \Half$ is a simple
curve with $\gamma(0+) =0$ parametrized so that
$\hcap[\gamma_t] = t$, where $\gamma_t = \gamma[0,t]$. 
Let $g_t:\Half \setminus \gamma_t
\rightarrow \Half$ be the conformal transformation with
$g_t(z) = z +o(1)$ as $z \rightarrow \infty$. Then we have the chordal Loewner differential equation, 
\[  \partial_t g_t(z) = \frac{1 }{g_t(z) - U_t} , \quad g_0(z)=z, \]
where $U_t = g_t(\gamma(t))$.   The proof of this correspondence, at
least as given in \cite{Lbook}, starts by proving a
``difference estimate'' to show that for small
$t$, 
\begin{equation}  \label{oct9.1}
  g_t(z) - z =  \frac{t}{z}
  + O\left(\frac{tr}{|z|^2} \right), 
  \end{equation}
  where $r$ denotes the radius of $\gamma_t$.  This
estimate does not require $\gamma_t$ to be the image of a curve
and in fact holds with an error term uniformly
bounded over  all hulls (see below) of half-plane capacity
$t$ and radius $r$.

We will say that 
 $K \subset \Half$ is a {\em (compact $\Half$-)hull}, 
 if $K$ is bounded and 
  $D_K := \Half \setminus
K$ is a simply connected domain. Define 
\[  r_{K} = \rad(K) = \max\{|z|: z \in K\},\;\;\;\;
  h_K = \hcap(K),\]
and recall that $h_K \leq r_K^2$. If $r_K$ is small, $K$ is located near $0$.
Let $g_{K}$ be the unique conformal transformation
of $D_K$ onto $\Half$ whose expansion at infinity is
\[   g_{K}(z) = z + \frac{h_K}{z} + O(|z|^{-2}).\]
 Suppose now we have a sequence of hulls of small capacity
 $K_1,K_2,\ldots$ and ``locations'' $U_1,U_2,\ldots  \in \R$, so that, roughly speaking $K_j + U_j$ is near $U_j$.
 Let \[r_j = r_{K_j}, \quad  h_j = h_{K_j}= h_{K_j + U_j}, \quad g^j
  = g_{K_j + U_j}\] and let
  \[  g_j =  {g^j \circ \cdots
   \circ g^1}.\]
Since the right-hand side of \eqref{oct9.1}
  depends only on $r$ and not on the exact shape
  of $K$, it follows that if we have two sequences
  for which the capacity increments and ``driving terms'', $h_j$ and $U_j$, are close, then  the functions $g_n$
 are   close.  We give a precise formulation of this
  in the next two proposition.   
 \begin{prop}\label{prop:loewner-comparison}  There exists $1 < c < \infty$ such
 that the following holds.  Suppose $(K_1,U_1),
 $ $ (K_2,U_2)\ldots$ and $(\tilde K_1,\tilde U_1),
 (\tilde K_2,\tilde U_2),$ $\ldots$ are two sequences as above
 with corresponding $r_j, h_j, g^j, g_j$ and
 $\tilde r_j,  \tilde h_j, \tilde g^j, \tilde g_j$.
 Let \[0 <   h <
 r^2 < \epsilon^2 < \delta^8 < 1/c,\] 
 and $n  \leq 1/h$ and suppose that
 for all $j=1,\ldots,n$,
 \[      |h_j -  h| \leq  hr/\delta  , \;\;\;\;
 |\tilde h_j - h| \leq hr/ \delta , \]
  \[      r_j, \tilde r_j  \leq r  , \]
   \[     |U_j -\tilde U_j| \leq \epsilon. \]
Suppose $z = x+iy \in \Half$ and let
$z_n = x_n +iy_n = g_n(z), \tilde z_n =
\tilde x_n + i \tilde y_n = \tilde g_n(z).$
Then, if $y_n,\tilde y_n \geq \delta$,
\begin{equation}  \label{halloween.1}
  |g_n(z) - \tilde g_n(z)| \leq c \, (\epsilon/\delta) \,
     (y \wedge 1).
     \end{equation}
\end{prop}
\begin{prop}
\label{prop:deriv-lb}  There exists $1 < c < \infty$ such
 that the following holds.  Suppose $(K_1,U_1),
 $ $ (K_2,U_2)\ldots$ is a sequence as above
 with corresponding $r_j, h_j, g^j, g_j$.
 Let \[0 <   h <
 r^2 <  \delta^8 < 1/c,\] 
 and $n  \leq 1/h$ and suppose that
 for all $j=1,\ldots,n$,
 \[      |h_j -  h| \leq  hr/\delta, \quad r_j \le r.  \]
Suppose $z = x+iy \in \Half$ and let
$z_n = x_n +iy_n = g_n(z)$.
Then if $y_n \geq \delta$,
\begin{equation}  \label{halloween.2}
  |g_n'(z)| = \exp\left\{-\sum_{j=0}^{n-1} \Re \frac{h}{(z_{j} - U_{j})^2}\right\} \left(1+O(\delta) \right).  
     \end{equation}
     In particular, there is a constant $c$ such that if
\begin{equation}\label{jan26.1}
 \nu=\min_{0\le j \le n} \left\{\sin\left[ \arg\left(g_j(z) - U_j \right) \right] \right\},
 \end{equation}
 then,
 \begin{equation}\label{jan26.2}
 |g'_n(z)| \ge c \left(\frac{y_n}{y}\right)^{1-2\nu^2}.
 \end{equation}
\end{prop}  
\subsection{Coupling}\label{sect:coupling}
We will consider 4-tuples $(A,a,b,F)$ where $(A,a,b) 
\in \whoknows$.  and  $F:D_A \rightarrow \Half$ is a conformal transformation
with $F(a) = 0, F(b) = \infty$.  As we have noted before, there is a one-parameter
family of such transformations $F$, so we will fix one of them. We define \[N = N(A,a,b,F) = |(F^{-1})'(10i)|\] and note that $N$ is half the conformal radius of $D_A$ seen from $F^{-1}(10i)$.   All of our results will hold only for 
  $N$ sufficiently large, and we will not always be explicit about this.

We fix a mesoscopic scale $h$, defined by
\[
h = N^{-2u/3},
\]
where $u$ is the exponent from \eqref{BLV1}. This is a somewhat arbitrary
choice, but we will use that $N^{-u} = O(h^{6/5})$.

Let $(A_0,a_0,b) = (A,a,b), D_0 = D_A, F_0^\LERW = F$. We will define a sequence $(A_n,a_n,b)$ with corresponding
simply connected domains $D_n$ and functions $F_n^\LERW$ recursively
   by saying that the conditional distribution of $(A_n,a_n,b)$
   given $(A_{n-1},a_{n-1},b)$ is that of  the LERW
 probability measure $\Prob_{n-1}:= \Prob_{A_{n-1},a_{n-1},b}$
 where the walk (taking microscopic lattice steps) is stopped at the first time $m = m_n$ such that
 \[   \diam[K^m] \geq h^{2/5} \;\;\; \mbox{ or } \;\;\;
   \hcap[K^m] \geq h , \]
   where
\begin{equation}  \label{feb16.2}
 K^j = F_{{n-1}}(D_{A_{n-1}} \setminus D_{A_{n-1} \setminus \eta^j}) \subset \Half
\end{equation}
and $\eta$ is LERW in $A_{n-1}$ from $a_{n-1}$ to $b$.
  We set
$D_n = D_{A_n}$ and 
\begin{equation}  \label{feb16.1}
 F_n^\LERW =  g_n \circ F_0  - U_n, \;\;\;\; U_n:= 
  g_n \circ F_0 (a_n) .  
 \end{equation}
 where $  g_n: F_0(D_0 \setminus D_n)
  \rightarrow \Half$ is the conformal transformation
  normalized so that $  g_n(\infty) = \infty$,
$ g_n'(\infty) = 1$.   Note that the transformation $F_n^\LERW$ is
translated so that $F_n^\LERW(a_n) = 0, F_n^\LERW(b) = \infty$.
Let $\xi_n = U_n - U_{n-1}$.

Let us be more precise. We write $D_0 = D_A, D_j = D_{A_j}$. Set $m_0=0, m_1=m$, where
\[
 m=\min\left\{j \ge 0: \,    \hcap\left[K_{j} \right] 
 \geq h \text{ or }   \diam\left[K_j \right]  \geq h^{2/5} \right\},\]
 where $K_j= F(D \setminus D_{A_j})$,
  and for $n=0,1, \dots$, and $j =0, 1, \ldots$,
\[
K_j = F(D \setminus D_j), \quad K^n_j = F_{m_n}(D_{m_n} \setminus D_{m_n + j}),
\]
and
\[
\Delta_n = \min\left\{ j \ge 0: \, \hcap[K_j^n] \ge h \, \text{ or } \,  \diam[K_j^n] \ge h^{2/5} \right\},
\]
\[
m_{n+1} = m_n + \Delta_n.  
\]
Write
\[
K^n = K^n_{\Delta_n}.
\]
Then 
   \[t_{m_{n+1}} =t_{m_n} + \hcap\left[K^n \right].\] 
and we set
   \[
  r_{m_{n+1}} = \diam\left[K^n \right].
   \]
 Let $g^{n+1}:\Half \setminus K^n \rightarrow \Half$
 be the conformal transformation with $g^{n+1}(z)
  - z = o(1)$ and set $F_{m_{n+1}} = g^{n+1} \circ F_{m_n}$ and  \[     g_{n+1} = g^{n+1} \circ g^{n} \circ \cdots
    \circ g^1 . \]
  We also define the ``driving process increment'',
  \[   \xi_{n+1} = g^{n+1} \circ F_{m_n}(a_{m_n})-\xi_{n}, \]
 giving a ``driving process'' 
   \begin{equation} \label{jan12.2}  U_{n+1} = \xi_1 + \cdots + \xi_{n+1}.\end{equation}
  Write also
  \[
  H_n = F(D_{m_n}) \subset \mathbb{H}.
  \]
 We continue this process until $n_0$,  the first time $n$ such
 that 
 \[    \diam \left[F(K_{m_n})
  \right] \geq 2 \;\;\;\;\;\; \mbox{ or }
  \;\;\;\;\;\;   \hcap \left[ F(K_{m_n})
  \right]\geq 2 . \]
 Note that $n_0 -1 \leq 1/h$ and that for $n < n_0$,
\[ \hcap \left[F(D_0 \setminus D_{n})
  \right] \leq 2, \;\;\;\;|X_n| \leq 2, \]
  \[   |(F_{m_n}^{-1})'(10i)| \asymp |(F^{-1})'(10i)|
  = N,\]
Using the Beurling estimate, we can see that
for $n < n_0$, the mesoscopic  increments satisfy
\[        t_{m_n} -t_{m_{n-1}}  \leq h +  O(N^{-1}),\;\;\;\;
       r_{m_n}-r_{m_{n-1}}  	\leq h^{2/5} + O(N^{-1/2}).\]
Let $\F_n$ denote the $\sigma$-algebra
generated by $(A_0,a_0,b),\cdots,
(A_{m_n},a_{m_n},b)$.

\begin{lemma}\label{lem:coupling-pt2}
There is a coupling of the LERW $\eta$ and a standard Brownian motion $(W_t, \mathcal{\tilde{F}}_t)$ and 
a sequence of strictly increasing stopping times $\{\tau_n\}$ for $(W_t, \mathcal{\tilde{F}}_t)$ 
such that the following holds.  Let $n_*$ be the minimum of $n_0 + 1$ and the
smallest $n'$ such that one of the following  does not  hold:
 \[
 \max_{n  \leq n'}|\tau_n-nh| \leq  h^{1/5}  ,
\] 
\[
 \max_{n \le n'}|W_{\tau_n}-U_n| \leq  h^{1/10} ,
\] 
  \[ \max_{n \le n'} \max_{\tau_{n-1} \leq t \leq \tau_n}
      |W_{t} - W_{\tau_{n-1}}|  \leq  h^{2/5} ,
\] 
   \[  \max_{t \leq \tau_{n'}}\;\; 
          \max_{t - h^{1/5} \leq s \leq t}\;\;
              |W_t - W_s| \leq  h^{1/12} .
\] 
Then $\Prob\{n_* \leq n_0\} = O(h^{1/10})$. 
Moreover, if $\G_{n}$ denotes the $\sigma$-algebra generated by $\hat \F_{n}$
and $\F_{\tau_{n}}$, then $t \mapsto W_{t + \tau_n}-
W_{\tau_n}$ is independent of $\G_{n}$ and
the distribution of the LERW given
$\G_n$ is the same as the distribution given
$\hat \F_n$.

\end{lemma}
Given the Brownian motion $W_t$, there is a corresponding SLE$_2$ Loewner chain $(g_t^\SLE)$ obtained by solving the Loewner differential equation with $W_t$ as driving term. The Loewner chain is generated by an SLE$_2$ path in
$\Half$ that we denote by $\gamma(t)$.  Let $\hat \gamma(t)
= \hat F^{-1}[\gamma(t)]$ which is an SLE$_2$ path
from $\hat a$ to $\hat b$  in $\hat D_A$
parametrized by capacity in $\Half$ (this parametrization depends
on $F$ but we have fixed $F$.)   
We write
\[
F^{\SLE}_n(z) = (g_{\tau_n}^\SLE \circ F)(z)-W_{\tau_n}
\]
and
\[
F^{\LERW}_n(z) = (g_n \circ F)(z) - U_n.
\]
Combining the coupling with the deterministic
estimates we get the following.

   \begin{lemma}\label{lem:coupling-of-maps}
If $n < n_*$,  we have  uniformly in $\zeta \in A$ such that $\Im F_n^\SLE(\zeta) \ge h^{1/80}$, 
  \[
  \left| F_n^{\LERW}(\zeta) -  F_n^{\SLE}(\zeta)\right|  \le c  h^{1/15}.
  \]
  \end{lemma}

    \section{Core argument}

 \subsection{Setup}
 
 At this point we will quickly review our setup.
 
 \begin{itemize}
 
 \item{We start with an analytic domain $D$ containing
 the origin and with distinct boundary points $a',b'$. For each integer $N>0$ we define $( A ,a ,b )$ (and the associated union of squares domain $D_{A }$) as
 the discrete approximations of $(ND,Na',Nb')$ with a choice of conformal transformation
     $ F: D_{A} \rightarrow \Half $
 with  $  F(a) = 0,   F(b ) = \infty,   \Im[F(0)] = 1 + o(1)$. All constants, implicit
 or explicit, may depend on $D,a',b', F$, but are otherwise
 universal.}

\item  Let \[h=h_N=N^{-2u/3}, \quad n_0 = n_{0,N} = \lfloor h^{-1}\rfloor,\]
   be the mesoscopic scale, where $u$ is the exponent
 in \eqref{BLV1}.

\item  Let $T_n = c_*^{-1} \left(m_1 + \cdots + m_n \right)$ denote
the number of steps of the LERW taken after $n$
mesoscopic steps, see Section~\ref{sect:deterministic}, rescaled by the constant in \eqref{nov3.1}.  

 \item  The scaled LERW $\eta(t), 0 \leq t \leq 1$,
  in $\check D$
parametrized by capacity is given by
\[        \check \eta(nh) = N^{-1} \eta_{c_{*}T_n}.\]
We choose the parametrization
to linearly interpolate between times $(n-1)h$ and $nh$.

 \item  As in Section~\ref{sect:coupling} we couple the LERW with an SLE$_2$ path
 from $a'$ to $b'$ 
 in $ D_A$, denoted $\hat{\gamma}$, parameterized  so
 that
 \[   \hcap \left( F \circ \hat{\gamma}[0,t]\right) = t .\]
 We let $ \check   \gamma(t) = N^{-1} \, \hat{\gamma}(t)$ be the
 corresponding SLE$_2$ in $\check D = \check{D}_A$.

 \item  Let
 \[    \check   T_{nh} = N^{-5/4} \, T_n  \]
be the rescaled number of steps in the walk with $\check{T}_{t}$ defined by linear interpolation between times $(n-1)h$ and $nh$.  Let
\[      \check \Theta_t = \Cont \left(\check  \gamma[0,t] \right)  \]
be the $5/4$-dimensional Minkowski content of $\check \gamma[0,t]$.

 \end{itemize}
 
 We can now state the main result. 
 
 \begin{thm}\label{thm:main-thm-core-sec}   There exists $\epsilon_N \rightarrow 0$
  such that 
  except for an event of probability
 at most $ \epsilon_N$, 
\begin{equation}  \label{feb15.1}
   \max_{0 \leq t \leq 1}  \left|\check \eta(t) - \check \gamma(t) \right|
   \leq   \epsilon_N , 
   \end{equation}
\begin{equation}  \label{feb15.2}
   \max_{0 \leq t \leq 1}   \left|\check T_t - \check{\Theta}_t \right|
  \leq   \epsilon_N. 
  \end{equation}
\end{thm}

 We will prove the theorem with
 \[\epsilon_N = c\, \left(\log N \right)^{-1/60}\] where the constant $c$ depends
 on $D,a',b',F$. 
 The estimate \eqref{feb15.1} with an unspecified sequence $\ee_N$ and for a slightly different coupling was done in \cite{LSW04}.
The convergence rate in the coupling of \cite{LSW04} was estimated in \cite{BJK, JV}. In \cite{LV_lerw_chordal_note} we obtain a  polynomial convergence rate for \eqref{feb15.1} in the case of the coupling used in this paper. In this paper, we will only worry about proving the second estimate
\eqref{feb15.2}.
  We encourage the reader to recall the general idea of the proof as outlined in Section~\ref{sect:overview}.

For the remainder, we fix $N$ and a coupling as above.
Where we use $n$, we will assume that $n  <  n_*$  where
$n_*$ is as in 
  Lemma  \ref{lem:coupling-pt2}.

 \subsection{Maximal estimate}
 
We will need to know that neither the Minkowski content nor the scaled
 number of steps visited by the loop-erased random walk can get large
 on a small region.  To make this precise,   let $\ball(z,\epsilon)$
 denote the closed disk of radius $\epsilon$ about $z$ and
 define
\begin{eqnarray*}
     \maxsle& = &   N^{-5/4} \, \sup_{z \in \C} \Cont\left[
     \hat \gamma \cap \ball(z,
     N/\log N)\right] \\
    &  =  & \sup_{z \in \C} \Cont\left[ {\check \gamma} \cap \ball(z,1/\log N)\right].
    \end{eqnarray*}
 The LERW analogue is
 \[  \maxlerw = N^{-5/4} \, \sup_{z \in \C}
      \sum_{\zeta \in A \cap \ball(z,N/\log N)}
               1\{\zeta \in \eta\}.\]
               
\begin{prop}  \label{maxprop}
There exists $c < \infty$ such that
\[      \E\left[\maxsletwo\right] + \E\left[\maxlerwtwo\right]
   \leq c \, \left(\log N \right)^{-5/4}.\]
 \end{prop}
 
  \begin{proof} 
The estimate for SLE was done in 
\cite{LR} where a similar maximal estimate is a key step
for establishing H\"older continuity of the Minkowski content
with respect to capacity parametrization.  In
Proposition \ref{prop:maximal-estimate}  we
use a similar 
  argument   for LERW  after establishing
 a bound
on the second moment for the number of steps of the walk.

\end{proof}

 \subsection{Open and closed squares: definitions}\label{sect:open}
 
\begin{figure}[t]
\centering
  \def\svgwidth{0.95\columnwidth}
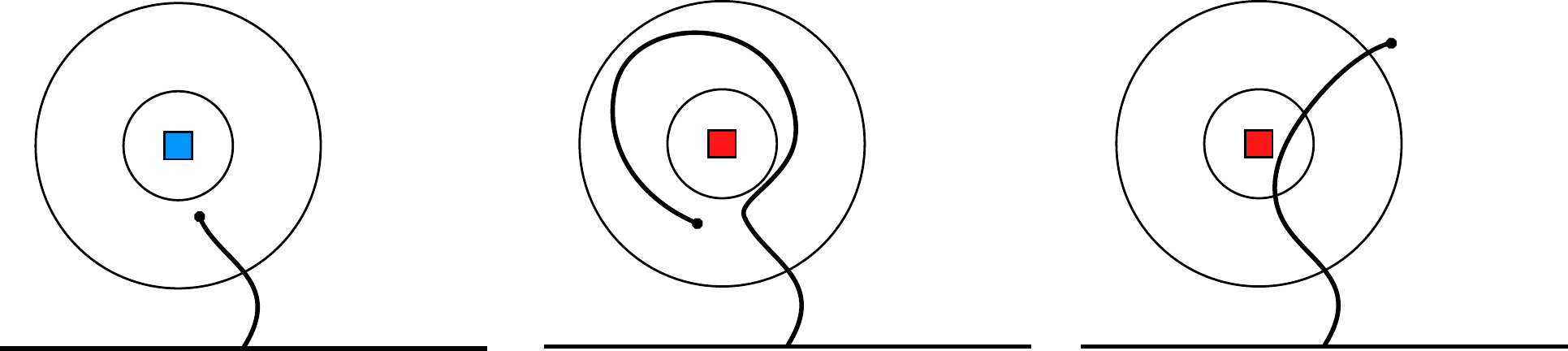
\caption{Left: The square is open. Middle: Closed square of Type I. The distance from the real line after uniformizing is still larger than $\lambda$, but the sine of the argument along the path has gotten small. Right: Closed square of Type II. The curve got close to the square and then escaped while the distance from the real line after uniformizing dropped below $\lambda$. } \label{fig:badpoints}
\end{figure}
    Throughout this section we set
\[           \lambda = \lambda_N = h^{ 1/100}.\]
In the following definition recall that $ \Im[ F_n^\SLE(\zeta)] $
is a decreasing function of $n$. See Figures~\ref{fig:badpoints} and \ref{fig:badpoints2}.
\begin{definition}   \label{sleopendef}
We will say that the square $\Square_\zeta, \zeta \in A,$ is {\bf closed for 
  SLE} at step $n$   if either of $I$ or $II$ holds at
$t= \tau_n$, where:
\begin{itemize}
\item[\bf{I}:] We have \[ 
     \lambda \leq \Im[ F_{ n}^\SLE(\zeta)] \leq 10 ,\]
    and 
       \[   \Sine_{\Half \setminus \gamma_{\tau_n}}(F(\zeta);\gamma({\tau_n}),\infty)
  \leq \frac{1}{\left(\log N \right)^{2/5}}.\]
  \end{itemize}
\begin{itemize}
\item[\bf{II}:] We have
\[   \Im[F_n^\SLE(\zeta)]  < \lambda , \]
 \[\dist(    \zeta, \partial (   D_A \setminus \hat {\gamma}_t)) \leq \frac{N}{\left(\log N \right)^5}, \]
 and
 \[  
 |   \zeta - \hat {\gamma}(t)| \geq \frac{N }{\log N}.\]
 \end{itemize}
 \end{definition}
 
\begin{definition}\label{lerwopendef}

We will say that the square $\Square_\zeta, \, \zeta \in A,$ is {\bf closed for 
 LERW} at step $n$   if either
of $I$ or $II$ holds at
$t= \tau_n$, where:
\begin{itemize}
\item[\bf{I}:] We have
 \[ 
      \lambda \leq \Im[ F_{{  n}}^{\SLE}(\zeta)] \leq 10 ,\]   
    and $\Square_\zeta$ is closed for SLE.
  \end{itemize}
\begin{itemize}
\item[\bf{II}:] We have 
\[   \Im[ F_{ n}^{\SLE}(\zeta)]  < \lambda , \]
 \[\dist( \zeta,\partial    {D}_n) \leq \frac{N}{\left(\log N \right)^5}, \]
 and \[   |   \zeta -    a_k|
  \geq \frac{N}{\log N}.\]
 \end{itemize}
 \end{definition}
\begin{figure}[t]
\centering
  \def\svgwidth{0.65\columnwidth}
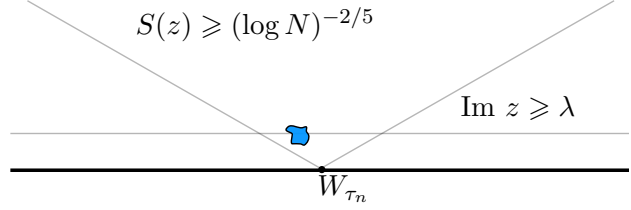
\caption{The image of an open square near the boundary. If $z=F(\zeta)$ and $S(z)=\Sine_{\Half \setminus \gamma_{\tau_n}}(z;\gamma({\tau_n}),\infty) \ge \left(\log N \right)^{-2/5}$ we have derivative estimates. The conclusion is that the distance to the boundary is $o(N/\left(\log N \right)^5)$ and so if the square is open it is at distance $O(\left(\log N \right)/N)$ of the tip.} \label{fig:badpoints2}
\end{figure}
In both cases, once a square is closed it stays closed forever.
A square is said to be open for SLE (for LERW) at step $n$ if it is not closed for SLE (for LERW) at step $n$.

We will write $\opensle_{n,\zeta}$ and $\openlerw_{n,\zeta}$ for the indicator functions
of the event that $\Square_\zeta$ is open for SLE and LERW, respectively.
Then we have the following properties:
\begin{itemize}
\item  If $\Im[F_n^\SLE(\zeta) ] \geq \lambda$, then 
$\opensle_{n,\zeta} = \openlerw_{n,\zeta}$.

\item If $n \leq m$,  and $\opensle_{n,\zeta} = 0$, then
$\opensle_{m,\zeta}  = 0. $  If $\openlerw_{n,\zeta} 
= 0$ then $\openlerw_{m,\zeta}  = 0$. 

\end{itemize}
 
The next observation is that a square $\Square_\zeta$
 cannot still be open  if it both 
 far from the tip  and the conformal map has a 
small imaginary part.  
The essential idea is that in order for the imaginary part to be small
but for the curve to not get close to the point, there must be a time
when the sine of the angle was small.

\begin{prop} \label{prop.feb13}
$\;$
\begin{itemize}
\item Suppose $\Square_ \zeta$ is open for SLE at step $n$.  Then either $\Im[F_{n}^\SLE(\zeta)]
\geq  \lambda $ or $|\zeta-\hat\gamma(t)| \leq N/\left(\log N \right)$.
\item  Suppose $\Square_ \zeta$ is open for LERW at step $n$.  Then either $\Im[F_{n}^\SLE(\zeta)]
\geq \lambda $ or $|\zeta - a_n|  \leq N/\left(\log N \right)$.
\end{itemize}
\end{prop}

\begin{proof}  Let $z = F(\zeta)$ and
suppose $\Im(z) \leq 20$ and let $\rho=\tau_k$ where $k$ is the first $n$
with $\Im[g_{\tau_n}(z)] \leq \lambda $. Then using Koebe's theorem, \[\dist\left(\zeta, \partial(D_A \setminus \hat\gamma_\rho) \right) \asymp N\frac{\lambda}{|g'_\rho(z)|}.\]If $\Square_\zeta$ is still open,
then \eqref{jan26.2} implies that
\[
\frac{\lambda}{|g'_\rho(z)|} \le c \lambda^{2 \nu^2}, \quad \nu = \left(\log N \right)^{-2/5}.
\]

Combining these estimates gives $\dist(\zeta,\partial ( D_A \setminus
 \hat{\gamma}_\rho) )=  o(N/\left(\log N \right)^5)$.  A similar argument (using Lemma~\ref{lem:coupling-of-maps})
shows the same for the LERW.   
\end{proof}
  
We can restate this as follows.  Suppose $\zeta \in A$ with $\Im[F(\zeta)]
\geq \lambda$,
\begin{itemize}
\item   The square $\Square_\zeta$ stays open until either the sine of the argument
gets too small or the imaginary part drops below $\lambda.$  We measure
the argument using the SLE path but by the coupling, since the imaginary part is at least $\lambda$, it is almost the same
as measuring using the LERW.
\item  If the sine gets too small, $\Square_\zeta$ closes.
\item  If  the imaginary part of $F_n^\SLE(\zeta)$
drops below $\lambda$ and  $\Square_\zeta$ has
not closed, we know that $\zeta$ is within distance $N/\left(\log N \right)^{5}$
of the boundary.
\item  The square now closes when the tip of the path
  gets distance $N/\log N$ away
from  $\zeta$. (This is defined separately for ``closed for SLE'' and
``closed for LERW''.)  It is possible that the square
$\Square_\zeta$ will be visited before it is closed; indeed, this is the
``typical'' behavior if the path will visit $\Square_\zeta$.
\end{itemize}
We will work with contents restricted to open squares. Define
\[ I_\zeta = c_*^{-1}\,1\{\zeta \in \eta\},\]
and
\[    I_\zeta^\circ = c_*^{-1}\, 1\{\exists k \mbox{ such that } \eta_k = \zeta
\mbox{ and } \Square_\zeta  \mbox{ is open for LERW at step } k-1\}. \]
Let
\[   T =  \sum_{\zeta \in \eta } I_\zeta, \;\;\;\;
  T_n = \sum_{\zeta \in \eta^n } I_\zeta, \]
\[    T^\circ = \sum_{\zeta \in \eta } I_\zeta^\circ , \quad T_n^\circ = \sum_{\zeta \in \eta^n } I_\zeta^\circ  \]
denote the number of  points and
number of open points visited by $\eta$ and $\eta^n$, respectively (both scaled by $c_*$).  

Now we define the corresponding SLE quantities. For each $\zeta \in A$, let
 \[j(\zeta) = \min\{n: \Square_\zeta {\text{ is closed for SLE at step }} n \} \] be the step at which  $\Square_\zeta$ closes for SLE and let  $\Theta^\circ_\zeta$
 denote the $5/4$-dimensional Minkowski content of the path in $\Square_\zeta$ before closing,
 \[     \Theta_\zeta =  
 \, \Cont\left[  \hat\gamma \cap
    \Square_\zeta \right], \;\;\;\;  \Theta^\circ_\zeta =
 \, \Cont\left[  \hat\gamma[0,\tau_{j(\zeta)}] \cap
    \Square_\zeta \right].\]
    Then we set 
    \[  \Theta = \Cont[\hat\gamma] =  \sum_{\zeta \in A} \Theta_\zeta, \;\;\;\;
      \Theta_n = \Cont\left[\hat\gamma[0,\tau_n]\right], \]
\[       \Theta^\circ = \sum_{\zeta \in A} \Theta^\circ_\zeta, \quad \Theta^\circ_n = 
 \sum_{\zeta \in A}   \Cont\left[\hat \gamma[0,\tau_{j(\zeta)} \wedge 
 \tau_n ] \cap
    \Square_\zeta \right].\]
 (There is some ambiguity in this notation.  We write $\Theta_\zeta$ and
 $\Theta_n$ and they mean different things whether or not the subscript
 is a point in $\Z^2$ ($\zeta$) or a nonnegative integer ($j,k,m,n$).
 We hope this will not cause confusion.)

\subsection{Proof of Theorem~\ref{thm:main-thm-core-sec}}\label{sect:main-proof}
The goal of this section is to prove the main result but we will leave proofs of some
facts for later sections. We will achieve this by proving the following statement.
\begin{prop}  \label{main-prop}
There exists $c$ such that for $N$ sufficiently large,
\begin{equation}  \label{feb18.5}
    \Prob\left\{\max_{0 \leq n \leq n_*} N^{-5/4} \, |T_n - \Theta_n| 
   \geq c\, \left(\log N \right)^{-1/60}  \right\}
      \leq c\, \left(\log N \right)^{-1/30}. \end{equation}
\end{prop}
We will argue that we can replace $T_n$ and $\Theta_n$ by $T_n^\circ$ and $\Theta_n^\circ$ as defined in the previous section. 

In this section stopping times and martingales will be discrete time
with respect to the filtration $\{\G_n\}$ of the coupling.

Note that
\[   \E\left[T^\circ \mid \G_n\right]
  = T_n^\circ + R_n^\circ, \;\;\;\;\;
   \mbox{ where } R_n^\circ=
     \sum_{\zeta \in A_n}  \openlerw_{n,\zeta}\,
     \E_n \left[I_\zeta^\circ\right], \]
  where we write  
 \[   \E_n \left[I_\zeta^\circ\right] =  
 \E_{A_n,a_n,b} \left[I_\zeta^\circ \right].\]
In particular,  $T_n^\circ + R_n^\circ$ is a martingale.
The corresponding SLE martingale is
  \begin{equation}  \label{feb18.1}
    \E\left[\Theta^\circ \mid \mathcal{G}_n \right] = \Theta^\circ_n + \sum_{\zeta \in A}
     \opensle_{n,\zeta} \,\E_n\left[\Theta^\circ_\zeta\right],
\end{equation}
 where
 $  \E_n\left[\Theta^\circ_\zeta\right]$
 is the expected value
 of  $\Theta^\circ_\zeta$  with respect to  SLE$_2$  from $\hat{\gamma}(\tau_n)$ to $b$ in
 $D_A \setminus \hat\gamma_{\tau_n}$.  
We consider the difference, which is also a martingale:
 \[    N^{-5/4}\E\left[ \Theta^\circ -   T^\circ
    \mid \G_n\right] =     Y_n^\circ + \tilde{B}_n^\circ , \]
    where
    \[    Y_n^\circ = N^{-5/4}\sum_{\zeta \in A}\left(\E_n \left[I_\zeta^\circ \right] - \E_n\left[\Theta^\circ_\zeta\right] \right) 
     , \;\;\;\;\;\;
        \tilde{B}_n^\circ  = N^{-5/4} \left[\Theta_n^\circ - \, T_n^\circ  \right].\] 
        It turns out to be convenient to modify this and replace $\tilde{B}_n^\circ$ by a predictable  (i.e., $\mathcal{G}_{n-1}$-measurable) version. For this we set
        \[
        B_n^\circ  = \sum_{j=1}^n\E \left[\tilde{B}_j^\circ-\tilde{B}_{j-1}^\circ  \mid \mathcal{G}_{j-1} \right].
        \]
        and define the martingale
        \[
        M_n^\circ = Y_n^\circ + B_n^\circ.
       \]
   The next lemma whose proof we delay shows that it suffices to
  prove \eqref{feb18.5} with $B_n^\circ$ in place of $N^{-5/4} \, (\Theta_n -
  T_n)$.
  
  \begin{lemma} \label{bvlemma}
   There exists $c < \infty$ such that
  \[  \Prob\left\{ \max_{n \leq n_0} \left|B_n^\circ - N^{-5/4}\, (\Theta_n -
   T_n) \right| \geq c \, \left(\log N \right)^{-5/128}
   \right\}  \leq c \, \left(\log N \right)^{-5/32}.\]
  \end{lemma}
  
  \begin{proof}  See Section~\ref{bvsec}, and in particular Proposition \ref{bvprop}.
  \end{proof}
  Given this, the strategy is to apply the following general lemma to the martingale $M_n^\circ = Y_k^\circ + B_k^\circ$ with  $\ee, \delta$ being chosen as suitable negative exponents of $\log N$.
\begin{lemma}\label{lem:l2-lemma}  Suppose $B_k, {M}_k$ are discrete time processes
with $M_k$ a square-integrable
martingale with respect to a filtration $\{\F_k\}$
with $M_0 = 0$.  Assume
that $B_k = X_k - Z_k$ where $X_k,Z_k$ are positive increasing
predictable (that is, $X_k,Z_k$ are $\F_{k-1}$-measurable)
processes with $X_0 = Z_0 = 0$.  Let $Y_k = M_k - B_k$.
Suppose that $\tau$ is a stopping time such that 
\[    \E[X_\tau  + Z_\tau ] \leq c_1, \]
and
\[       |Y_j| \leq \epsilon, \quad  |B_{j+1} - B_{j}| \le \ee, \quad \quad   j < \tau.\]
   Then for every $y > 0$,
 \[  \Prob\left\{\max_{0 \leq j \leq k \wedge \tau}
   |B_j| \geq y +2\epsilon \right\}
    \leq 
y^{-2} \, \left(\E \left[ Y_{k \wedge \tau}^2 \right] 
       + 3 \, \epsilon \, c_1\right). \]   
\end{lemma}
\begin{proof}See the end of the section.
\end{proof}
With this lemma in mind we see that we need to find a stopping time $\tau$ for which it holds that $\max_{n < \tau}|Y_n^\circ|$, $\max_{n < \tau}|B_{n}^\circ-B_{n-1}^\circ|$, and $\E\left[|Y_\tau^\circ|^2 \right]$ are all small. We will define the stopping time in terms of an estimate of $|Y_n^\circ|$, which we will now derive. For a fixed $n$, let
\[
S_n(\zeta) = \sin\left[\arg F_n^{\SLE}(\zeta)  \right]
\]
and then
\begin{align*}  A_n' & = \left\{\zeta \in A:\, \Im\left[F_n^\SLE(\zeta)\right]
   \geq \lambda\, ; \,S_n(\zeta) \geq  \left(\log N \right)^{-3/8} \right\},\\
    A_n'' &=  \left\{\zeta \in A:\, \Im\left[F_n^\SLE(\zeta)\right]
   \geq \lambda\, ; \,S_n(\zeta) <   \left(\log N \right)^{-3/8} \right\}. \end{align*}  
  The choice of $3/8$ is somewhat arbitrary and we have not optimized it.
We will use the fact that 
 $\frac 13 <\frac 38
    <\frac 25$.
 We write
 \[      \E\left[T^o \mid \G_n \right] = T_n^\circ+ \sum_{\zeta \in A_n'
 \cup A_n''} \opensle_{n,\zeta} \, \E_n\left[I_\zeta^o \right]
   + \sum_{\zeta \in A_n \setminus( A_n' \cup A_n'')} 
     \openlerw_{n,\zeta} \, \E_n\left[I_\zeta^o \right],\]
 \[    \E \left[\Theta^\circ\mid \G_n \right] = \Theta_n^\circ +
 \sum_{\zeta \in A_n'
  \cup A_n'' }  \opensle_{n,\zeta} \,\E_n\left[\Theta_\zeta^\circ \right]
   + \sum_{\zeta \in A_n \setminus ( A_n' \cup A_n'')} 
    \opensle_{n,\zeta} \,  \E_n \left[\Theta_\zeta^\circ \right]. \]
  Here we are using the fact that $\opensle_{n,\zeta}
   = \openlerw_{n,\zeta}$ in $A' \cup A''$.
Since
\[Y_n^\circ = N^{-5/4}\sum_{\zeta \in A}\left(\E_n \left[I_\zeta^\circ \right] - \E_n\left[\Theta^\circ_\zeta\right] \right) \]
we can estimate 
\begin{equation}  \label{feb23.18}
  |Y_n^\circ| \leq  | Y_n'|
      + Q_n + \tilde Q_n, 
      \end{equation}
 where
 \[     Y_n' = N^{-5/4}  \sum_{\zeta \in A_n'} \left(
   \E_n[I_\zeta^\circ] -  \E_n[\Theta_\zeta^\circ]\right), \]
 \[    Q_{n} =  N^{-5/4} \sum_{\zeta \in A_n''}  \opensle_{n,\zeta} \,  \left(
   \E_n[I_\zeta^\circ] +  \E_n[\Theta_\zeta^\circ]\right), \]
   \[   \tilde Q_n =  N^{-5/4} \sum_{\zeta \in A \setminus (A_n' \cup A_n'')} 
    \left( \openlerw_{n,\zeta}\,
   \E_n[I_\zeta^\circ] +  
    \opensle_{n,\zeta} \, \E_n[\Theta_\zeta^\circ]\right). \]
We can then describe the stopping time as follows.  Let $n_1$ be the minimum
of $n_*$ and the first $n$ such that either
\[   Q_n \geq  \frac 12 \, \left(\log N \right)^{-1/30}\]
or 
\[     \E\left[\maxsle + \maxlerw \mid \G_n \right]
    \geq \left(\log N \right)^{-1/2}. \] 
\begin{lemma}\label{lem:n_1}
We have
    \begin{equation}\label{feb26.1}
    \Prob\left\{n_1 < n_* \right\} = o\left(\left(\log N \right)^{-1/30}\right),
    \end{equation}
   \begin{equation}\label{feb23.192}
   Q_{n} \le \left(\log N \right)^{-1/30},\quad n \le n_1,
   \end{equation}
   \begin{equation}\label{feb23.212}
     \tilde{Q}_n \le \left(\log N \right)^{-1/2},\quad n< n_1,\end{equation}
   and \begin{equation}\label{feb23.213}
   \E\left[\tilde{Q}^2_{n_1} \right] = O\left(\left(\log N \right)^{-5/4}\right).
    \end{equation}
    \end{lemma}
    \begin{proof}Write $S_n(\zeta) = \sin\left[ \arg F_n^{\SLE}(\zeta)\right]$.
Note that if $n \leq n_1$, and $\zeta \in A_n''$, then deterministically
 (for $N$ sufficiently large)  \[S_{n-1}(\zeta) <  2\,\left(\log N \right)^{-3/8}, \quad \Im[F_{n-1}^\SLE(\zeta)] \geq \lambda.\] 
We shall prove in Proposition~\ref{feb23.prop1} that this gives \eqref{feb23.192}.
  On the other hand, as we will see, Proposition \ref{feb23.prop1} also shows that for any stopping time $\tau$ we have the estimate
$\E[Q_\tau] \leq O \left(\left(\log N \right)^{-1/8} \right)$, and hence 
\begin{equation}  \label{feb23.10}
      \Prob\{Q_{n_1} \geq \left(\log N \right)^{-1/16} \}\leq c \, \left(\log N \right)^{-1/16}.
   \end{equation}  
Using   Proposition \ref{prop.feb13}, we see that for any
stopping time $n \leq n_0$, 
\[    \tilde Q_n \leq \E \left[\maxsle \mid \G_n \right] + \E\left[\maxlerw \mid \G_n \right]\]
so we get \eqref{feb23.212}. Using Proposition \ref{maxprop} we see that
\[  \E\left[\E(\maxsle \mid \G_n)^2 \right]
  \leq \E\left[\E(\maxsletwo \mid \G_n) \right]
     \leq \E(\maxsletwo) \leq  c\, {\left(\log N \right)^{-5/4}}, \]
      and similarly for  $\E\left[\E(\maxlerw \mid \G_n)^2 \right]$.
Hence \eqref{feb23.213} follows.
Also, using Chebyshev's inequality,
\begin{equation}  \label{feb23.11}
 \Prob \left\{\E(\maxsle+
\maxlerw \mid \G_n)  \geq \left(\log N \right)^{-1/2} \right\}
   \leq  c \, \left(\log N \right)^{-1/4}.
   \end{equation}
Combining \eqref{feb23.10} and \eqref{feb23.11}, we get \eqref{feb26.1}. \end{proof}

 It remains to handle the main term, $Y_n'$. 
 \begin{lemma}
There is a constant $c< \infty$ such that if $n_1$ is as above, then
 \begin{equation}  \label{feb23.22}
    \E\left[ (Y_{n_1}')^2 \right] \leq c\, \left(\log N \right)^{-1/4}.
\end{equation}
\end{lemma}
\begin{proof}
Suppose $n \le n_1$.
We first use Lemma~\ref{lem:coupling-of-maps}
 to see that   if $\zeta \in A_n'$, then
 \[   F_n^{\LERW}(\zeta) =  F_n^{\SLE}(\zeta) \, \left[1 + O \left(h^{1/20} \right)
 \right]. \]
 Moreover, the Beurling estimate shows that (if $N$ is sufficiently large), $S_n(\zeta) \ge r_A(\zeta)^{-u}$ for all $\zeta \in A_n'$. Hence, from \eqref{BLV1}, integrating the Green's function over $\Square_\zeta$,
 \[  \E_n \left[ I_\zeta \right] = \E_n\left[\Theta_\zeta\right]
  \, \left[1 + O \left(h^{1/30} \right)
 \right].\]
Note that all closed squares in $A_n'$ are of Type I. Therefore using Proposition~\ref{feb23.prop1}, we see that
 \[     \E_n \left[ I_\zeta^\circ \right] = \E_n \left[ I_\zeta\right]
  \, \left[1+O\left(\left(\log N \right)^{-1/8}\right)\right], \]\[
    \E_n\left[ \Theta_\zeta^\circ \right]
  = \E_n \left[\Theta_\zeta \right]
  \, \left[1+O\left(\left(\log N \right)^{-1/8} \right)\right], \]
 and hence 
 \[     \left|\E_n \left[ I_\zeta^\circ \right] - \E_n \left[ \Theta_\zeta^\circ \right]\right|
    \leq  c\, \left(\log N \right)^{-1/8}\, \E_n \left[ \Theta_\zeta^\circ \right]. \]
Since  
\[
\sum_{\zeta \in A_n'}\E_n\left[ \Theta_\zeta^\circ\right] \le  \left| \E_n \left[ \Theta_\infty \right] -\Theta_n \right|,  
    \]
 after recalling that $Y_n'$ is rescaled, it follows that
 \[      \E\left[(Y_n')^2 \right] \leq c\, \left(\log N \right)^{-1/4}\, N^{-5/2} \,
 \E\left[ \left|\E_n\left[ \Theta_\infty - \Theta_n \right] \right|^2
 \right].\]
However, as shown in \cite{LZhou}, if $\kappa < 8$, and $\check{D}_t = \check{D} \setminus \check{\gamma}_t$, then for
any stopping time $\tau$,
\begin{align*}
\lefteqn{ N^{-5/2}\E\left[\E \left[\Theta_\infty - \Theta_\tau \mid \G_\tau \right]^2\right]}
  \hspace{.5in}\\
& = \E\left[\int_{\check D_\tau  \times \check D_\tau} G_{\check D_\tau}(z;\check \gamma(\tau),\check b)
 \, G_{\check{D}_\tau} (w;\check \gamma(\tau),\check b) \, dA(w)\,  dA(z) \right]\\
 &  \leq   
  c\, \E\left[\int_{\check D_\tau  \times \check D_\tau} G_{\check D_\tau}(z,w; 
  \check \gamma(\tau),\check b) \, dA(w)\,  dA(z) \right]\\
   & =   c\,  \E\left[\int_{\check D_\tau  \times \check D_\tau} \E\left[G_{\check D_\tau}(z, 
 w;\check \gamma(\tau),\check b)\mid \mathcal{G}_{\tau}\right] \, dA(w)\,  dA(z) \right]\\ 
 & \leq  c\,   \int_{\check D   \times \check D } G_{\check D}(z,w) \, 
  dA(w)\,  dA(z)     < \infty.
\end{align*}
  Here $G_{\check D_t}(z,w;\check\gamma(t), \check b)$ denotes the (unordered) two-point
  SLE$_\kappa$ Green's function which is a positive supermartingale justifying
  the last equality.  The first inequality is a general estimate
  about the two-point Green's function. The conclusion is that we have proved \eqref{feb23.22}. \end{proof}
\begin{proof}[Proof of Proposition~\ref{main-prop}]
  Combining \eqref{feb23.18},
  \eqref{feb23.192}, \eqref{feb23.212}, 
  and \eqref{feb23.22}, we see that
  \[  \E\left[Y_{n_1}^2\right] \leq c \, \left(\log N \right)^{-1/15}.\]
Proposition~\ref{main-prop} then follows from Lemma~\ref{lem:l2-lemma} using 
\[ \epsilon =  \left(\log N \right)^{-1/15} , \;\;\;\;
    y = \left(\log N \right)^{-1/60},\]
 to get
 \[ \Prob\left \{\max_{0 \leq j \leq n_1}
        |B_j| \geq 3\, \left(\log N \right)^{-1/60} \right\}
 \leq  c\, \left(\log N \right)^{-1/30}.\]
\end{proof}

It remains to prove Lemma~\ref{lem:l2-lemma}.
\begin{proof}[Proof of Lemma~\ref{lem:l2-lemma}]
We write $\Delta Y_j = Y_j - Y_{j-1},\,
\Delta B_j = B_{j} - B_{j-1},$ and  $\Delta M_j = M_j - M_{j-1}$.
Using the assumptions that $B_k$ is $\F_{k-1}$-measurable
and $M_k$ is a martingale, we get 
\begin{align*}
 \lefteqn{\E[Y_{k \wedge \tau}^2 \mid \F_{k-1}]   - Y_{(k-1) \wedge \tau}^2}
 \hspace{.4in}\\
  & =  1\{\tau >  n-1\} \, \left(2 Y_{k-1} \,\E\left[  \Delta Y_k 
   \mid \F_{k-1}\right] + \E \left[(\Delta Y_k)^2 \mid
     \F_{k-1} \right] \right)\\
  & =  1\{\tau > k-1\} \, \left(2  \,Y_{k-1} \,  \Delta B_k
   + (\Delta B_k)^2 + \E\left[(\Delta M_k)^2 \mid \F_{k-1} \right] \right)
 \end{align*}
By taking expectations of both sides and adding we see that 

\begin{align*}
\E[Y_{k \wedge \tau}^2]
  & =  \E \left[M_{k \wedge \tau}^2 \right]\\
  & \hspace{6ex}+
  2\sum_{j=1}^k \E \left[Y_{j-1} \,  \Delta B_j 
  ; \tau > j-1 \right] +\sum_{j=1}^k \E \left[(\Delta B_j)^2
  ; \tau > j-1 \right]\\
  & \geq  
 \E\left[M_{k \wedge \tau}^2 \right]\\
  & \hspace{6ex} - 
  2\epsilon \sum_{j=1}^n \E \left[ |\Delta B_j|
   ; \tau > j-1 \right] -\epsilon \sum_{j=1}^n \E \left[|\Delta B_j|
  ; \tau > j-1 \right]\\
  & \geq  \E\left[M_{k \wedge \tau}^2 \right] -
    3 \epsilon \, \E \left[X_\tau + Z_\tau \right ] .
  \end{align*}
Therefore,
\[ \E\left[M_{k \wedge \tau}^2 \right]
   \leq \E\left[Y_{k \wedge \tau}^2 \right] + 3 \epsilon \, 
     \E \left[X_\tau + Z_\tau \right] \leq \E\left[Y_{k \wedge \tau}^2 \right]
       + 3 \, \epsilon \, c_1. \]
Hence by the $L^2$ maximal principle,
\[  \Prob\left\{\max_{0 \leq j \leq k \wedge \tau}
   |M_j| \geq y \right\} \leq y^{-2} \, \left( \E\left[Y_{k \wedge \tau}^2 \right]
       + 3 \, \epsilon \, c_1 \right). \]
 Hence, recalling that $|B_j| = |M_j - Y_j|$, and $|Y_j|1_{j < \tau} \le \ee 1_{j < \tau}$,
\begin{align*}
 \Prob\left\{\max_{0 \leq j \leq k \wedge \tau}
   |B_j| \geq y +2\epsilon \right\}
    &  \leq   \Prob\left\{\max_{0 \leq j < k \wedge \tau}
   |B_j| \geq y + \epsilon \right\}\\
   & \leq    \Prob\left\{\max_{0 \leq j <k \wedge \tau}
   \left( |M_j| + |Y_j| \right) \geq y + \ee \right\}  \\ 
   & \leq \Prob\left\{\max_{0 \leq j <k \wedge \tau}
   |M_j|  \geq y \right\} \\
   & \le
y^{-2} \, \left(\E\left[Y_{k \wedge \tau}^2 \right] 
       + 3  c_1  \epsilon \right),
       \end{align*}
       which is what we wanted to prove.
\end{proof}

\section{Open and closed squares: estimates}  \label{badsec}
\subsection{Expected number of visits in closed squares}

In this subsection, we will show that the expected contribution
to the natural time for squares that are closed goes to zero
by proving the following.

\begin{prop}  \label{feb21.1}
There exists $c < \infty$ such that 
\[    \E[\Theta - \Theta^\circ] 
    +   \E[T - T^\circ] \leq c \, \left(\log N \right)^{-1/5}
     \, N^{5/4}.\]
\end{prop}  
Before giving the proof we need several lemmas.
Recall that
\[ \Theta - \Theta^\circ= \sum_{\zeta \in A}
   \E[\Theta_\zeta - \Theta^\circ_\zeta], \;\;\;\;
    \E[T - T^\circ] = \sum_{\zeta \in A}
   \E[I_\zeta -I^\circ_\zeta]. \]
We prove the estimates separately for SLE and LERW although the arguments
are similar.  
We start with a simple estimate that uses only the smoothness
of $D$ and the Green's function for SLE and LERW.

 \begin{lemma} \label{feb22.lemma1} Suppose $D$ is an analytic domain with $A=A(N,D)$.
 $\;$  There exists $c < \infty$ such that
for all $\delta > 0$ the following statements hold.
 \begin{enumerate}
 \item  If $A^{\delta,1}
  = \{\zeta \in A: S_A(\zeta;a,b) \leq \delta \}$. 
  \[    N^{-5/4}\sum_{\zeta \in A^{\delta,1}}  \E\left[
       \Theta_\zeta + I_\zeta\right]
          \leq c
            \, \delta^{13/4}.\]
            
  \item If  $A^{\delta,2}
   = \{\zeta \in A:\dist(\zeta, \partial D_A) \leq \delta
  N\, \}$, then 
       \[      N^{-5/4}\sum_{\zeta \in A^{\delta,2}}  \E\left[
       \Theta_\zeta + I_\zeta\right]\leq c
            \, \delta^{5/4}.\]     
  \end{enumerate}
   \end{lemma}

   \begin{proof}  We use only the Green's function estimate
   \[     \E\left[
       \Theta_\zeta + I_\zeta\right] \leq c \,r_A(\zeta)^{-3/4}
       \, \left[S_{A}(\zeta;a,b)^3 + O(r_A(\zeta)^{-u}) \right].\]  
\begin{enumerate}
\item For $\zeta$ that are distance
 $2^{-k}N$ to $2^{-k+1}N$ from $a$, in order for $S_A(\zeta;a,b)$
 to be less than 
$\delta$, the points must be with distance $O(\delta 2^{-k}N)$
 of the boundary.  The number of such points is $O(2^{-2k} \, \delta
 \, N^2 )$
 and the value of $\E[\Theta_\zeta + I_\zeta]$
  for these points
 is bounded by  $O((\delta 2^{-k} N)^{-3/4}  \, \delta^3).$  Hence
 the sum over this region is bounded by $2^{-5k/4} \, \delta
 ^{13/4}\, N^{5/4}$.  We can sum over $k$ and handle points near $b$ similarly.
 \item
 The sum over $\zeta$  at distance $O(\delta N)$ of $a$ or $b$ is
 $O(\delta^{5/4} \, N^{5/4})$. For the points that are distance
 between $k \delta N$ and $(k+1) \delta N$, there are $O( (\delta N)^2)$
 points with typical value of the Green's function being of order
 $k^{-3} \, (\delta N)^{-3/4}$.  Hence the sum over that region
 is $O(k^{-3} \,\delta^{5/4} \, N^{5/4})$ and we can sum over $k$.
\end{enumerate}

 \end{proof}

We will consider separately  ``Type I'' and  ``Type II''
closures
using the notation of Definition \ref{sleopendef}.

\begin{lemma}  \label{feb23.sle}
There exists $c < \infty$ such that
the following holds. Suppose $D$ is a simply connected domain containing
the origin,  and $a,b$
are distinct boundary points.  Let $\gamma$ be an $SLE_2$ path
from $a$ to $b$ in $D$ and let
\[        S_t =  S_{D \setminus \gamma_t}(0;\gamma(t),b) . \]
Let
\[   \sigma_s = \inf\left\{t: |\gamma(t)| \leq e^{-s }\right\}, \]
where we set $\sigma_s = \infty$ if $\dist(\gamma,0)
 > e^{-s}$.  Let
 \[         \Psi = \Psi_s =   \min_{0 \leq t \leq
     \sigma_{s-2} }   S_t. \]
Then,
\[      \Prob\{\Psi \leq \delta; \sigma_{s} < \infty\}
   \leq    c \, e^{-3s/4} \, \delta^3  \, \left[s\,S_0^{3}
      +  1 \right].\]

 \end{lemma}

\begin{proof}  Let $\rho = \inf\{t: S_t \leq \delta\}$,
and let 
\[   E_k =   \{\sigma_{k-1} \leq \rho < \sigma_k \}, \;\;\;\;\;
  V_k = \{\sigma_k < \infty\}. \]
Set $k^* = \lceil s \rceil$.   We
then have
\begin{align*}
 \Prob\{\Psi \leq \delta; \sigma_{s} < \infty\}
 & \le    \sum_{k=1}
 ^{k^*-2} 
     \Prob(E_k \cap V_s)\\
     &  =  \sum_{k=1}
 ^{k^*-2} 
     \Prob(E_k \cap V_{k-1})  \,  \Prob(V_s \mid
     E_k \cap V_{k-1}).
     \end{align*}
 The strong Markov property applied at time $\rho$ implies that
 for $k \leq k^*$, 
\begin{equation}  \label{mar7.1}
  \Prob(V_s \mid
     E_k \cap V_{k-1}) \leq c \, \delta^3 \, e^{3(k-s)/4}.
     \end{equation}
 If $k=1$, we will use the trivial bound $\Prob(E_1 \cap V_0) \leq 1$.
 However, for $k  > 1$, we use
 \[  \Prob(E_k \cap V_{k-1}) \leq \Prob(V_{k-1})
   \leq c \, S_0^3 \, e^{-3k/4}.\]
   The lemma is obtained by summing over $k$.

\end{proof}

  \begin{lemma}  \label{feb23.lerw} There exists $c,q$ such that the following is true.
Suppose $(A,a,b) \in \whoknows$ and let $\delta >0$.  In the
measure $\Prob_{A,a,b}$ let $\Psi$ be the minimum of $S_k$ over 
times $k$ before the first visit to the disk of radius $\delta^{-q}$
about the origin.  Then,
\[      \Prob_{A,a,b} \{\Psi \leq \delta; 0 \in \eta \}
   \leq    c \, r_A^{-3/4}  \, \delta^3  \, \left[s\,S_0^{3}
      +  1 \right].\]
\end{lemma}

\begin{proof}  This is proved similarly as the
previous lemma  using \eqref{jul14.0} to justify
the analogue of \eqref{mar7.1}.  The
condition on $\delta^{-q}$ is included so that the error terms
in \eqref{jul14.0} are smaller than the dominant term.
\end{proof}

\begin{prop}  \label{feb23.prop1} \label{feb23.prop1alt}
There exists $c < \infty$ such that
the following holds. Suppose $\Im F(\zeta) \ge \lambda$ and    $1/3<p<2/3$. Let
$  S_n^\SLE(\zeta) =
 \sin\left[ \arg  F_{ n}^\SLE(\zeta)\right]$,   let 
  $\rho_\zeta$ be the
first $n$ such that \[ S_n^\SLE(\zeta)
  \leq 2 \, \left(\log N \right)^{-p},\] and let $\tau_\zeta$
 be the first $n$ such that $\Im[ F_{n-1}^\SLE(\zeta)]
  < \lambda$. Let $Q_I = Q_I^{\SLE}
   + Q_I^{\LERW}$ where 
  \[   Q_I^{\SLE} = N^{-5/4} \sum_{\zeta}
         1\{\rho_\zeta  \leq \tau_\zeta \}\,
       \Theta_\zeta ,\]\[
         Q_I^{\LERW} = N^{-5/4} \sum_{\zeta}
         1\{\rho_\zeta  \leq \tau_\zeta \}\,
       1\{\zeta \in \eta\}  .\] 
Then,
$$ \E\left[Q_I
  \right]  \leq  c \, \left(\log N \right)^{-(3p-1)}.$$
In particular, for every stopping time $\sigma$
and every $r > 0$,
\[  \Prob
\left\{\E\left[Q_I \mid \G_\sigma\right]  \geq r \right\}
  \leq c \, r^{-1} \, \left(\log N \right)^{-(3p-1)}.\]
\end{prop}

\begin{proof}  
Let $\sigma_\zeta$ be the hitting time of $\Square_\zeta$. The Beurling estimate and Lemma \ref{feb23.sle} with $\delta =
  2 \, \left(\log N \right)^{-p}$ shows that \[\Prob\left\{\rho_\zeta \le \tau_\zeta \right\} \le \Prob \left\{ \rho_\zeta \le \sigma_\zeta \right\} \le c r_A(\zeta)^{-3/4}\left(\log N \right)^{-3p}.\] 
 Assume that $S(\zeta) \ge \left(\log N \right)^{-1/3}$. Then the Green's function satisfies \[G(\zeta) \ge c r_A(\zeta)^{-3/4} \left(\log N \right)^{-1}\]  
and consequently, since $\E\left[ \Theta_\zeta\right] = \int_{\Square_\zeta}G(z) \, dA(z) \asymp G(\zeta)$, 
 \begin{align*}   
 \E\left[\Theta_\zeta; \rho_\zeta  \leq \tau_\zeta \right] & = \E\left[\Theta_{\zeta} \mid \rho_\zeta  
  \leq \tau_\zeta \right]\Prob\left\{\rho_\zeta \le \tau_\zeta \right\}\\
   &\leq c \, \left(\log N \right)^{-(3p-1)}
     \, \E[\Theta_\zeta].\end{align*}
Therefore,
\begin{align*}
  \E\left[Q_I^{\SLE} \right]
& \leq  c \, \left(\log N \right)^{-(3p-1)} \, N^{-5/4} \,\E[\Theta]
   + \sum_{S(\zeta) \leq \left(\log N \right)^{-1/3}}\,
   N^{-5/4}\E[\Theta_\zeta]\\
   &   \leq   c \, \left(\log N \right)^{-(3p-1)}.
   \end{align*}
     
 The estimate for $Q_I^{\LERW}$ is done similarly
 using Lemma \ref{feb23.lerw}.

\end{proof}

For the Type II closures,  we start with the following
lemma, see \cite{lawler_continuity, lawler_field}.
\begin{lemma}   \label{lemmaII}
 There exists $c < \infty$ such that if $D$
is a simply connected domain containing the unit disk  with
distinct boundary points $a,b$ with $|a| = 1$, $0 < r \leq 1/2$, and
\[   \tau_r = \min\{t: |\gamma(t)|  = r \} \]
\[   \tau_R = \min\{t: |\gamma(t)| = R \}, \]
then 
\[  \Prob\{ \tau_R < \tau_r
  \mid \tau_r < \infty\} \leq c \, R^{-3/2}. \]
 \end{lemma}
 \begin{prop}\label{feb23.prop2}
There exists $c < \infty$ such that
the following holds.  Let
  $\rho_\zeta$ be the
first $n$ such that $\dist(\zeta, 
D \setminus \gamma_{\tau_n}) \leq \left(\log N \right)^{-5}
\, N$ and let $\psi_\zeta$ be the first
$m > \rho_\zeta$  such that $|\gamma(\tau_m) -
  \zeta| \geq \left(\log N \right)^{-1}\, N$.  Let
  \[   Q_{II}^{\SLE} = N^{-5/4} \sum_{\zeta}
         1\{\psi_\zeta  < \infty\}\,
       \Theta_\zeta   .\] 
Then,
\[ \E\left[Q_{II}^\SLE
  \right]  \leq  c \, \left(\log N \right)^{-6}.\]
In particular, for every stopping time $\sigma$
and every $r > 0$,
\[  \Prob
\left\{\E \left[Q_{II}^{\SLE} \mid \G_\sigma \right]  \geq r \right\}
  \leq c \, r^{-1} \, \left(\log N \right)^{-6}.\]
\end{prop}

\begin{proof}  Fix $\zeta$, let $\tau_\zeta = \tau_{\psi_\zeta}$
 and let $\sigma_\zeta$ be the
first $t$ with $|\gamma(t) - \zeta| \leq 4$.
Then,
\[   \Prob\{\tau_\zeta < \sigma_\zeta < \infty\}
    \leq c \, \left(\log N \right)^{-6} \, \Prob\{\sigma_\zeta < \infty\}
       \leq c \, \left(\log N \right)^{-6} \, \E\left[\Theta_\zeta\right].\]
 Also,
 \[   \E\left[\Theta_\zeta \mid \tau_\zeta < \sigma_\zeta < \infty
    \right] \asymp  1.\]
 Therefore,
 \[   \E\left[\Theta_\zeta \,;\, 
  \tau_\zeta < \sigma_\zeta < \infty\right] \leq c \, 
  \left(\log N \right)^{-6} \, \E\left[\Theta_\zeta\right].\]
 
 There is another term which corresponds to the event 
 $\sigma_\zeta <  \tau_\zeta $.  Given this event, we need the
 expected Minkowski content of $\gamma[\tau_\zeta,\infty)
   \cap \Square_\zeta$.  Using  Lemma \ref{lemmaII}, we can
   see that
 \[ \E\left[\Cont (\gamma[\tau_\zeta,\infty)
   \cap \Square_\zeta) \mid \sigma_\zeta <  \tau_\zeta\right]
     = o\left(\left(\log N \right)^{-6}\right).\]
  Therefore,
  \[   \E\left[\Theta_\zeta \,;\, 
  \psi_\zeta < \infty\right] \leq c \, 
  \left(\log N \right)^{-6} \, \E\left[\Theta_\zeta\right],\]
 \[   \E\left[Q^{\SLE}_{II} \right]
    \leq c \, \left(\log N \right)^{-6} \, N^{-5/4}
   \sum_\zeta \E\left[\Theta_\zeta\right]
    \leq c\,  \left(\log N \right)^{-6}.\]

\end{proof}

\begin{prop}
There exists $c < \infty$ such that
the following holds.  Let
  $\rho_\zeta$ be the
first $n$ such that $\dist(\zeta, 
\partial D_n  ) \leq \left(\log N \right)^{-5}
\, N$ and let $\psi_\zeta$ be the first
$m > n$  such that $|\eta_{\sigma_n} -
  \zeta| \geq \left(\log N \right)^{-1} \, N$.  Let
  \[   Q_{II}^{\LERW} = N^{-5/4} \sum_{\zeta}
         1\{\tau_\zeta  < \infty\}\,
       \Theta_\zeta   .\] 
Then,
$ \E\left[Q_{II}^\LERW
  \right]  \leq  c \, \left(\log N \right)^{-4}.$
In particular, for every stopping time $\sigma$
and every $r > 0$,
\[  \Prob
\left \{\E[Q_I \mid \G_\sigma]  \geq r \right\}
  \leq c \, r^{-1} \, \left(\log N \right)^{-4}.\]
\end{prop}

\begin{proof}  This is proved in the same way
using Proposition \ref{dec23.1}.
\end{proof}
\begin{proof}[Proof of Proposition \ref{feb21.1}]
The proof follows from Proposition~\ref{feb23.prop1} by choosing $p=2/5$ together with Proposition~\ref{feb23.prop2}.
\end{proof}
\subsection{Comparing the time  and the predictable version}
   \label{bvsec}
Recall the definition of $\tilde{B}_n^\circ = N^{-5/4}(\Theta^\circ_n - T_n^\circ)$ and the predictable version \[
B_n^\circ = \sum_{j=1}^n \E_{j-1}\left[\tilde{B}_j^\circ - \tilde{B}_{j-1}^\circ\right].
\] In this section we will show that $B_n^\circ$ is close to $N^{-5/4}\left(\Theta_n-T_n \right)$, that is, prove Lemma~\ref{bvlemma}. We will do this in two steps: we first first compare $B_n^\circ$ with $\tilde{B}_n^\circ$ and then $\tilde{B}_n^\circ$ with $N^{-5/4}\left(\Theta_n-T_n \right)$. We will argue separately for the LERW and SLE parts, but the idea of the argument is the same in both cases.

One of the basic theorems of martingale theory is that a continuous
 martingale with paths of bounded variation is zero.  The
 next lemma  can be considered a
  discrete time analogue  where  the notions
 of ``continuous'' and ``bounded variation'' are approximated.

\begin{lemma} \label{bvlemma2}
Suppose $\{X_k\}$ is an increasing
process with $X_0 = 0$
 adapted to $\{\G_k\}$. Let $\Delta_k = X_k - X_{k-1}$,
 $L_k = \E[\Delta_k
 \mid \G_{k-1}],$ and let $Z_k$ be the martingale
 \[    Z_n = \sum_{j=1}^n [\Delta_j - L_j] = X_n - \tilde X_n , \;\;\;\;
  \mbox{where } \;\; \tilde X_n  = \sum_{j=1}^n L_j .\]
 Let
  \[  J_n = \max\{\Delta_j : j=1,\ldots,n\},\]
  \[  \bar Z_n = \max \{|Z_j|: j=1,\ldots,n\}.\]
Suppose that $\E \left[J_n^2 \right] \leq \epsilon^2 \leq 1$
 and $\E \left[X_n \right] = K < \infty$. 
  Then.
  \[    \Prob\{\bar Z_n \geq \epsilon^{1/16}\}
    \leq  7\, \epsilon^{1/4} + 2 \, \epsilon^{1/2}
     \, K.
    \]
  \end{lemma}
  
In the hypotheses the bound on $\E \left[J_n^2 \right]$ can be considered
an ``almost continuous paths'' assumption and the bound
on $\E[X_n]$ will give the bound on the total variation
of $Z_n$.
       
\begin{proof}  
Note that $\E \left[J_n \right] \leq \sqrt{\E \left[J_n^2 \right]} \leq  \epsilon$
and $\E \left[\tilde X_n \right]
  = \E \left[X_n \right] \leq K$. 
  Let   $\tau$ be the minimum of $n$ and the first $k$
with $L_{k+1} \geq \epsilon^{3/4}$. (Note that $L_k$ is predictable.) If $k < n$, then by definition
$L_k \leq \E \left[J_n \mid \G_{k-1} \right]$.   Hence, 
\[  \E \left[L_{\tau + 1}^2; \tau < n \right] \leq  \E\left[\E( J_n \mid \G_\tau)^2
\right] \leq \E\left[\E(J_n^2 \mid \G_\tau) \right]
   = \E\left[J_n^2\right] \leq \epsilon^2.\]
and hence 
\[  \Prob\left\{\tau < n\right\} = 
 \Prob\left\{L_{\tau +1} \geq \epsilon^{3/4}; \tau < n \right\}
  \leq \epsilon^{-3/2} \, \E \left[ L_{\tau+1}^2; \tau < n
\right] \leq  \epsilon^{1/2} .\]
Also,  if $j \leq \tau$,
\[   |\Delta_j - L_j| \leq
  \Delta_j + L_j \leq   J_n + \epsilon^{ 3/4} .\]

  Let
   $\sigma = \sigma_n$ be the minimum of $\tau$
  and 
 \[     \min\{j: X_j + \tilde X_j \geq  \epsilon^{-1/2} \}.\]
Note
 that 
\begin{align*}
  \sum_{j=1}^{ \sigma}
   |\Delta_j - L_j| & \leq     X_{(n \wedge \sigma) - 1}
    + \tilde X_{(n \wedge \sigma) -1} +   \Delta_{n \wedge \sigma}
       + L_{n \wedge \sigma}   \\
     & \leq    \epsilon^{-1/2} + J_n + \epsilon^{3/4} 
    \\
   &\leq  2\,\epsilon^{-1/2} + J_n.
   \end{align*}
 Therefore, 
 \begin{align*}
  \E \left[Z_{\sigma}^2 \right]    &  =    \sum_{j=1}^
  { n} \E\left[(\Delta_j- L_j)^2; \sigma \leq j\right] \\
    & \leq  \E\left[(J_n + \epsilon^{3/4}) \,  \sum_{j=1}^{ \sigma}
      |\Delta_j-L_j |\right]\\
      &  \leq  \E\left[(J_n + \epsilon^{3/4})\, (2\epsilon^{-1/2}+ J_n)\right]
         \\
      & \leq  2 \epsilon^{1/4} + \left[\epsilon^{3/4} + 2 \epsilon^{-1/2} \right]
       \, \E(J_n) + \E(J_n^2) \\
       & \leq   2 \epsilon^{1/4} + \epsilon^{7/4}
           + 2 \epsilon^{1/2} +\epsilon^2 \leq    6 \, \epsilon^{1/4}
   \end{align*}
 By the $L^2$-maximal inequality, we see that
 \[   \Prob\{\bar Z_\sigma \geq \epsilon^{1/16}\}
   \leq \epsilon^{-1/8} \, \left(
   6\,\epsilon^{1/4}
   \right)  \leq  6 \, \epsilon^{1/8}.\]
  Also,
  \begin{align*}
    \Prob\{\sigma <  \tau \}
  & \leq     \Prob\left\{X_n + \tilde X_n \geq 
  \epsilon^{-1/2}
  \right\} \\
  & \leq  \epsilon^{1/2} \, \left(\E[X_n] + \E
  [\tilde X_n] \right) \leq 2 \, \epsilon^{1/2} \,
   K.
   \end{align*}
 Since
 \[ \{\bar Z_n \geq \epsilon^{1/16}\}
 \subset \{\bar Z_\sigma \geq \epsilon^{1/16}\}
  \cup \{\sigma < \tau\} 
   \cup \{\tau < n\},\]
the proof is finished. 
\end{proof}

\begin{prop}
There exists $c < \infty$ such that if 
\[    Z_n^{\LERW} = N^{-5/4} \left(T_n^\circ - \sum_{j=1}^n \E\left[T_j^\circ -
  T_{j-1}^\circ \mid \G_n \right]\right), \]
\[     Z_n^{\SLE} = N^{-5/4} \left(\Theta_n^\circ - \sum_{j=1}^n \E\left[
\Theta_j^\circ -
  \Theta_{j-1}^\circ \mid \G_n \right]\right), \]
then, 
\[     \Prob\left\{\max_{1 \leq j \leq n_*} |Z_n^\LERW| \geq \left(\log N \right)^{-5/128}
\right\}    \leq c \, \left(\log N \right)^{-5/32}. \]
\[   \Prob\left\{\max_{1 \leq j \leq n_*} |Z_n^\SLE| \geq \left(\log N \right)^{-5/128}
\right\}    \leq c \, \left(\log N \right)^{-5/32}. \]

\end{prop}

\begin{proof}  We will apply the previous lemma with $Z = Z^{\LERW}, Z = Z^{\SLE}$.
We claim that
\[   \max_{1 \leq j \leq n_*}    N^{-5/4} \,\left(T_j^\circ - T_{j-1}^\circ
 \right)  \leq \maxlerw, \]
 \[  \max_{1 \leq j \leq n_*}    N^{-5/4} \,\left(\Theta_j^\circ - \Theta_{j-1}^\circ
 \right)  \leq \maxsle. \]
 To see this, note that no vertex $\zeta$  with $\Im[F_n^\SLE(\zeta)]
  \geq \lambda$  can be reached by an increment of capacity
  $O(h)$.  Hence the only points  that could be visited have
  $\Im[F_n^\SLE(\zeta)]
   \leq \lambda.$  By Proposition \ref{prop.feb13},  if
   $\Im[F_n^\SLE(\zeta)]
   \leq \lambda$ and $S_\zeta$ is open,
   then it is within  distance $N/\left(\log N \right)$  of  $a_n$.  
   Therefore   $T_{n+1}^\circ - T_{n}^\circ$ is bounded
  above by  the number of sites visited by the 
 walk within distance $N/\left(\log N \right)$  of  $a_n$ and this
 is bounded by $N^{5/4} \, \maxlerw$.  The argument in the SLE case is the same.

Moreover, using Proposition~\ref{maxprop} we know that 
\[ \E[\maxlerwtwo + \maxsletwo] \leq c \, \left(\log N \right)^{-5/4}.\]  Also
$\E[T_n^\circ + \Theta_n^\circ] \leq \E[T_n + \Theta_n] \leq c \, N^{5/4}$. 
Hence we can use the lemma with $\epsilon = O(\left(\log N \right)^{-5/8}),
K = O(1)$.

\end{proof}

\begin{prop}  \label{bvprop}

There exists $c < \infty$ such that if 
\[    \hat Z_n^{\LERW} = N^{-5/4} \left(T_n  - \sum_{j=1}^n \E\left[T_j^\circ -
  T_{j-1}^\circ \mid \G_n \right]\right), \]
\[     \hat Z_n^{\SLE} = N^{-5/4} \left(\Theta_n   - \sum_{j=1}^n \E\left[
\Theta_j^\circ -
  \Theta_{j-1}^\circ \mid \G_n \right]\right), \]
then, 
\[     \Prob\left\{\max_{1 \leq j \leq n_*} |\hat Z_n^\LERW| \geq c\, \left(\log N \right)^{-5/128}
\right\}    \leq c \, \left(\log N \right)^{-5/32}. \]
\[   \Prob\left\{\max_{1 \leq j \leq n_*} |\hat Z_n^\SLE| \geq c\, \left(\log N \right)^{-5/128}
\right\}    \leq c \, \left(\log N \right)^{-5/32}. \]

 \end{prop}
 
\begin{proof}
In Proposition \ref{feb21.1} we show in fact that
\[   \E[T  - T ^\circ] + \E[\Theta - \Theta^\circ]
  \leq c\, \left(\log N \right)^{-1/5} \, N^{5/4}. \]
It follows from the Markov inequality that
\[  \Prob\left\{(T - T^ \circ) +
   ( \Theta- \Theta^\circ) \geq  \left(\log N \right)^{-5/128} \,  N^{5/4}
   \right\} \leq  o((\log N)^{-5/32}).
   \]
 Since
 \[   0 \leq  T_n - T_n^\circ \leq T - T^\circ, \;\;\;\;
   0 \leq  \Theta_n - \Theta_n^\circ \leq \Theta - \Theta^\circ,\]
 we get the result. 
 \end{proof}

\section{LERW estimates}\label{LERWsec}

\subsection{Introduction and notation}

In this section we establish some ``two-point''
estimates
  about LERW that have
  independent interest.   Because we are working
  only
  with the LERW and not the scaling limit, we will  
  consider subsets of $\Z^2$ rather than of $N^{-1} \, \Z^2$. 
  We make the convention that 
all constants, including
implicit constants in $O(\cdot)$ or $\asymp$ notation,
are assumed to be universal, that is, do not depend on
$A,r,a,b,z,w,\ldots$.   We will use the notation
from Section \ref{sect:set-up} and some more that
we give now. This section does not use any results about SLE.

\begin{itemize}

\item Let $\whoknows$
denote the set of triples $(A, a, b)$ where
  $A$  is  a finite, simply connected subset of $\Z^2$  
containing the origin,  and $  a,   b$ are distinct
elements
of $\partial_e A$, the edge boundary of $A$.

\item We identify the edge $a$ with its midpoint
and write $e_a,e_b$
for the directed inward pointing edges which start in $\partial A$ and end in $A$.
We will write  $a^*,b^* \in A$ (rather than $a_+,b_+$)
for the terminal points of $e_a,e_b$.

\item If $(A,a,b) \in \whoknows$,  we let
$f = f_A$ be the
  unique conformal transformation $f: D_A \rightarrow
 \Disk$ with $f(0) = 0, \arg f(a) = 0$.  We set
 \[    r_A = r_A(0) = |f'(0)|^{-1} , \]
 \[         S_{A,a,b} 
      =    \sin\left[ \frac{\arg f(b)}2\right] 
      ,\]
One can check that these definitions
agree with our  
previous definitions of $r_A, S_{A,a,b}$.

\item  If $r \geq 1$, we let $C_r$ denote the discrete open disk of radius
$r$ about the origin,
\[    C_r = \{z \in \Z^2: |z| < r \}. \]
If $\zeta \in \Z^2$, we let $C_r(\zeta) = C_r + \zeta
 =  \{z + \zeta: z
\in C_r\}$.

\item Let $I_r$ be the set of self-avoiding walks (SAWs) that include at least one
vertex in $C_r$.  Note that $I_1$ is the set of SAWs that
go through the origin.

\item
Let $\whoknows_r$ be the set of $(A,a,b)$ such that
$C_r \subset A$, that is, such
that  $\dist(0, \partial A)
 \geq r$.  In particular,  $\whoknows_1 = \whoknows$.

 \item  Let $\soup_r$ be the set of
$(A,a,b) \in \whoknows_r$ such that $a^*,b^* \in C_r$.
Note that if $(A,a,b) \in \soup_r$, then
\[     r \leq \dist(0,\partial A) < r+1, \;\;\;\;
r-1 \leq |a^*|,|b^*| < r.\]

\end{itemize}

Using the Koebe $1/4$-theorem  we see that
$  r_A \asymp r $ if $  (A,a,b)
 \in \soup_r.$
Also, we claim that there exists $c > 0$
such that if $(A,a,b) \in \soup_r$, then 
\begin{equation} \label{sep25.1}
      S_{A,a,b} \geq c \, \frac{|a-b|}{r}.
      \end{equation}
One way to prove this is  to  consider a Brownian
motion starting at the origin and estimating the
probability of the  event
that the path makes clockwise and counterclockwise loops
around $a$ without encircling $b$.
Since we will not need this estimate, we will not give
a full proof.  We will, however, need the following
special case which can be proved in this way;
we  leave the  details 
  to the reader.
\begin{itemize}
\item For every $\delta > 0$, there exists $c_\delta = c(\delta) >0$
such that if $(A,a,b) \in \soup_r$ with $|a-b|
 \geq \delta \, r$, then
\begin{equation} \label{sep25.11}
        S_{A,a,b} \geq c_\delta . 
\end{equation}
\end{itemize}
We note for interest that the analogous upper bound in \eqref{sep25.1}
does not hold.  Using the Beurling estimate, we could show that
$
    r \, S_{A,a,b} \leq c \, |a-b|^{1/2}$
but we will not need this. 
 
\subsection{Statements}\label{sect:main-lerw}

We will state the main estimates here leaving some of
the proofs until later.
We start by restating the main result from \cite{BLV}.

\begin{thm}  \label{BLV}
There exists $\hat c \in (0, \infty)$ and $u >0$ such that if
$(A,a,b) \in \whoknows$, then
\[
   \Prob_{A,a,b}\{0 \in \eta\}
    = \hat c \, r_A^{-3/4} \, \left[S_{A,a,b}^3 +
      O(r_A^{-u})\right].
      \]
In particular, there exist $ 0 < c_1 < c_2 < \infty$,
such that if $(A,a,b) \in \soup_r$,
\begin{equation}  \label{jul14.0}
 c_1 \, r^{-3/4} \, 
 \left[S_{A,a,b}^3 
      -r^{-u}\right]  \leq
      \Prob_{A,a,b}\{0 \in \eta\} \leq
      c_2 \, r^{-3/4} \, 
 \left[S_{A,a,b}^3 
      +r^{-u}\right]. \end{equation}
        
\end{thm}

A standard technique for
estimating the probability of hitting or getting
near a point (for example, the origin) is to 
observe the path up to the first time it gets within
a fixed distance, say $r$,  of the point.  We will do something
similar here, except that we will grow the loop-erased
walk from both the beginning and the end.  If a path from
$a$ to $b$ enters $C_r$ we consider the first and last
visits to $C_r$, that is, the path up to the
 first visit 
 and the reversed path up to its first visit.

To be more precise, 
   if $(A,a,b)
 \in \whoknows_r$, and  $\eta \in \saws(A,a,b)
  \cap I_r$, then there is a unique decomposition
  \[   \eta = \eta^1 \oplus \tilde \eta \oplus
  \eta^2, \] 
where $\eta^1$ is the initial segment of $\eta$
stopped at the first visit to $C_r$ and $[\eta^2]^R$
is the initial segment of $\eta^R$ stopped
at the first visit to $C_r$.  We write
$(A_r,a_r,b_r) \in \soup_r$, where $a_r,b_r$
are the final edges of $\eta^1,[\eta^2]^R$,
respectively; $a_r^*,b_r^* \in C_r$ are the
terminal vertices of $\eta^1,[\eta^2]^R$,
respectively; and $A_r$ is the connected component
of $(A \setminus [\eta^1 \cup \eta^2]) \cup
\{a^*_r,b_r^*\}$ containing the origin.  We also write $S_r =
S_{A_r,a_r,b_r}$; 
 if $\eta \not\in I_r$,
we set $S_r = 0$. 

If $(A,a,b) \in \whoknows_{2r}$, then we would like to say that
\[    \Prob_{A,a,b}[I_r] :=  \Prob_{A,a,b}\{ \eta \in I_r\}
\asymp r^{3/4} \, \Prob_{A,a,b}\{0 \in \eta\}, \]
or equivalently, that $$  \Prob_{A,a,b}\{0 \in \eta \mid
 \eta \in I_r\} \asymp r^{-3/4}.$$  This will follow from 
\eqref{jul14.0} provided that with a reasonable probability
we know that $S_r$ is not too small, see Figure~\ref{fig:separation}.
  The technical tool for
establishing this is called a ``separation lemma''.   We will
need to prove a particular version here, but the basic idea
of the proof is similar to other versions (see, e.g.,
\cite{Lcut,Masson}).  This can be considered as a 
generalization of a boundary Harnack principle. Roughly speaking,
if we condition a process to get away from a boundary point by
a certain time, then it is unlikely to stay near the boundary
for a long period of time. Indeed, the probability of staying close
to the boundary for a long period of time is much less than
the probability of getting away quickly. 

\begin{figure}[t]
\centering
  \def\svgwidth{0.7\columnwidth}
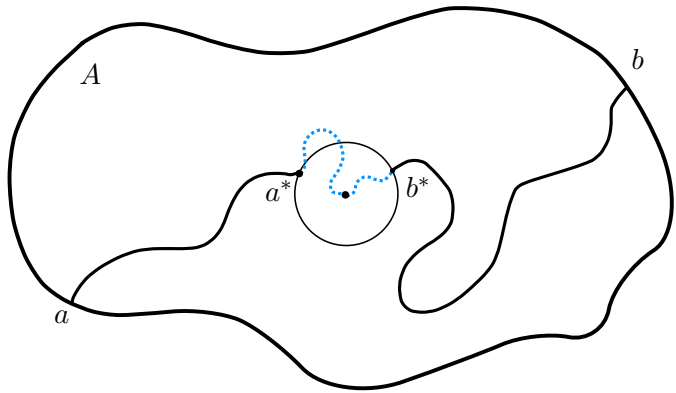
\caption{If conditionally on the LERW $\eta$ hitting $C_{r}$, the first and last visits to $\partial C_{r}$ (i.e. $a^{*},b^{*}$) are well-separated with positive probability, then the one-point estimate implies that the conditional probability of $\eta$ visiting $0$ is comparable to $r^{-3/4}$.}\label{fig:separation}
\end{figure}

\begin{thm}[Separation Lemma]  \label{seplemma.1}
There exists $c > 0$
such that if $(A,a,b)$ $ \in \whoknows_{2r}$
with $\Prob_{A,a,b}(I_r) >0$,  then
\begin{equation}  \label{jul14.1}
   \Prob_{A,a,b}\left\{S_r \geq c \mid \eta \in I_r\right\}
 \geq c . 
 \end{equation}
 \end{thm}

\begin{proof} See Section \ref{sepsec}.
 We actually prove 
 that there exists $c > 0$ such that 
 \[ \Prob_{A,a,b}\left\{|a_r - b_r|\geq c \, r \mid \eta \in I_r\right\}
 \geq c , \]
but \eqref{jul14.1} follows immediately from
this and  \eqref{sep25.11}.\end{proof}

\begin{cor}  \label{corollary.bastille1}
If $(A,a,b) \in \whoknows_{2r}$, then
\begin{equation}  \label{jul14.3}
  \Prob_{A,a,b}\{0 \in \eta \mid \eta \in I_r\} \asymp
   r^{-3/4}, 
   \end{equation}
and hence
\begin{equation}  \label{jul14.4}
       \Prob_{A,a,b}[I_r]
    \asymp r^{3/4} \,   \Prob_{A,a,b}\{0 \in \eta  \}.
    \end{equation}
\end{cor}

\begin{proof}  Since $(A_r,a_r,b_r) \in \soup_r$, we know that
$r_{A_r} \asymp r$, and hence from $\eqref{jul14.0}$,
on the event $\{\eta \in I_r\}$, we have 
 \[ \Prob_{A,a,b}\{0 \in \eta \mid (A_r,a_r^*,b_r^*)\} 
     \asymp \, r^{-3/4} \, [S_r^3 + O(r^{-u})]. \]
Combining this with \eqref{jul14.1} gives
\eqref{jul14.3}.  Since  
\[   \Prob_{A,a,b}\{ 0 \in \eta\} = 
 \Prob_{A,a,b}[I_r] \, \Prob_{A,a,b}\{0 \in \eta \mid \eta \in I_r\},\] 
 we see that \eqref{jul14.4} also
 follows.
\end{proof}

 The next lemma is an upper bound for the ``two-point Green's function'', that is, the 
 probability that LERW visits two points,  when the points
 are not too close together.  We follow with a  
 corollary for  points that are close together.
 Here we use $H_A(0,a)$ for the discrete Poisson kernel,
 that is, the probability that a simple random 
 walk starting at $0$ exits $A$ along the edge
 $a$. Since  $H_A(0,a) = G_A(0,a^*)/4$, 
 we can replace $H_A(0,a)$ with 
 $G_A(0,a^*)$ on the right-hand
 side of the estimate of Theorem~\ref{lemma.twopoint}. We state this 
 in terms of 
  $\hat P_{A,a,b}$ rather than the probability
 measure $\Prob_{A,a,b}$.

\begin{thm}[Two-point estimate]  \label{lemma.twopoint} There exists $c < \infty$
such that the following holds.  Suppose $1  \leq  s \leq r <
\infty$,  $(A,a,b)
\in \whoknows_r$, and $\zeta \in A$ with
$\dist(\zeta,\partial A) \geq s$ and $|\zeta| \geq r/4$.
  Then
\[ \hat P_{A,a,b}
\{ 0 , \zeta \in \eta\} \leq 
   \frac{c\, G_A(0,\zeta )
   \,  \left[H_A(0,a) \, H_A(\zeta,b)   + H_A(0,b) \, H_A(\zeta,a)
   \right]}{ r^{3/4} \, s^{3/4}}
       .\]
\end{thm}

\begin{proof}  See Section \ref{twopointsec}.  We actually
show the slightly stronger result that the $\hat P$-measure of paths in $\saws(A,a,b)$
that  visit $0$ first and then $\zeta$
is bounded above by a constant times
\[  \frac{ G_A(0,\zeta )\,
   H_A(0,a) \, H_A(\zeta,b)    
   }{ r^{3/4} \, s^{3/4}}.
       \]\end{proof}

\begin{cor}  There exists $c < \infty$
such that the following holds.  Suppose $(A,a,b)
\in \whoknows_r$, and $|\zeta| \leq r/4$. 
   Then,
\begin{equation}  \label{sep23.2}
   \Prob_{A,a,b}\{\zeta \in \eta\mid
0 \in \eta\} \leq c \, |\zeta|^{-3/4}.
\end{equation}
\end{cor}

\begin{proof}
Let $s = 2|\zeta| \leq r/2$. Corollary \ref{corollary.bastille1}
implies   that
\[  \Prob_{A,a,b}\{\eta \in I_s\} \asymp s^{3/4} \,
\Prob_{A,a,b}\{ 
0 \in \eta\}, \]
We also claim that
\[  \Prob_{A,a,b}\{0,\zeta \in \eta
 \mid \eta \in I_s\}  \leq c \, s^{-3/2}.\]
Indeed, we now show that there exists $c$
such that for all $(A',a',b') \in \soup_s$,
\begin{equation}  \label{eve.1}
\Prob_{A',a',b'}\{0,\zeta \in \eta
 \}  \leq c \, s^{-3/2}.
 \end{equation}
 To see this, we first note that $G_{A'}(0,\zeta)
  \asymp 1$ and 
\[ \Prob_{A',a',b'}\{0,\zeta \in \eta
 \}   = H_{\partial A'}( a',b')^{-1} \,
   \hat P_{A',a',b'}\{0,\zeta \in \eta\}\]
The Harnack inequality implies
 that $H_{A'}(0,a') \asymp H_{A'}(\zeta,a')$
 and 
 $H_{A'}(0,b')$ $ \asymp H_{A'}(\zeta,b')$.
 Hence \eqref{eve.1} will follow from
Theorem~\ref{lemma.twopoint} provided we show
 that
 \[   H_{A'}(0,a') \, H_{A'}(0,b')
   \leq c \, H_{\partial A'}(a',b').\]
For this, we will show that
the left-hand side of the last display is comparable to the $p$-measure
of walks from $a'$ to $b'$ staying in $A'$ that
 intersect $C_{s/2}$, which is obviously bounded above by $H_{\partial A'}(a', b')$.  
To finish the proof, split any such walk $\omega$
as $\omega = \omega_1 \oplus \omega_2$ where
$\omega_1$ is stopped at the first time that
the walk reaches a point in $C_{s/2}$.  Given $\omega_1$,
the measure of the set
 of choices for $\omega_2$ is $H_{A'}(w,b')$
where $w$ is the endpoint of $\omega_1$.  By
the Harnack inequality, we know that for
each such $w$,  $H_{A'}(w,b') \asymp H_{A'}(0,b')$.
The reversal of any random walk path starting
at $0$ stopped when it leaves  $A'$
at $a'$ can similarly be written as $\omega_1 \oplus
\omega_3$ where $\omega_1$ is as above and $\omega_3$
is a walk from $w$ to $0$.  For each $w$, the measure
of choices for $\omega_3$ is $G_{A'}(w,0) \asymp 1$
and hence the measure of the set of acceptable $\omega_1$
is comparable to $H_{A'}(0,a')$.
 \end{proof}

\subsection{Estimates for analytic domains}
In this section we discuss various estimates under the assumption that the domain we consider is analytic.
We first consider second moment bounds on the number of steps in a LERW. We will derive some consequences of the estimates stated in Section~\ref{sect:main-lerw}. The issue is that the estimates given there are not very sharp
near the boundary in general. Here we will show that the estimates
are good enough if the discrete sets are sufficiently ``nice''. 

For the remainder of this subsection we fix a  
  simply connected domain $D$ with 
  $0 \in D \subset \Disk $ with analytic boundary $\partial D$
  and two distinct boundary
points $a',b' \in \partial D$.   We emphasize that all constants in this subsection,
either explicit or implicit, are allowed to depend on  $D,a',b'$.
Let $f:D \rightarrow \Disk$ be the unique
conformal transformation with $f(0) = 0, f(a') = 1$.  Recall that the analyticity assumption means
 that $f$ extends to a conformal transformation of a neighborhood
of $\overline D$. In particular, $|f'|$ is uniformly bounded above and away from $0$ on $\overline{D}$.    

  For each
$n=1,2,\ldots$, let $A_n$ be the connected component 
containing the origin of the lattice set
\[    \{ z \in \Z^2: \Square_z \subset n  D\}, \]
and let \[D_n = n^{-1} \, D_{A_n}.\]  
(We are using a different notation from previous section here; we are now writing $D_n$ instead of $\check{D}=\check{D}(A,N)$.) Note that $D_n \subset D$ is a simply connected domain and that every point on $\partial D_n$ is at distance $O(n^{-1})$ from $\partial D$.
Therefore, since $|f'|$ is bounded above, there exists $c_1=c_1(D) < \infty$ such that for all $n$,
\[             \left( {1} - \frac{c_1}{n} \right)
 \, \Disk \subset f(D_n) \subset \Disk. \]
 Also, $(\diam A_n)/r_{A_n} \asymp 1$.   
 We let $a_n,b_n$ be points in $\partial_e A_n$ (considered as points
 in $\partial D_{A_n}$) that are closest to $na',nb'$.  (If there are
 ties for ``closest'' we can choose arbitrarily.)  We will write
 $\Prob_n$ for $\Prob_{A_n,a_n,b_n}$.  If $z \in A_n$,
 we write
 \[    \delta(z) = \delta_{A_n}(z) = \min\{|z-a_n|, |z-b_n|\}, \]
 and
 \[    d_z = d_{z,A_n} = \dist(z,\partial A_n).\]

 Our first goal is to use Theorem~\ref{lemma.twopoint} to prove the following two-point
 estimate.

\begin{prop} \label
{sept23.prop1} For every $(D,a',b')$ as above, with $\partial D$ analytic, there exists $c    < \infty$
such that if $z,w \in A_n$ with $z\neq w$ and $\delta(z) \leq \delta(w)$,
 then
\begin{equation}  \label{sep23.1}
  \Prob_n\{z,w \in \eta\} \leq  
  c \, \delta(z)^{-3/4} \, \left[\delta(w)^{-3/4} + |z-w|^{-3/4} \right]. 
  \end{equation}
  \end{prop}
  
The proof uses the following facts that we will not prove.  
These estimates  can be considered versions
of the well-known gambler's ruin estimate for random walk (see,
e.g., \cite[Theorem 5.7.1]{LL} for a uniform version)
 and strongly use the fact
that the boundary of $D$ is smooth and hence locally looks like a straight
line.
  
\begin{lemma}  \label{sep30.lemma1}
For every $(D,a',b')$ as above, with $\partial D$ analytic, there exists $c < \infty$ such that 
\[            c^{-1} \leq n^2 \, H_{ \partial A_n}(a_n,b_n)
 \leq c . \]
Moreover, if 
  $z \in A_n$,   $S_j$ is a simple random walk starting at $z$, 
and $\tau = \tau_{A_n } = \min\{j: S_j \not \in A_n \}$, then
\[  \Prob^z\left\{\diam(S[0,\tau]) \geq r \, \dist(z,\partial A_n) \right\}
  \leq \frac{c}{r} . \]
  \end{lemma}

 \begin{proof}[Proof of Proposition \ref{sept23.prop1}]
We first claim that without loss of generality we may
assume that  $d_z \geq  \delta(z)/20, d_w
 \geq \delta(w)/20$; in particular, 
 $d_z \asymp \delta(z), d_w \asymp \delta(w)$.  If this is not the case, we can add
 the disk of radius $\delta(z)/20$ about $z$ and the disk
 of radius $\delta(w)/20$ about $w$ to $A_n$.  Clearly
 this only increases $\hat P_{A_n,a,b}
 \{z,w \in \eta\} $ and it is not difficult to see that
 this increases $H_{\partial A_n}(a,b)$ by at most a universal
 multiplicative constant.  Hence, adding the disks
 would decrease
 $\Prob_{n}\{z,w \in \eta\}$ by at most
 a multiplicative constant.

 Combining the estimates in  Lemma \ref{sep30.lemma1}, we first claim that 
\begin{equation}  \label{dec23.3}
 H_{  A_n} (z,a_n)   \leq \frac{c d_z}{|z-a_n|^2} . 
 \end{equation}
To see this, think of the right-hand side as the product of two terms.  The 
probability starting at $z$ of getting   to distance $|z-a_n|/2$ without
leaving $A_n$ is bounded above by $c d_z/{|z-a_n|}$. 
Given that the walker succeeds in doing this,
 since $(\diam A_n)/r_{A_n} \asymp 1$ and $\partial D$ analytic, the probability of leaving $A_n$ at $a_n$ is  
$O(1/|z-a_n|)$. We can write \eqref{dec23.3} as 
\[  \frac{H_{  A_n} (z,a_n) }
   { d_z^{3/4}} \leq c \,   \frac{ d_z^{1/4}}{|z-a_n|^2} , \]
 and similarly for $  {H_{  A_n} (w,b_n) }/
   { d_w^{3/4}}$.
We also claim that 
\[    G_{A_n}(z,w) \leq c\, \frac{d_z \, d_w}
    {|z-w|^2} . \]
  To prove this, we can view
  the right-hand side as the product of three terms.  The probability
that a walk starting at $z$ moves distance $|z-w|/4$ without leaving
$A_n$ is comparable to $1 \wedge (d_z/|z-w|)$ and similarly for a random walk
starting at $w$.  Given that both of these happen, the expected number of visits to the other point
is   $O(1)$.  Combining the last two displayed inequalities,
 we get  
\begin{align}\label{lhs}
    \frac{  H_{  A_n} (z,a_n)\, H_{A_n}(w,b_n)  \, G_{A_n}(z,w) }
   { d_z^{3/4} \, d_w^{3/4} }
    \leq   \frac{c\, d_z^{5/4} \, d_w^{5/4} }{|z-a_n|^2 \, |w-b_n|^2
     \, |z-w|^2}.\end{align}
Note that the left-hand side of \eqref{lhs} corresponds to one of the two terms in the upper bound of Theorem~\ref{lemma.twopoint}; the other is obtained by interchanging $z$ and $w$. We will derive \eqref{sep23.1} from \eqref{lhs} using Theorem~\ref{lemma.twopoint}. There are several cases to consider.    
     
\medskip
   
  \textbf{Case 1.}  $|z-w|  \geq d_z/4$.
In this case, we will
bound the  left-hand side of \eqref{sep23.1} by 
$c\,\delta(z)^{-3/4} \, \delta(w)^{-3/4}$. 

  We claim that the right-hand side of \eqref{lhs} satisfies
  \[ 
  \frac{d_z^{5/4} \, d_w^{5/4} }{|z-a_n|^2 \, |w-b_n|^2
     \, |z-w|^2} \le \frac{c}{n^{2} \, \delta(z)^{3/4} \, \delta(w)^{3/4}}.
  \]
  
 To see this, we first note that the triangle inequality
 implies that  at least
 one of the following must hold: $|z-w| \geq |a_n-b_n|/2$,
 $|z-a_n| \geq |a_n-b_n|/4$, or $|w - b_n| \geq |a_n-b_n|/4$. 
 Also, we know that $|a_n - b_n| \geq cn$.
\begin{itemize}
\item  If $|z-w| \geq |a_n-b_n|/2$, then
 $|z-w| \geq c \, n$.  Since $\delta(z) \geq d_z$,
$\delta(w) \geq d_w$, and, by definition $|z-a_n| \ge \delta(z)$, $ |w-b_n| \ge \delta(w)$,
the claim holds.

\item If $|w - b_n| \geq |a_n-b_n|/4$, then $|w-b_n| \ge c n$. We also have $|z-a_n| \ge \delta(z) \asymp d_z$ and that $|z-w| \geq d_z/4$. 
\begin{align*}
 \frac{  d_z^{5/4} \, d_w^{5/4} }{|z-a_n|^2 \, |w-b_n|^2
     \, |z-w|^2}       \le &   \frac{ c\,d_z^{5/4} \,  d_w^{5/4}}
            {n^2 \, \delta(z)^2\, d_z^{2}}\\
             \le &   \frac{c\,
           d_w^{5/4}}
            {n^2 \, \, d_z^{3/4}\delta(w)^2}\\
        \leq &   \frac{c}{n^2 \, \delta(z)^{3/4} \,
            \delta(w)^{3/4}} .
    \end{align*}
\item In the same way, if  $|z-a_n|  \geq |a_n-b_n|/4$, we see that
\begin{align*}
\frac{  d_z^{5/4} \, d_w^{5/4} }{|z-a_n|^2 \, |w-b_n|^2
     \, |z-w|^2}      \le &  \frac{c\,
           d_w^{5/4}}
            {n^2 \, \, d_z^{3/4}\delta(w)^2}\\
        \le &   \frac{c}{n^2 \, \delta(z)^{3/4} \,
            \delta(w)^{3/4}} .
            \end{align*}
\end{itemize}
 
    By interchanging
 the role of $z$ and $w$ and using  
  Theorem~\ref{lemma.twopoint}  and the estimate
  $H_{\partial
  A_n}(a_n,b_n) \asymp n^{-2}$, we see that 
 \[ \Prob_n \{z,w \in \eta\}
   \leq  \frac c 
   {\delta(z) ^{3/4} \, \delta(w)^{3/4}}  .\]

\textbf{ Case 2:} $|z-w| \leq d_z/4$. 
Using  \eqref{sep23.2}, we see that
 \[     \Prob_{n} \{z,w \in \eta\} \leq
    c \, |z-w|^{-3/4} \,  \Prob_n\{z \in \eta\} \leq c \,
    |z-w|^{-3/4} \, d_z^{-3/4}, \]
and we use $\delta(z) \asymp d_z$.
 \end{proof}
 
We consider now the number of steps in a LERW running in an approximation of $D$. Let 
\[  T =  T_{n,D,a',b'}= \sum_{z \in A_n} 1\{z \in \eta \}. \]
More generally if 
 $\zeta \in \C$ and $r > 0$, let
\[   T(r;\zeta ) 
=  T_{n,D,a',b'}(r;\zeta ) = \sum_{| z- \zeta| \leq r} 1\{z \in \eta\} \]
be the number of steps inside a ball of radius $r$ about $\zeta$. Finally we define the associated maximal function
\[    \overline {T(r)} =  \overline {T_{n,D,a',b'}(r)} = \max\left\{T(r;\zeta): \zeta \in A_n \right\}\]
which is important for our main argument.
We will estimate the second moments of these random variables.

\begin{prop} \label{prop:maximal-estimate}
For every $(D,a',b')$ as above, with $\partial D$ analytic, there exists $c < \infty$ such that
for  every $0 < r \leq 1$ and $
\zeta \in A_n$,
\begin{equation}  \label{sep24.1}
    \E_n\left[T(rn;\zeta)^2\right] \leq  c \, (rn)^{13/4} \left( |\zeta-a_n|^{-3/4}  + |\zeta-b_n|^{-3/4} \right). 
    \end{equation}
\begin{equation}  \label{sep24.2}   \E_n\left[\overline {T(rn)} ^2\right] \leq c \, r^{5/4}
\, n^{5/2}.
    \end{equation}
    In particular,
  \[   \E_n\left[T^2\right] \leq c \, n^{5/2}.\]
\end{prop}

We will be using this lemma with $r \asymp 1/\log n$ in which case we get
\[        \E_n\left[\overline {T(rn)} ^2\right] \leq c \, n^{5/2} \,
\left(\log n \right)^{-5/4}.\]
 
\begin{proof} Note that
\[    \E_n\left[T(rn;\zeta)^2\right] =
  \sum_{|z-\zeta| \leq rn} \sum_{|w-\zeta| \leq rn} 
    \Prob_n\{z,w \in \eta\}.\] 
The first estimate \eqref{sep24.1} follows from \eqref{sep23.1}
and the easy estimates
\[    \sum_{z:|z-\zeta| \leq rn} (1+|z-a_n|)^{-3/4}  \leq c \, |\zeta-a_n|^{-3/4}
  \, (rn)^{2}, \]
  \[   \sum_{|z-\zeta| \leq rn} \sum_{|w-\zeta| \leq r} 
     (1+|z-a_n|)^{-3/4} \,  (1+|z-w|)^{-3/4} \leq c \, |\zeta-a_n|^{-3/4} \, (rn)^{13/4}.\]
which can be obtained by approximating by an integral.

To prove \eqref{sep24.2},
  let $m$ be the integer such that
$2^{m-1} <  rn \le 2^{m} $, and consider  
\[       L_{rn} =     \left\{ j 2^{m} + i k 2^m : j,k \in \Z \right\}
 \, \cap \left\{|z| \leq 2n\right\} . \]
Let
\[        K = \max_{\zeta \in L_{rn}} T(2^{m+1},\zeta),\]
and note that  
\[ \overline{T(rn)}  ^2 \leq K^2 \leq \sum_{ \zeta \in L_{rn}}  T(2^{m+1},\zeta)^2
 . \] 
 (Recall that $\diam A_n < 2n$.) Using \eqref{sep24.1},
we see that
\begin{align*}
\E_n[K^2] & \leq  \sum_{\zeta \in L_{rn}} \E_n\left[T(2^{m+1},\zeta)^2
\right]\\
& \leq   c \, (rn)^{13/4} \, \sum_{\zeta \in L_{rn}}
    [1+|\zeta|]^{-3/4} 
    \\
    & \leq   c \, (rn)^{13/4} \,  2^{-3m/4} \,
        [n2^{-m}]^{5/4} 
       \leq   c \, r^{5/4} \, n^{5/2}.
\end{align*}

\end{proof}

\begin{rem}We note that the estimate in the last lemma is really just noting that
\[    \int_{|z| \leq R}  \int_{|w | \leq R} \frac{dA(z) dA(w)}{|z|^{3/4} \,
   |z-w|^{3/4}} \asymp R^{5/4} \cdot R^{5/4} = R^{5/2}, \]
\[  \int_{|z| \leq R}  \int_{|w-z| \leq rR} \frac{dA(z)dA(w)}{|z|^{3/4} \,
   |z-w|^{3/4}} \asymp R^{5/4} (rR)^{5/4} = r^{5/4} \, R^{5/2}. \]
 
\end{rem}

 \subsection{Separation lemma: proof of Theorem~\ref{seplemma.1}}  \label{sepsec}
 In this section we will prove Theorem~\ref{seplemma.1}. As a start we state a lemma about simple
 random walk.  Suppose we
 start a random walk at $z \in \partial_i C_r$ and $s < r$.  Consider
 the first time that the random walk gets distance $s$
 from $z$.  Then it is easy to see (for example, by comparison
 with Brownian motion) that there exists $ c> 0$
 (independent
 of $z,s,r$) such that with probability at least
 $c$, the random walker stops  within distance
 $r - (s/3)$ of the origin.  Now suppose that $A \supset C_r$
 and we condition that the walk stays in $A$ before it reaches
 distance $s$.  If anything, this should push the random walker
 closer to the origin and hence there should be
 a uniform lower bound on  the probability
 of being within distance $r - (s/3)$.  The next lemma states
 that   this intuition is correct, and we can find a
 constant that is independent of $r,s,z,A$.  For a proof
 of the first statement, see \cite[Proposition 3.5]{Masson};
 the second is done similarly, and,
 in fact, is slightly easier.
 
 \begin{lemma} \label{lemma.masson}
  There exists $c >0$ such that the following
 holds. 
 \begin{itemize}
 \item  Suppose $s < r,$  
  $C_r \subset A$ and $z \in \partial_i C_r$.  Let
$S$ be a simple random walk starting at $z$ and let
\[   \tau = \tau_A = \min\{j: S_j \not \in A\}, \]
\[    \sigma  = \sigma_s = \min\{j: |S_j - S_0|
  \geq s\}.\]
Then,
\[   \Prob\{|S_\sigma| \leq  {r-(s/3)} \mid \sigma
   < \tau\} \geq c . \]
  \item
  Suppose $s < r$ and
  $\Z^2 \setminus C_r
   \subset A$ and $z \in \partial C_r$.  Let
$S$ be a simple random walk starting at $z$ and let
\[   \tau = \tau_A = \min\{j : S_j \not \in A\}, \]
\[    \sigma  = \sigma_s = \min\{j: |S_j - S_0|
  \geq s\}.\]
Then,
\[   \Prob\{|S_\sigma| \leq  {r+(s/3)} \mid \sigma
   < \tau\} \geq c . \]
   \end{itemize}
 \end{lemma}

If $(A,a,b) \in \soup_r$, we define
$e_a(A,a,b)$ to be the probability that a simple random
walk starting at $a^*$ reaches distance $|a^* - b^*|/3$
from $a^*$ without leaving $A$.  
We define
\[        e(A,a,b) = e_a(A,a,b) \, e_b(A,b,a).\]
 
 \begin{lemma} \label{sep28.lemma1}
  There exists $0 < c_1 < c_2 < \infty$
 such that the following holds.  Suppose $(A,a,b)
  \in \soup_r$.  Then
  \[      c_1 \, e(A,a,b) \leq H_{\partial A}(a,b)
    \leq c_2 \, e(A,a,b).\]
  Moreover,  the $p$-measure of the set of
   walks in $\paths_A(a,b)$ 
    of diameter
  less than $c_2 \, |a^* - b^*|$ is at
  least $c_1 \, H_{\partial A}(a,b)$.
  \end{lemma}
  
 The last assertion can be rephrased as saying that
 the probability that an excursion from $a$ to $b$ in $A$
 has diameter less than $c_2 \, |a^* - b^*|$  is at
 least $c_1$.
  
 \begin{proof} 
 Let $s = |a^* - b^*|/3$.
 It suffices to prove the result for $s$ sufficiently
 large (for small $s$ one can give a direct proof,  
 which we omit, by constructing specific paths).
 
  Let $h_a(z) = h_{A,a}(z)$ be
 the probability that a random walk starting
 at $a^*$ reaches distance $s$ from $a^*$
 without exiting $A$ and that the first point at distance $s$ that it hits is $z$.
 That is, using the notation of the previous lemma,
 \[   h_a(z) = \Prob^{a^*}\{\sigma_s < \tau_A,
  S(\sigma_s) = z\}. \]
 We define $h_b(w)$ similarly, and note
 that (for $s$ large enough)
 \[  4H_{\partial A}(a,b) = G_A(a^*,b^*) = \sum_{(z,w) \in U} h_a(z) \, h_b(w)
  \, G_A(z,w),\]
  where $U$ denotes the set of  $(z,w)$ with $s \leq
  |z-a^*| < s+1, s \leq |w-b^*| < s+1.$
  Using simple connectedness of $A$, it is not hard to verify that $G_A(z,w) \leq c_2$
  for all such $(z,w) \in U$; and if
   $(z,w) \in U \cap (C_{r - (s/3)}\times C_{r - (s/3)})$,
  then $G_A(z,w) \geq c_1 $.
  Note that
  \[   \sum_{(z,w) \in U} h_a(z) \, h_b(w)
   = e(A,a,b), \]
  and Lemma \ref{lemma.masson} implies that
  \[ \sum_{z,w \in C_{r - (s/3)} } h_a(z) \, h_b(w)
   \geq c \, e(A,a,b). \]
  Taken together these estimates give the first assertion.
For the second assertion, we consider the set
\[        V = \{\zeta \in C_r: |\zeta - a^*|
 \le 4 s\} \]
 and show that 
 \[
 H_{\partial V}(a,b) \ge c  H_{\partial A}(a,b).
 \]
Indeed, one can easily check that there is a constant $c' > 0$ such that $(z,w) \in U\cap  (C_{r - (s/3)} \times C_{r - (s/3)})$ implies
\[     G_V(z,w) \geq c'.\]
Consequently,
\begin{align*}
4H_{\partial V}(a,b)  = \sum_{(z,w) \in U}h_a(z)h_b(w) G_V(z,w) \ge c' \sum_{z,w \in C_{r - (s/3)}}h_a(z)h_b(w).
\end{align*}
But we have already shown that the last term is comparable to $H_{\partial A}(a,b)$.

  \end{proof}
  
  The next lemma is easy but important. It gives a lower
 bound on the probability that the LERW grown simultaneously from $a$ and $b$ reaches $C_{r-4|a-b|}$ and that at this time the distance
 of the endpoints has been increased by a factor of two.

\begin{lemma}   \label{dec29.lemma1}
 There exists $c > 0$ such that
if $(A,a,b) \in \soup_r$, and $s = |a-b|  \leq
r/10$, then
\[   \Prob_{A,a,b}\left[I_{r-4s} \cap \{|a_{r-4s} ^* -
 b_{r-4s} ^*|
    \geq  2 \, |a^* - b^*| \}\right] 
      \geq c . \]
\end{lemma}

\begin{proof} [Sketch of proof]  As in the previous
proof we first consider a random walk up to the
time that it gets distance $s/3$ from $a$ and $b$.
We consider $(z,w) \in U$ and consider random
walk paths from $z$ to $w$ whose loop-erasure
will stay in $C_{r-4s}$ and satisfy
$|a_{r-4s} ^* -
 b_{r-4s} ^*|
    \geq  2 \, |a^* - b^*|$.  We could give a
specific event, but we leave this to the reader.

\end{proof}

   In order to prove separation of the LERW, it is useful to consider
 an event defined in terms of the random walk from
 $a$ to $b$ in $A$.  Suppose $(A,a,b) \in \soup_{s}$
 with $3r/2 \leq s \leq 2r$.
 Consider the set of random walk paths
 \[  \omega = [\omega_0 , \omega_1,\ldots,
  \omega_{n-1}, \omega_{n}  ] \in \paths_A(a^*,b^*), \]
  satisfying the following conditions.
  \begin{itemize}
  \item $\omega \cap C_r \neq \eset$.
  \item  $\omega \cap C_r \subset \{x+iy: |y|\leq r/10\}$.
   \item  Let $j_-,j_+,k_-,k_+,l_-,l_+$ be the first and
   last visits to     $ C_r \cap \{\Re(z) < -   r/3
    \}   $, $ C_r \cap \{\Re(z) =0
    \}, C_r \cap \{\Re(z) >   r/3
   \} $, respectively.  Then
   \[   0 <
    j_- \leq j_+  \leq k_- \leq k_+ \leq l_- \leq l_+ 
    < n. \]
 Implicit in this condition
 is the fact that $\omega$   visits
 all three of $C_r \cap \{\Re(z) < -   r/3
    \}   $, $ C_r \cap \{\Re(z) =0
    \}, C_r \cap \{\Re(z) >   r/3
   \} $. Note that if $r \geq 3$, then we would
   also have $j_+ < k_-, k_+ < l_-$.
 \item
 \[   [\omega_0, \ldots,\omega_{j_+}]
   \cap [\omega_{k_-},\ldots,\omega_n] = \eset.\] 
   \end{itemize}
   
In particular, the walk $\omega$ enters 
  $ C_r \cap \{x+iy: |y|\leq r/10\}$ from the left
  and leaves from the right.  We let $J_r$ be
  the set of $\omega$ such that either $\omega$
  or the reversal of $\omega$ satisfies the
  conditions above.  An important fact that
  is easy to verify is the following:
 \begin{itemize}
 \item If $\omega
  \in J_r,$ then $LE(\omega) \in J_r$. 
  \end{itemize}  
  
With the aid of
Lemma \ref{lemma.masson} and
the invariance principle (by considering
an appropriate event for Brownian motion and approximating
by random walk), it is not hard to show the following.

\begin{lemma}  For every $\delta > 0$, there exists $ 
c_\delta > 0$, such that if $(A,a,b) \in \soup_{s}$ with
$3r/2 \leq s \leq 2r$ and $|a-b| \geq \delta \, r$, then
\begin{equation}  \label{sep28.1}
      \Prob_{A,a,b}(J_r) \geq c_\delta . 
      \end{equation}
\end{lemma}

We emphasize that the
constant $c_\delta$ depends strongly on $\delta$ and  
  goes to zero with $\delta$.  The separation
lemma is established by showing that there exists $c > 0$
such that if $(A,a,b) \in \mathcal{J}_{2r}$, then
\[      \Prob_{A,a,b}(J_r \mid I_r) \geq c  . \]
Here the constant is independent of $\delta$ but we are
only estimating a conditional probability.

\begin{figure}[t]
\centering
  \def\svgwidth{0.6\columnwidth}
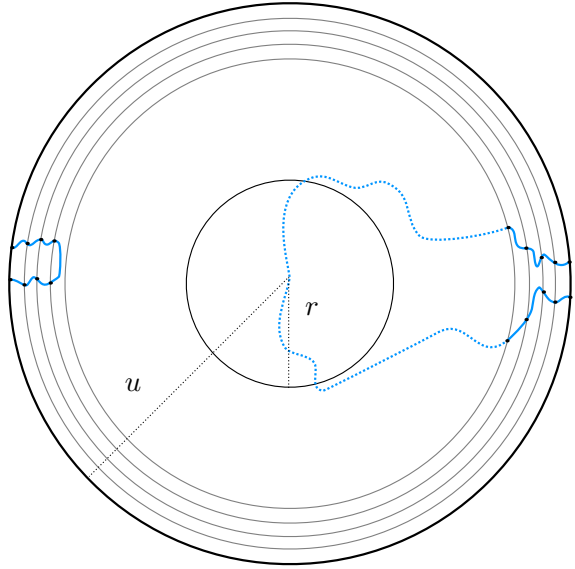
\caption{Sketch for the proof of the separation lemma for LERW. It shows two realizations of the LERW sub-Markov chain. In both cases the two paths start at distance $s=u2^{-m-1}$ and $u \asymp r$. On the left, the chain has ``hooked up'' and terminated before reaching $C_{r}$ and before separating by a factor $2$. On the right, the paths separate and eventually reach $C_{r}$ (and hit $0$). At each step, if the paths have not separated, there is a probability $c > 0$ that they hook up so the event that the paths reach $C_{r}$ and have \emph{not} separated by a factor of $2$ in $s m^{2}$ steps is $O(e^{-c m^{2}})$. The probability that the paths reach $C_{r}$ is at least a constant time $e^{-\beta m}$. Hence the paths are likely to separate in the conditional measure.}\label{fig:separation2}
\end{figure}

To prove the separation lemma, we
  start with $(A_{u},a_{u},b_{u}) \in \soup_{u}$ where
$u$ is a positive integer,  and
consider the (reverse time) subMarkov chain
\[ (A_{u},a_{u},b_{u}), (A_{u-1},a_{u-1},b_{u-1}),
  (A_{u-2},a_{u-2}, b_{u-2}), \ldots . \]
 induced by the measure $\Prob_{u} := \Prob_{A_{u},a_{u},b_{u}}$. 
It is a sub-Markov chain because the process is
killed at step $k$ on the event $I_{k+1} \setminus
I_k$. In words, unless the chain terminates, going from step $k$ to $k+1$ we grow both ends of the LERW one at a time until both paths reach distance $u-(k+1)$ from $0$. It stops at $(A_{1},a_1,b_1)$; the path is
still ``alive'' at that time if and only if $0 \in \eta$.
The path is growing at both the front and the back.
 The domain Markov property implies that
on the event $I_{s}$, the conditional distribution
of the remainder of the LERW is given by
$\Prob_{A_s,a_s,b_s}$. See Figure~\ref{fig:separation2}.
Let
\[       \sigma_{\delta}
 = \sigma_{\delta,u}  = \min\{k: |a_{u-k} - b_{u-k}|
      \geq \delta  \}. \]
 We claim that it suffices to show the following:
 \begin{itemize}
 \item  There exists $ 0 < \epsilon < 1/4$ such that
if $(A,a,b) \in \soup_{2r}$, then
\begin{equation}  \label{sep28.2}
   \Prob_{A,a,b}\left \{\sigma_{\epsilon r} \leq r/2
 \mid I_r\right\}  \geq \epsilon  . 
 \end{equation}
 That is, uniformly with probability at least $\ee$ the chain (conditioned on reaching $C_{r}$) separates by a factor of $\ee$ before reaching $C_{u-r/2}$. \end{itemize}
 Indeed, if we have this, since $J_r \subset I_r$,
 \eqref{sep28.1} and the domain Markov property
  imply that
\[
   \Prob_{A,a,b}\left\{ J_r \mid I_r \right\} \geq \epsilon \, c_\epsilon. 
\]
 
 In order to establish \eqref{sep28.2} we prove the following.
 \begin{itemize}
\item  There exists $c < \infty, \beta > 0$ such that
if $(A,a,b) \in \soup_u$ with $3r/2 \leq u \leq 2r$
and  $|a-b| \geq u 2^{-(m+1)}=:s,$ 
then
\begin{equation}  \label{sep28.3}
    \Prob_{A,a,b}\left \{\sigma_{2s} \geq 2s m^2  \, 
 \mid I_r\right\}  \le  c \, e^{-\beta m}. 
 \end{equation}
 \end{itemize}
 Indeed, choosing $\ee > 0$ so that $\sum_{ k = -\log \ee}^{\infty} k^2 2^{-k}< 1/4$, continued use of 
 \eqref{sep28.3} and the
 domain Markov property gives \eqref{sep28.2}.  To get \eqref{sep28.3}
 one proves two estimates,
\begin{equation}  \label{sep28.34}
   \Prob_{A,a,b} [I_r]  \geq   c_1 \, e^{- \beta m},
   \end{equation}
\begin{equation}\label{sep28.4}
    \Prob_{A,a,b}\left[\left \{\sigma_{2s} \geq 2s m^2 \,
 \right\}
 \cap  I_r\right]  \leq c_2 \, e^{-2  \beta m}, 
 \end{equation}
 for some $c_1,c_2$.
 
 For \eqref{sep28.34}, we can actually prove the stronger
 estimate   $\hat P_{A,a,b} [I_r]  \geq   c_1 \, e^{- \beta m}$.
To see this, we can either use estimates on the probability that random walk stays in a cone
or just repeated application of Lemma \ref{dec29.lemma1}.
For \eqref{sep28.4}, we use the final assertion of Lemma
\ref{sep28.lemma1} to see that if $|a-b| \leq u2^{-m}$, then
there is a positive probability that that random walk (and consequently the LERW) will
not hit $C_{u(1 - c'2^{-m})}$ for some constant $c'$. Instead the two ends of the LERW will ``hook up'' and the chain will terminate at this point.  By iterating this,
we can see that in $m^2$ attempts, except for an
event of probability $\exp\{-c m^2\} = o(e ^{-2\beta m})$, on at least one of the $m^{2}$ attempts
the random path will either fail to proceed another distance
$c'u2^{-m}$ inward or the endpoints will separate by a factor at least $2$.

\subsection{Two-point estimate: proof of Theorem~\ref{lemma.twopoint}}
   \label{twopointsec}\
    Given $(A,a,b)$ and $0,\zeta \in A$, let us write $\{a\rightarrow 0 \rightarrow \zeta 
  \rightarrow b \}$ for the event
 that the LERW $\eta$ from $a$ to $b$ first goes through $0$ and then later
 through $\zeta$.   Our goal is to show that
\begin{equation}  \label{sep29.5}
  \hat P_{A,a,b}\{a\rightarrow 0 \rightarrow \zeta 
  \rightarrow b\}  \leq \frac{c \, G(0,\zeta)\, H_A(0,a) \,
    H_A(\zeta,b)}{r^{3/4} \, s^{3/4} }. 
    \end{equation}
 Once we have this, we can conclude the proof of Theorem~\ref{lemma.twopoint} by interchanging the role of
 $a$ and $b$.  Before going through the details, let us quickly
 sketch the idea to show where the terms on the right-hand
come from, See Figure~\ref{fig:2ptfig}.  If $\eta$ is a SAW from $a$ to $b$ going
 through $0$ and then $\zeta$, we can write $\eta$ uniquely 
 as
 \[    \eta = \eta^- \oplus \eta^0 \oplus \tilde \eta
 \oplus \eta^\zeta \oplus \eta^+ \]
where $\eta^0$ is a
SAW starting and ending on $\partial C_{r/40}$ and otherwise staying in $C_{r/40}$, and going through $0$. Similarly, $\eta^\zeta$  is a SAW starting and ending on $\partial C_{s/40}(\zeta)$, staying in $C_{s/40}(\zeta)$, and going through $\zeta$.
By Theorem \ref{BLV}, the measure of possible choices for $\eta^0,\eta^\zeta$
are $O(r^{-3/4})$ and $O(s^{-3/4})$, respectively. Making this precise is what requires most of the work in this section. In particular, we will have to be able to compare several different loop-erased measures on walks in the discs around $0,\zeta$.

\begin{figure}[t]
\centering
  \def\svgwidth{0.7\columnwidth}
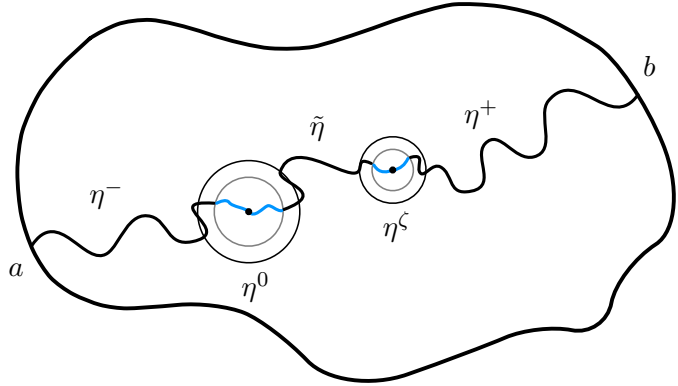
\caption{The proof of the two-point estimate \eqref{sep29.5}  splits the path into $5$ sub-paths. The contributions of $\eta^{-}, \tilde{\eta}, \eta^{+}$ are estimated by random walk quantities: $H_{A}(a,0)$, $G_{A}(0,\zeta)$ and $H_{A}(\zeta,b)$, respectively. Given $\eta^{-}, \tilde{\eta}, \eta^{+}$, the remaining $\eta^{0}$ and $\eta^{\zeta}$ are paths in $C_{r/40}$ and $C_{s/40}(\zeta)$ and going through $0$ and $\zeta$ respectively. We want to say that the LERW measures on these paths are $O(r^{-3/4})$ and $O(r^{-3/4})$. To get the latter estimates we need to estimate the escape probability of a random walk in a disc, given a LERW from the center of the disc.}\label{fig:2ptfig}
\end{figure}

We then have to multiply by the measure of possible choices
for $\eta^- , \tilde \eta,
  \eta^+ $ and this gives terms of  $H_A(0,a),
   G(0,\zeta), H_A(\zeta,b)$, respectively.  Our arguments
 do not use the fact that there are avoidance constraints
for the  paths $\eta^- , \tilde \eta,
  \eta^+ $, and this is why we only get an upper bound.
If $0,\zeta$ are in the interior, then our bound tends to be correct
up to a multiplicative constant, while if $0$ or $\zeta$ is
near the boundary, our estimate  is not sharp (but does
suffice for the needs in this paper).  

We start by focusing on the SAW $\eta^0$.

\begin{itemize}
\item{If $A$ is a finite simply connected subset of $\Z^2$
containing the origin
and $a \in \partial_e A$, we let $\saws_{A,0,a}$ denote the
set of SAWs starting at the origin, ending with
$a$, and otherwise staying in $A$.  We write
$\hat P_{A,0,a}$ for the usual loop-erased measure on such paths
(with total mass $H_A(0,a)$) and $\Prob_{A,0,a}$
for the normalized probability measure.}
\item{ We write
$\saws_{0,r}$ for the set of SAWs starting at the origin,
ending on $\partial C_r$ and otherwise in $C_r$.  In other
words,
\[\saws_{0,r} = \bigcup_{a \in \partial_e C_r}  \saws_{C_r,0,a}.\]
If $\eta \in  \saws_{0,r}$, we write $\eta^*$
for the terminal vertex, that is, the point in $\partial C_r$
at which the walk terminates.}
\end{itemize}
We will consider several related probability measures on
$\saws_{0,r}$, and eventually ``escape probabilities'' related to these measures, see Figure~\ref{fig:nonintersection}.
\begin{itemize}
\item{
The first corresponds to the usual LERW
in the disk $C_r$ stopped at the boundary:
  take 
  simple random walk starting at the origin, stop
  the walk when it reaches $\partial C_r$, and 
  then erase the loops.  We will write
$\pi_{r}$ for the induced probability measure
on paths, for which we know that \cite[(9.5)]{LL}
\begin{equation}  \label{sep29.1}
 \pi_{r}(\eta) = 4^{-|\eta|} \, F_\eta(C_r) . 
 \end{equation}
Here $|\eta|$ denotes the number of steps of $\eta$ and
$\log F_\eta(C_r)$ is the random walk loop measure of
loops in $C_r$ that intersect $\eta$.  Here we use
the ({\em rooted}) loop measure $m$   
defined by $m(l) = 
 |l|^{-1} \, 4^{-|l|}$  for each rooted loop
$l$ with $|l| > 0$ (we could also use the unrooted loop measure,
but to be definite we will choose the rooted measure).}
\item{More generally, if $C_r \subset A$, we write
$\pi_{r,A}$ for the probability measure 
on $\saws_{0,r}$ obtained by starting a
simple random walk at the origin, stopping when it reaches $\partial A$,
erasing loops, and then considering the resulting
SAW up to the first visit of $C_r$. 
 We
write $\pi_{r,s}$ for  $\pi_{r,C_s}$.  Under this
definition, $\pi_r = \pi_{r,r}$.   As in \eqref{sep29.1},
we can write 
\[  \pi_{r,A}(\eta)
= 4^{-|\eta|} \, F_\eta(A) \, e_A(\eta) , \]
where \[e_A(\eta) = H_{\partial (A \setminus \eta)}
 (\eta_*,\partial A) = \sum_{a \in \partial A} e_A(\eta;a)\] is the (escape) probability that a
simple random walk starting at 
$\eta^*$ reaches $\partial A$ without returning to $\eta$ and  
\[
e_A(\eta;a) = H_{\partial (A\setminus \eta)}(\eta^*,a).
\]
By definition, $e_{A}(\eta) = 1$
if $\eta^* \in \partial A$.}
\item{ Similarly, if
$a \in \partial A$, we write
$\pi_{r,A,a}(\eta)$ for the corresponding
probability law obtained as in the previous bullet if we replace
the simple random walk with a random walk $h$-process conditioned
to leave $A$ at $a$.
In this case,
\begin{equation}  \label{sep29.2}
  \pi_{r,A,a}(\eta)
= 4^{-|\eta|} \, F_\eta(A) \, e_A(\eta;a) \, H_A(0,a)^{-1}.
\end{equation}}
\end{itemize}
 
The measures $\pi_{r}$ and $\pi_{r,A}$
can be significantly different especially at the terminal
point of the walk.  However, as we show in the next lemma,
  the measures $\pi_{r,A}$ and $\pi_{r,A,a}$
are comparable to $ \pi_{r,2r}$  provided that $C_{2r} \subset A$.

\begin{lemma}   \label{comppi}
 There exist $0 < c_1 < c_2 < \infty$
such that if $C_{2r} \subset A$ and $a \in \partial A$,
then for all $\eta \in \saws_{0,r}$,
\[  c_1\,  \pi_{r,2r}(\eta) \leq
     \pi_{r,A,a}(\eta)  \leq c_2 \, 
     \pi_{r,2r}(\eta), \]
\[  c_1\,  \pi_{r,2r}(\eta) \leq
     \pi_{r,A}(\eta)  \leq c_2 \, 
     \pi_{r,2r}(\eta). \]
\end{lemma}

\begin{proof}  We will prove the first displayed
expression which will imply the second since
\[    \pi_{r,A}(\eta)
  = \sum_{a \in \partial A} H_A(0,a) \, 
   \pi_{r,A,a}(\eta).\]
Since each $ \pi_{r,A,a}$ is a probability
measure, it suffices to find functions
$v_r,q_r$ such that $\pi_{r,A,a}$ factorizes up to constants:
\begin{equation}
\label{july18.1}
        \pi_{r,A,a}(\eta) \asymp v_r(\eta) \,
 \, q_r(A),
 \end{equation}
where $v_r$ depends only on $\eta$ and $r$, 
while  $q_r$ depends only on $A$ and $r$.

Recalling that $F_{\{0\}}(A) = G_A(0,0)$
and using the fact that every loop that hits
$0$ intersects every $\eta \in \saws_{0,r}$,
we get that 
\[ \pi_{r,A,a}(\eta)
= G_A(0,0) \, 4^{-|\eta|}  \, F_\eta(A\setminus \{0\}) 
\,  e_A(\eta;a) \, H_A(0,a)^{-1} .\]  From this we see 
that it is enough to factorize $F_\eta(A\setminus \{0\})$ and $e_A(\eta;a)$.

We start by looking at $F_\eta(A\setminus \{0\})$. We first partition the loops in $\Z^2 \setminus
\{0\}$ that intersect $C_{r}$ 
into three sets:
\begin{itemize}
\item $L_r^0$:  loops that lie entirely in $C_{2r}\setminus
\{0\} $;
\item  $L_r^1$:  loops in $\Z^2 \setminus \{0\}$
 that do not lie entirely  in $C_{2r}$
and disconnect $0$ from $\partial C_r$;
\item  $L_r^2:$ loops in $\Z^2 \setminus \{0\}$
that do not lie entirely in $C_{2r}$
and do not
 disconnect $0$ from $\partial C_r$. 
 \end{itemize}
We then write 
\[   F_\eta(A \setminus \{0\})
  = \prod_{j=0}^2 \lambda_j(\eta;A), \]
  where
  \[
    \lambda_j(\eta,A) = 
    \exp \left\{  m\{\ell \in L_r^j: \ell \subset
  A\setminus \{0\}
  , \ell \cap \eta \neq \eset \} \right\}.\]
  
Clearly, $\lambda_0(\eta;A) = \lambda_0(\eta;C_{r})$
for all $A \supset C_r$, so it depends only on $r,\eta$.
If  $\eta \in \saws_{0,r}$ and
$\ell \in L_r^1$, then $\ell \cap \eta 
\neq \eset$.  Hence,
\[   \lambda_1(\eta;A) = \lambda_1(C_r;A) = \exp \left\{  m\{\ell \in L_r^1: \ell \subset
  A\setminus \{0\}
  , \ell \cap C_r \neq \eset \} \right\},\]
which depends only on $r$ and $A$.

In \cite[Lemma 11.3.3]{LL} it is proved that 
 exists $c < \infty$ such that for each $r$,
 $m(L_r^2)  \leq c $.
Indeed,  the proof gives a  stronger fact:
there exists $c < \infty$ such that for each $r$
and each positive integer $k$, the loop measure
of loops in $C_{(k+2)r}$ that do not lie entirely
 in $C_{(k+1)r}$ and do not disconnect $0$
 from $\partial C_{r}$ is $O(k^{-2})$. 
Since $\lambda_2(\eta;A) \leq \exp
 \left\{ m(L_r^2) \right\}$, this implies that
 $\lambda_2(\eta;A) \asymp 1$.

Combining these estimates, we see that
  for all
 $\eta \in \saws_{0,r}$,
 \[     \frac{ F_\eta(A \setminus \{0\})}
 {F_\eta(C_{2r}\setminus \{0\})}
  \asymp  \frac{\lambda_1(C_r;A)}
  {\lambda_1(C_r; C_{2r})}.\]
Note that the right-hand side 
depends only on $r,A$. This gives the desired factorization of $F_\eta(A \setminus \{0\})$. 

It remains to consider $e_A(\eta;a)$. Using the Harnack inequality and
Lemma \ref{lemma.masson}, we can see that
$  e_{C_{3r/2}}(\eta) \asymp e_{C_{2r}}(\eta)$
and 
for every $A \supset C_{2r}$ and $a \in \partial A$,
\[    
e_A(\eta;a)   \asymp e_{C_{3r/2}}(\eta) \, H_A(0,a).\]
Combining all of this, we see that
$  \pi_{r,A,a}(\eta)$ is comparable to
\[\left[ 4^{-|\eta|} \,e_{C_{3r/2}}(\eta)
\, F_\eta(C_{2r}\setminus \{0\})
\right] \;\; \left[ G_A(0,0) \,\lambda_1(C_r;A)  
 \, \lambda_1(C_r; C_{2r} )^{-1}\right] .\]
This gives \eqref{july18.1}.\end{proof}

\begin{figure}[t]
\centering
  \def\svgwidth{0.8\columnwidth}
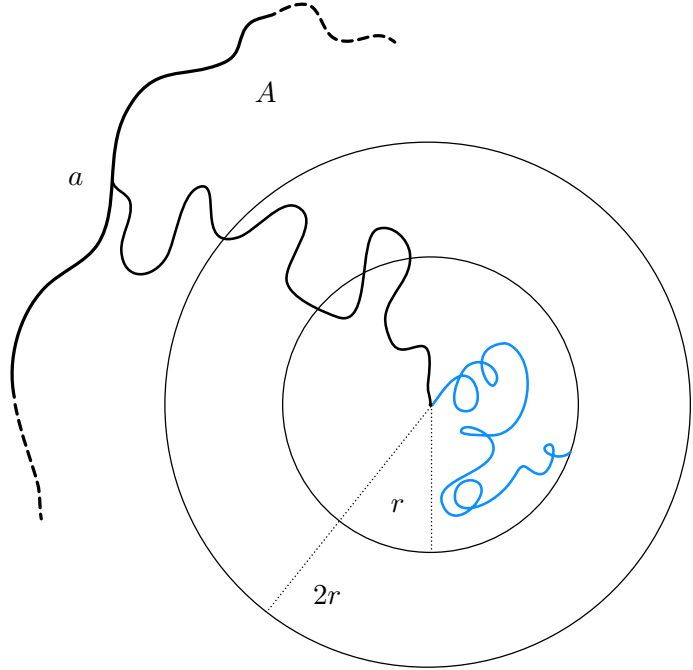
\caption{Given a LERW $\eta$ from $0$, we estimate the conditional probability that a random walk from $0$ in $C_{r}$ escapes to $\partial C_{r}$ without returning to $\eta$. We need to compare three different ``radial'' distributions on $\eta$. Theorem~\ref{BLV} gives the case when $\eta$ is LERW from $0$ to $\partial C_{r}$. Proposition~\ref{prop.nonintersection} compares this with the cases of LERW from $0$ to $\partial C_{2r}$ (stopped at $\partial C_{r}$) and LERW from $0$ to a fixed $a \in \partial A$ (stopped at $\partial C_{r}$).}\label{fig:nonintersection}
\end{figure}

Given $\eta \in \saws_{0,r}$, we let $h_r(\eta)$
denote the (conditional) non-intersection probability that a simple random walk
starting at the origin reaches $\partial C_{r}$ without
returning to $\eta$.  An immediate
consequence of Theorem~\ref{BLV}, specifically the up-to-constants version stated in \eqref{jul14.0}, is the following.

\begin{prop}  \label{jul18.prop1}
  There exist $0 < c_1 < c_2
 < \infty$ such that
 \[  c_1 \, r^{-3/4}
  \leq   \E_{\pi_{r}}[h_r(\eta)] \leq c_2
  \, r^{-3/4} . 
\]
\end{prop}

We will need the corresponding upper bound for
the other measures.  (The lower bound also
holds but we will not need this.)

\begin{prop}\label{prop.nonintersection}  There exists $c
 < \infty$ such that if $A \in \whoknows_{2r}$ and
 $a \in \partial_e A$, then 
\begin{equation}  \label{sep29.7} 
   \E_{\pi_{r,A,a}}[h_r(\eta)] \leq c
  \, r^{-3/4} . 
\end{equation}
\end{prop}

\begin{proof}
By Lemma \ref{comppi}, it suffices
to show that
 \[ 
    \E_{\pi_{r,2r}}[h_r(\eta)] \leq c
  \, r^{-3/4} . 
\] 
We fix an $\epsilon > 0$ such that the following
holds.  
\begin{itemize}
\item Suppose that $r \geq 1$ and
 $S_j$ is a simple random walk
starting at $z$ with $|z|  \leq \epsilon r$
and let $T = T_{r/4}$ be the first $j$ with
$|S_j| \geq r/4$.  Then the probability that
$S[0,T-1]$ disconnects $0$ from $\partial C_{r/4}$
is at least $.99$.  (By disconnection we mean
that if $\tilde S$ is another simple random walk
starting at the origin independent of $S$, then
the probability that $\tilde S$ visits $S[0,T-1]$
before reaching $\partial C_{r/4}$ is one.  By definition,
if $0 \in S[0,T-1]$, then $S[0,T-1]$ disconnects.)
\end{itemize}
 To show that such an $\epsilon$ exists,
 we first find an $\epsilon, r_0$ such that this  
holds  for  $r \geq r_0$ by   
  the invariance principle.  
Once we have this we can prove it for all $r$ by
choosing perhaps a smaller $\epsilon$ so that
$\epsilon \, r_0 \leq 1/2$.  In this case if
$r < r_0$, then $|z| \leq \epsilon r$ implies
that $z = 0$.   

Let $\eta \in \saws_{0,r}$ be chosen from
  the distribution $\pi_{r,2r}$.
Let $\omega$ denote a random walk path started
uniformly on $\{\pm 1, \pm i\}$ and stopped
when in reaches $\partial C_r$ and we write $\Prob_\omega$ for the probability law of $\omega$.  We write $\omega_*$
for the terminal point of $\omega$.  Note that
\[     h_r(\eta) = \Prob_\omega\{\eta \cap \omega
 = \eset \}. \]  
We let 
\[     h_r^*(\eta) = \Prob_\omega\{\eta \cap \omega
 = \eset ; \dist(\omega_*,\eta) \geq \epsilon \, r\}. \] 
 The definition of $h_r^*$ depends on $\epsilon$, but
 since we have fixed $\epsilon$ we will not include
 it in the notation. We claim the following.
 \begin{itemize}
 \item  There exists $\delta > 0$ such that
\begin{equation}  \label{oct1.1}
\E_{\pi_{2r,2r}}[h_{2r}] 
    \geq \delta\, \E_{\pi_{r,2r}}[h_r^*].
    \end{equation}
 \end{itemize}
 To see this, we first note that in the measure $\pi_{2r,2r}$
 the conditional distribution of the remainder of the path
 given $\eta$, the SAW up to the first visit to $\partial C_r$, can be
 obtained by starting a random walk at the endpoint $\eta$
  conditioned to reach $\partial C_{2r}$ without returning to $\eta$
  and then erasing loops.    Using the Separation Lemma (Lemma \ref{lemma.masson}), 
  we can see that in this conditioned distribution  
  there is a positive probability $\rho$ that the random walk (and hence also its loop-erasure) stays in
  $C_{2r} \setminus C_{r - (r\epsilon/5)}$ and that its argument
  does not vary by more than $\epsilon/10$.  We get similar
  estimates for the extension of the random walk $\omega$ to
  $\partial C_{2r}$.
  
 From \eqref{oct1.1} and Proposition \ref{jul18.prop1}, 
 we see that there exists $c < \infty$
 such that for all $s \geq 2r$,
 \begin{equation}  \label{jul12.1}
    \E_{\pi_{r,s}}[h_r^*] \leq c \, r^{-3/4}.
    \end{equation}
    
For each nonnegative integer $k$, we let $\eta^k,
\omega^k$ be the initial segments of these
paths stopped at the first visit to $\partial C_{r/2^k}$.
We define the events
\[   U_k = \left\{
\eta^k \cap \omega^k = \eset\right\}, \;\;\;\;
 V_k =\left\{\dist(\eta^k,\omega^k_*) \geq  \epsilon \, 2^{-k}
 \, r\right\} .\]
Here $\epsilon$ is as defined above. Using \eqref{jul12.1}, we see that
\begin{equation}\label{jan6.1}
  \Prob \left[U_{k}
            \cap V_k \right] \leq c \, (r/2^k)^{-3/4},\end{equation}
            where we now write $\Prob$ for the coupling where $\eta$ is distributed according to $\pi_{r,2r}$ and $\omega$ is an independent simple random walk.
  We want to prove that $\Prob  \left[ U_0 \right] \le c r^{-3/4}$. Note that by
the definition of $\epsilon$,
\[ \Prob \left[U_{k-1} \mid  (\eta^k, \omega^k) \right]
    \leq 1_{U_k} \, [1_{V_k} + (.01) \, 1_{V_k^c}],\]
    so that
    \[
    \Prob \left[U_{k-1} \cap V_k^c \mid (\eta^k, \omega^k) \right] \le (.01)1_{U_k}.
    \]
 By iterating this and recalling that $U_k$ are increasing events in $k$, we see that
 \[  \Prob \left[U_{0} \cap V_1^c \cap  \cdots \cap  V_{k-1}^c \mid  (\eta^k, \omega^k) \right]
     \leq  (.01)^{k-1} \, 1_{U_k}. \]
 Hence 
 \[ \Prob\left[U_{0} \cap V_1^c \cap  V_2^c \cdots\cap   V_{k-1}^c
  \cap V_k \right]   \leq (.01)^{k-1} \, \Prob \left[U_k \cap V_k \right]. \]
We can write
\[ U_0 \subset \left[U_0 \cap V_1^c \cap  V_2^c \cdots\cap   V_{k'}^c \right]
   \cup \left[\bigcup_{k=1}^{k'}
        (U_{k} \cap V_1^c \cap  V_2^c \cdots\cap   V_{k-1}^c
  \cap V_k) \right],\]
  where $k' = k'_r$ is defined to be the smallest integer $k$
  such that $(.01)^k \leq r^{-3/4}.$  We
 therefore get, using \eqref{jan6.1},
 \[  \Prob \left[U_0 \right] \leq r^{-3/4}
     + \sum_{k=1}^{k'}
          (.01)^{k-1} \, \Prob \left[U_{k}
            \cap V_k \right] \le c \, r^{-3/4}.\]
      and the lemma follows.\end{proof}

We are now ready to establish the two-point estimate \eqref{sep29.5}. 
Let $r' = r/40, s' = s/40$.
Let $\Gamma^*$ denote the set of
nearest neighbor paths
$\omega$  in $\paths_A(a,b)$
  that visit both
$0$ and $\zeta$ and such that the last visit
to $\zeta$ occurs after the last visit
to $0$.  Each $\omega
\in \Gamma^*$  has a unique decomposition 
\begin{equation}  \label{sep29.6}
 \omega = [\omega^1]^R \oplus \tilde{\omega}
\oplus \omega^2,
\end{equation}
where:
\begin{itemize}
\item $\omega^1$ is a nearest neighbor path
starting at $0$ leaving $A$ at $a$.
\item  $\omega^2$ is a nearest neighbor
path starting at $\zeta$ in $A \setminus \{0\}$
leaving $A \setminus \{0\}$ at $b$.
\item  $\tilde \omega$ is a nearest
neighbor path starting at $0$, ending
at $\zeta$, and otherwise staying
in $A \setminus \{0,\zeta\}$.
\end{itemize}
Let $\Gamma$ be the set of $\omega
\in \Gamma^*$ such that in the decomposition
above, 
\[    \omega^2 \cap LE(\omega^1) = \eset, \]
\[    (\tilde \omega)^{oo} \cap [LE(\omega^1)
  \cup LE(\omega^2)] =  \eset.\]
  Here $(\tilde \omega)^{oo}$ is $\tilde \omega$
  with the initial and terminal vertices
 removed.  If $\omega \in \Gamma$, we
define the SAW
\[ \eta = [LE(\omega^1)]^R
\oplus LE(\tilde \omega) \oplus LE(\omega^2)
  .\] 
Note that $\eta \in \saws_A(a,b)$  and $\eta$ visits
$0$ before visiting $\zeta$.  Moreover, for any such
$\eta$, 
  the
measure of the set of $\omega$ such that
$\eta$ is produced is  $4^{-|\eta|} \, F_\eta(A)$,
that is, we get the usual LERW measure.  In particular,
we can see that $ \hat P_{A,a,b}\{a\rightarrow 0 \rightarrow \zeta 
  \rightarrow b\}$ equals the measure of $\Gamma$.

To give an upper bound on the measure of $\Gamma$,
we refine the decomposition \eqref{sep29.6}
 by writing
\[  \tilde \omega =  \tilde \omega^1
 \oplus \tilde \omega' \oplus [\tilde \omega^2]^R
   , \]
 where 
 \begin{itemize}
 \item $\tilde \omega^1$ is a path starting
 at $0$ stopped when it reaches $\partial C_{r'}$.
 \item $\tilde \omega^2$ is a path starting
 at $\zeta$ stopped when it reaches $\partial C_{s'}(\zeta)$.
 \item $ \tilde \omega'$ is a path starting
 at the terminal point of $\tilde \omega^1$
 and ending at the terminal point of $\tilde \omega^2$.
 \end{itemize}
 We let $\Gamma'$ be the set of  paths $\omega \in \Gamma^*$ such
 that
 \begin{itemize}
 \item $(\tilde \omega^1)^o \cap LE(\omega^1)
  = \eset$, where $(\tilde \omega^1)^o$ denotes
  $\tilde \omega^1$ with the initial vertex
  removed.    
 \item $(\tilde \omega^2)^o \cap LE(\omega^2)
  = \eset$,
 \end{itemize}
 Note that $\Gamma \subset \Gamma'$. 
To estimate the measure of $\Gamma'$ we see that
\begin{itemize}
\item  The measure of possible $\omega^1$
is $H_A(0,a)$.
\item  
The measure of possible $\omega^2$
is $H_A(\zeta,b)$.
\item  Using Proposition~\ref{prop.nonintersection}, we see
that the probability that $(\tilde \omega^1)^o$
avoids $LE(\omega^1)$ is $O(r^{-3/4})$.
\item  Similarly, the probability that $(\tilde \omega^2)^0$
avoids $LE(\omega^2)$ is $O(s^{-3/4})$.
\item  Given $\tilde \omega^1,\tilde \omega^2$,
with terminal vertices $z,w$, respectively
the measure of paths in $A$ starting at $z$ and
ending at $w$ is $G_A(z,w)$.  Using the discrete
Harnack
inequality we see that this is comparable to
$G_A(0,\zeta)$.
\end{itemize}
By combining these bounds, the proof of Theorem~\ref{lemma.twopoint} is complete. \qed
\subsection{Estimates of bottleneck events}

We will need an estimate that shows that the LERW path does not have too many ``bottlenecks''; that is, that it is
unlikely for LERW to get near a point, then
get far away, and then subsequently get even
closer.

\begin{prop}  \label{dec23.1}
 There exist $c < \infty$ such that the following holds.
 Suppose $0 < r < R$ and $(A,a,b) \in \whoknows_r$ with
 $|a^*| < r$. Let $E'$ denote the set of  
  $\eta = [\eta_0,\ldots,\eta_n]   \in \saws_A(a,b)
$ such that  there exists $ 0 < j
   < k < n$, with $|\eta_j| \geq R, |\eta_k|
 \leq r.$  Then
 \begin{equation}\label{mar9.1}  \Prob_{A,a,b}\left\{ E' \right\} \leq   
 c \, (r/R)  .\end{equation}
 \end{prop}
  Proposition~\ref{dec23.1} is 
 an immediate corollary of Lemma~\ref{bottleneck-lemma} which
 is the corresponding statement for random walk excursions. 
 \begin{rem}The proof of Lemma~\ref{bottleneck-lemma} (with an obvious modification) yields an estimate similar to \eqref{mar9.1} for the event that the path \emph{twice} goes from radius $r$ to $R$ and returns to radius $r$. The only difference is that the probability of this event is $O((r/R)^2)$. This ``$6$-arm'' estimate leads to a sufficient regularity estimate for LERW which can be used to prove convergence of the (chordal) LERW path to the SLE$_2$ path parametrized by capacity from the coupling of Section~\ref{sect:coupling}, see \cite{LV_lerw_chordal_note}. 
 \end{rem}

By analogy  
 with SLE$_2$, we conjecture that this lemma  
 can be strengthened so that $(r/R)$ is replaced by
 $(r/R)^{3/2}$ where $3/2 = (8/\kappa-1)/2$.
 We have not proved the stronger result, but this
 lemma suffices for our purposes.   

 \begin{lemma}\label{bottleneck-lemma}
 There exist $c < \infty$ such that the following holds.
 Suppose $0 < r < R$ and $(A,a,b) \in \whoknows_r$ with
 $|a^*| < r$. Let $E$ denote the set of 
  $\omega = [\omega_0,\ldots,\omega_n]
   \in \paths_A(a,b)$ such that there exists $ 0 < j
   < k < n$, with $|\omega_j| \geq R, |\omega_k|
 \leq r.$  Then
 \[  \Prob ( E ) \leq   
 c \, (r/R) \, H_{\partial A}(a,b) .\]
 
 \end{lemma}
 
 Proposition~\ref{dec23.1} follows from this lemma since
 \[  \{\omega\in \paths_A(a,b): LE(\omega) \in E'\}
   \subset E.\]
 Since it is possible for $\omega \in E$ but $LE(\omega)
 \not\in E$, we can see why Lemma \ref{dec23.1}
 might not be a sharp estimate.
 
 \begin{proof}
 
Let $S$ denote a simple random walk starting at $a^*$ and
let
\[   \rho = \min\{j: |S_j - a^*|
  \geq r/2\},\;\;\;\;  \sigma = \min\{j: |S_j| \geq R\},\]
   \[ \tau_{2r} = \min\{k \geq \sigma:
     |S_j| \leq 2r\}, \;\;\;\;  \tau_r = \min\{k \geq \sigma: |S_j| < r \},
\]     \[
   T = \min\{n: S_n \not\in A\}. \] 
Then by the strong Markov property,
\[   \Prob (E) = \sum_{w \in \partial_i C_r}
      \Prob \{\tau_r < T, S(\tau_r) = w\}\,   H_{A}(w,b).\]
Note that  
\begin{multline*}   \Prob \{\tau_r < T, S(\tau_r) = w\} \\
 \leq \Prob \{ \rho < T \}
 \, \Prob \{ \tau_{2r} < T \mid  \rho < T \}
 \,  \Prob \{ S(\tau_r) = w \mid \tau_{2r} < T ,\rho < T \}.\end{multline*}
 Using the discrete Beurling estimate 
 (see, e.g., \cite[Theorem 6.8.1]{LL}), we see that
 \begin{equation}\label{beurling-twice}    \Prob \{ \tau_{2r} < T \mid  \rho < T \} \leq
   c \, (r/R). \end{equation}
 Indeed, we get a factor comparable to  $\sqrt{r/R}$ as an upper
 bound for the probability  to go from $\partial C_{2r}$
 to $\partial C_R$ staying in $A$, and we get
 another such factor for the probability of 
 returning from $\partial C_R$ to $\partial C_{2r}$ without exiting $A$. Using a standard estimate for the Poisson kernel in $\Z^2
 \setminus C_r$ (see, e.g., \cite[Lemma 6.3.7]{LL}),  we see that
 for each $w \in \partial_i C_r$,
 \[ \Prob \{ S(\tau_r) = w \mid \tau_{2r} < T ,\rho < T \}
  \leq c \, r^{-1}.\]
 Combining these estimates, we get 
 \[ \Prob(E) \leq  \frac{c}{R} \, \Prob^{a^*}
    \{\rho < T \} \, \sum_{w \in \partial_i C_r}
        H_A(w,b).\]
   Hence, to prove the lemma it suffices to prove
 that
 \[  H_{\partial A}(a,b) \geq \frac{c}{r} \, 
  \Prob 
    \{\rho < T \} \, \sum_{w \in \partial_i C_r}
        H_A(w,b).\]
 Note that if $w \in \partial_i C_r$ then either
 $w \in \partial C_{r-1}$ or $w$ has a nearest
 neighbor in $\partial C_{r-1}$.  Using this 
 we can see that
\[ \sum_{w \in \partial_i C_r}
        H_A(w,b) \leq c \, \sum_{w \in \partial C_{r-1}}
        H_A(w,b) .\]

  Using Lemma \ref{lemma.masson}, the strong Markov
  property,  and the Harnack inequality,
  we have
     \[  H_{\partial A}(a,b) \geq c\, 
  \Prob 
    \{\rho < T \} \, H_A(0,b).\]
  But using the estimate for the Poisson kernel \cite[Lemma 6.3.7]{LL} again we see that
  \[   H_A(0,b) = \sum_{w \in \partial C_{r-1}}
       H_{C_{r-1}}(0,w) \, H_A(w,b)
         \asymp r^{-1} 
         \sum_{w \in \partial C_{r-1}}
       H_A(w,b).\]
  Combining these estimates completes the proof. \end{proof}

 \section{Estimates about the metric}  \label{metricsec}

Here we collect  some  facts about continuity of the
SLE and LERW measures with respect to the Prokhorov 
metric.  We will not try to give optimal bounds.  We fix
a (bounded) analytic, simply connected domain $D$ containing
the origin with distinct boundary
points $a',b'$. We allow constants to depend on $D,a',b'$.  
Let $f: \Disk \rightarrow D$ be 
the unique conformal transformation with $f(0) = 0,
f'(0) > 0$. Since $D$ is analytic, $f$ extends to
a conformal transformation of $(1+\epsilon) \Disk$
for some $\epsilon > 0$.  In particular, there exists
$K = K_D < \infty$ such that
\[     \frac 1 K \leq |f'(z)| \leq K , \;\;\;
   |z| \leq 1 . \]
 Let $\check D = 
\check D_N$ be the lattice approximation of $D$ as before,
and let $\check f = \check f_N$ be the corresponding
map from $\Disk$ to $\check D$.  Since $\check D$
has a Jordan boundary, $\check f$ extends to a homeomorphism
of $\overline \Disk$.  Let $\psi = f \circ \check f^{-1}$
which is a conformal transformation of $\check D$ onto
$D$ with $\psi(0) =0, \psi'(0) > 0$.   Note that $\psi$
extends to a homeomorphism of the closures.

\begin{lemma}  \label{mar5.lemma1}
There exists $c < \infty$, such that
for all $z \in \check D$, 
\[    |\psi(z) - z| \leq c \, \frac{\log N}{N}.\]
Moreover, if $  \dist(z, \partial D) \geq  c/N$, then
\[   |\psi'(z)-1| \leq \frac{c}{N \,\dist(z, \partial D)}.\]
\end{lemma}

\begin{proof}  
Let $U = f^{-1}(\check D) \subset \DD$
and   $g:  f^{-1} \circ \psi \circ 
f =
\check f^{-1} 
\circ f$ which is the unique conformal transformation of
$U$ onto $\Disk$ with $g(0) = 0, \, g'(0) > 0$.
By considering $z = f(w)$ and using the fact
that $|f'| $ and $1/|f'|$ are uniformly bounded,
we see that 
\[  \max_{z \in \check D} |\psi(z) - z|
 \leq K \,  \max_{z \in \check D} |f^{-1}(\psi(z)) - f^{-1}(z)|
  =K \, \max_{w \in   U} |g(w) - w|. \]
 
  Since $\dist(\partial D,
\partial \check D) \leq \sqrt 2/N$, we have that $U$
contains the disk $(1-r) \Disk$ where $r =K  \sqrt 2/N$.
Let  $q(z) = g(z)/z$. The normalization of $g$ implies that we can choose a branch of $\log q(z)$ which is analytic in $U$ and $\Im \log q(0) = 0$. 
Using the Schwarz lemma, we have
\[        1 \leq \frac{|g(z)|}{|z|} \leq \,  \frac{1}{1-r},
\;\;\;\; z \in U.\]
Moreover, since $g'(0)=|g'(0)|$,
\[
1\le g'(0) \le \frac{1}{1-r}.
\]
It follows that if $L(z) = \Re \log q(z)
  = \log |g(z)| - \log |z|$, then
\[          |L(z)| \leq -\log (1-r)  \leq \frac c N.\]
Note that $L(z)$ is a positive harmonic funcion in $U$. Using the last estimate we can use the Poisson kernel to see that 
\[    |\nabla L(z) | \leq \frac{c}{N \, \dist(z, \partial U)}
       \leq \frac{c'}{N \, [1-|z|]}, \;\;\;\;
         |z|  \leq 1 - \frac{2K}{N}.\]
Since $L(z) = \Re \log q(z)$, this implies that the same bound holds for $|[\log q(z)]'|$.  Using
$\log q(0) = O(r)$, we can integrate to see 
that for  $|z| \leq  1 - 2K/N$, 
\[   |g(z) - z| \leq \frac{c \, |\log (1 - |z|)|}{N}, \]
\[   |g'(z) - 1| \leq \frac{c \, |\log (1 - |z|)|}{N }.\]
From this we see that there exists $c > 0$ such that for
$z \in D$ with $\Delta_z:= 1/\dist(z, \partial D)  \leq c
\, N$,
\[    |\psi(z) - z| \leq \frac{c \,(1 +
  \log \Delta_z) }{N}, \]
  \[   |\psi'(z) - 1| \leq 
   \frac{c \, (1 +
  \log \Delta_z) }{N}.\]
  
Up to this point, we have not used the special properties
of $\check D$ as a domain formed using the square lattice. By doing this
we get that for \emph{all} $z \in \check{D}$,
\[  |\psi(z) - z| \leq 
  \frac{c \, 
  \log N  }{N}. \]

\end{proof}

\begin{lemma} There exists $c < \infty$ such that if $a,b, a', b' \in \partial D$ and 
  if $|a-a'| \leq \delta \leq  |a-b|/3$
and $|b - b'| \leq \delta \leq |a-b|/3$, then there
exists a   conformal
transformation $F:D \rightarrow D$ with $F(a') = a, \, F(b') = b$ and such that
\begin{equation}  \label{nov1.1}
 |F(z) - z| + |F'(z) - 1|
  \leq c \, \delta , 
  \end{equation}
for all $z$.
\end{lemma}

\begin{proof}  If $D = \Disk$, $F$ is a M\"obius
transformation which
can be given explicitly and one readily checks the stated estimates in this case.  For other analytic
$D$ we write $F = f \circ M \circ f^{-1}$
where $M$ is an appropriate M\"obius transformation and use that  $|f'|$ is bounded above and away from $0$ in $\overline{D}$.
We omit the details.   
\end{proof}
 
\begin{cor}  \label{mar5.cor1}
There exists $c < \infty$ such that for $\epsilon $
sufficiently small, if $a,b, a', b' \in \partial D$ with
$|a-a'|, |b-b'|<\epsilon$, 
\[ 
 \wp \left[\mu_D(a',b'), \mu_D(a,b)\right]
  \leq c \, \sqrt \epsilon.
\]
  \end{cor}
  
 \begin{proof}
 We use the $F$ from the previous lemma and write
 $\Prob,\E$ for probabilities and expectations with
 respect to $\mu_D(a',b')$.   If
 $\gamma$ is a curve parametrized by Minkowski content, then \eqref{nov1.1} implies that
 \[    \rho(\gamma,F \circ \gamma) \leq c \left( T_\gamma + 1 \right)
  \, \epsilon , \]
 where $t_\gamma$ is the total content of $\gamma$.  In the case of SLE$_{2}$ we know that
 $\E[T_\gamma^2] < \infty$, so by
 the Chebyshev inequality 
 \[   \Prob\{T_\gamma \geq \epsilon^{-1/2}\}  \leq
    O(\epsilon). \]
  Hence there exists $c$ such that
  \[   \Prob\{\rho(\gamma,F \circ \gamma)  \geq
    c\, \sqrt \epsilon\} \leq \epsilon,\]
    which implies the bound in the statement.

 \end{proof}

  The metric $\rho$ is continuous under truncation in the
  following sense. We omit the easy proof.  
  \begin{lemma}
  Suppose $\gamma(t), \, t \in [0, T_\gamma],$
  is a curve and that $r,s$ are chosen so that $0 < r< T_\gamma - s 
   \leq  T_\gamma$.  Let $\tilde \gamma(t), \, t \in [0,  s-r],$  be defined by $\tilde \gamma(t)
   = \gamma(t+r)$.
   Then,
   \[  \rho(\gamma,\tilde \gamma)
     \leq \max\{ r, s\} +
      \max \left\{  \diam\left(\gamma[0,r]\right)
       ,\diam\left(\gamma[t_\gamma -s,t_\gamma]\right)\right\}.\]
   \end{lemma}
       
\begin{lemma}\label{lemma-oct-20}  There exists $c < \infty$ such that 
the following holds.  Suppose  
   $D$ is a domain with
distinct boundary points $a,b$.
Suppose $\epsilon >0$  and $f:D \rightarrow
f(D)$ is a conformal transformation that extends to
a homeomorphism of $\overline D$ satisfying
$|f(z) - z| \leq \epsilon$ for all $z \in
D$.
Suppose there exists $V \subset D$ such that
\begin{equation}  \label{oct23.2}
   G_D( D \setminus V;a,b) + G_{f(D)}
  \left(f(D \setminus V); f(a), f(b)\right) \leq  \epsilon , 
  \end{equation}
and such that for $z \in V$, 
\begin{equation}  \label{oct23.3}
    \left|\log |f'(z)|\right|  \leq \epsilon . 
     \end{equation}
Suppose also that
\begin{equation}  \label{oct23.1}
    G_D(D;a,b) 
  \leq  K . 
  \end{equation}
Then,
\[      \wp_\rho\left[\mu_D(a,b),
   \mu_{f(D)}(f(a),f(b)) \right] \leq c \, 
   (K+1) \sqrt{\epsilon}. \]
   \end{lemma}

 \begin{proof} We use natural parametrization  throughout
 this proof. Since we consider SLE$_{2}$ this is the same as parametrizing by $5/4$-dimensional Minkowski content.  We write $\Prob$ for probabilities
 under the measure $\mu_D(a,b)$.

  If $\gamma$
 is a curve from $a$ to $b$ in $D$, we write $f \circ \gamma$
 for the corresponding curve in $f(D)$, parameterized by natural time; more specifically we have
 \[   (f \circ \gamma)(\phi(t)) = f(\gamma(t)), \]
 where
 \[   \phi(t) = \int_0^t |f'(\gamma(s))|^{5/4} \, ds . \]
 We write $\gamma_f$ for the image curve without
 the change of parametrization,
 \[  \gamma_f(t) = f( \gamma(t)).\] 
 For any curve $\gamma$ we write $T_\gamma$ for the total content of the curve, that is, the total time
 duration in the natural parametrization.  In particular,
 \[   T_{f \circ \gamma}  
  = \int_0^{T_\gamma} |f'(\gamma(s))|^{5/4} \, ds . \]
  Let
  \[   T_{\gamma, V^c} = \int_0^{T_\gamma} 1\{\gamma(s) \in V^c\}
   \, ds =  \Cont\left[\gamma  \setminus V \right]. \] 
 Using the identity reparametrization, by the assumption on $f$, we see that
 \[  \rho(\gamma,\gamma_f) \leq \sup_{0 \leq s \leq T_\gamma}
   |\gamma(s) - f(\gamma(s))| \leq \epsilon.\]
On the other hand, using the reparametrizaton $\phi$, 
\[  \rho(\gamma_f,f\circ \gamma)
  \leq \sup_{0 \leq t \leq T_\gamma}
     | \phi(t) - t|, \]
 and hence
 \[     \rho(\gamma,f \circ \gamma)
     \leq \epsilon+ \sup_{0 \leq t \leq T_\gamma}
     | \phi(t)-t|. \]

To give an upper bound on the Prokhorov distance
between   
 $\mu_D(a,b)$ and $\mu_{f(D)}
(f(a),f(b))$  we use the coupling $\gamma \longleftrightarrow
f \circ \gamma$.   Then, we have for all $\delta >0$,
\[  \rho\left[\mu_D(a,b),\mu_{f(D)}
(f(a),f(b))\right] \leq \delta + \epsilon 
 + \Prob\left\{\sup_{0 \leq t \leq T_\gamma}
   |\phi(t) - t| \geq \delta \right\}.\]

Using
\eqref{oct23.2} and \eqref{oct23.1}, we see that
\[   \E\left[T_{\gamma, V^c} + T_{f\circ \gamma,f(V)^c}\right]
\leq \epsilon, \;\;\;\;
  \E\left[T_\gamma + T_{f \circ \gamma}\right] \leq   3K.\]
  Hence, except perhaps on an event of probability $O(\sqrt \epsilon)$,
\begin{equation}  \label{oct23.4}
  T_{\gamma,V^c} + T_{f\circ \gamma,f(V)^c} \leq \sqrt \epsilon,
\;\;\;\; T_\gamma   \leq K /\sqrt \epsilon.
\end{equation}
Using \eqref{oct23.3}, 
\begin{eqnarray*}
 \int_0^{\phi(t)} 1\{f \circ \gamma(s) \in f(V)\} \, ds
  & = & 
 \int_0^t  1\{\gamma(s) \in V\} |f'(\gamma(s))|^{5/4}\, ds\\
  & = & [1 + O(\epsilon)] 
 \int_0^t  1\{\gamma(s) \in V\} \, ds
 \end{eqnarray*}
Also, if $t \in [0, T_{\gamma}]$,
\[  \left|t -  \int_0^t  1\{\gamma(s) \in V\}  \, ds\right|
  \leq  T_{\gamma,V^c} , \]
\[ \left|
 \phi(t) -  \int_0^{\phi(t)} 1\{f \circ \gamma(s) \in f(V)\} 
 \, ds \right|  \leq  T_{f\circ \gamma,
    f(V)^c}.\]
Combining these estimates, we see that
\[  \sup_{0 \leq t \leq T_\gamma}
   |\phi(t) - t|  \leq c \, \epsilon\, T_\gamma +
      T_{\gamma,V^c} + T_{f \circ \gamma,
      f(V)^c}.\]
Therefore, 
 on the event that
\eqref{oct23.4} holds we have 
\[  \sup_{0 \leq t \leq T_\gamma}
   |\phi(t) - t|  \leq c \, 
     (K+1) \, \sqrt \epsilon.\]

\end{proof}

\begin{cor}  \label{mar5.cor3}
Under the assumptions of Theorem  \ref{thm:main_complete}, 
there exists $c$ such that for every $N$,
\[   \wp_\rho\left[\mu_D(a',b'),
 \mu_{\check D}(\check a, \check b)\right]
   \leq c \, \left[N^{-5/16} + |a' - \check a|^{1/2}
    + |b' - \check b|^{1/2}\right].\]
\end{cor}

 This estimate is not optimal but suffices
 for our purposes.

\begin{proof} 
 Let  $\psi : \check{D} \to D$ be the conformal
 transformation from Lemma \ref{mar5.lemma1}.
We let $V = \{z \in \check D: \dist(z,
\partial \check D) \geq 1/\sqrt N\}.$
Using Part $2$ of Lemma \ref{feb22.lemma1} with $\delta = N^{-1/2}$, we see that
\[  G_{\check D}(V^c;\check a, \check b)
  + G_{D}(\psi(V)^c; \psi(\check a), \psi(\check b) ) \leq c \, N^{-5/8}.\]
Also $G_{\check D}(V^c;\check a, \check b)$
is uniformly bounded in $N$.  Hence, Lemma~\ref{lemma-oct-20} yields  
\[ \wp_\rho\left[\mu_{\check{D}}(\check a, \check b),
 \mu_{ D}(\psi(\check a), \psi(\check b))\right]
   \leq c \, N^{-5/16}.\] 
 We then use Corollary \ref{mar5.cor1} to see that
 \[
 \wp_\rho\left[\mu_{D}( a',  b'),
 \mu_{ D}(\psi(\check a), \psi(\check b))\right]
\le c \left(|\psi(\check a) - a'|^{1/2} + |\psi(\check b) - b'|^{1/2}  \right).
 \]
 But using the properties of $\psi$ we also have $|\psi(\check{a}) - \check{a}| + |\psi(\check{b}) - \check{b}| \le c \, (\log N)/N$.
\end{proof}
 
We will also consider truncated measures.  This
must be done separately for SLE and LERW but
the argument is essentially the same.  We will
do the SLE case considering the measure
$\mu_{\check D}(\check a, \check b)$.  
  Suppose for each $\gamma$, there
is a time $t_1  \leq T_\gamma$ such that
$|\gamma(t) - b| \leq r$ for $t \geq t_1$.
If $\gamma_1$ denotes the truncated curve, 
$\gamma_1(s) = \gamma(s), 0 \leq s  \leq t_1$,
then
\[   \rho(\gamma,\gamma_1) \leq    
         r + (T_\gamma - t_1).\]
In particular if we take a random time
$\tau$ for the Brownian motion and let
$\mu_\tau$ denote the measure induced
by $\mu_{\check D}(\check a, \check b)$
by truncating at $\tau$, we have
\[  \wp_\rho\left[
\mu_\tau, \mu_{\check D}(\check a, \check b)
\right] \leq 2 (\epsilon \vee \delta)  \]
provided that $\epsilon,\delta$ are chosen so that
\[    \Prob\left\{\diam\left(\gamma[\tau,T_\gamma]
\right) \geq \epsilon \right\} \leq {\delta},\]
\[  \Prob\left\{T_\gamma - t
 \geq  \epsilon \right\} \leq {\delta}.\]
 Here $\Prob$ denotes probabilities with
 respect to the measure     
 $\mu_{\check D}(\check a, \check b)$.  
The first estimate is an SLE estimate about
continuity at the endpoint, see \cite{lawler_field}, and the second can be
obtained from Markov's inequality after
estimating the expected Minkowski content
in the set $\{z: |z - \check b| \leq r\}$.

For LERW the estimate for the number of points
visited in $\{z: |z-\check b| \leq r \}$ is the same.  We
use the following estimate.
\begin{itemize}
\item Let $\tau$ denote the first $n$ such that
$|\eta_n - Nb| \leq rN$.  Then the probability
that there exists a later point of the LERW
distance $RN$ away from $Nb$
 is bounded above by $c(r/R)^2$.
\end{itemize}
This estimate is not optimal; indeed, this estimate
is true for the random walk excursion which implies
it is valid for the LERW.  We omit the details, but
sketch the idea of the proof.  The probability that
a random walk starting distance $rN$ of $bN$ gets
distance $RN$ away is $O(r/R)$ by a gambler's ruin
estimate.  Also the Poisson kernel farther away
is $r/R$ times the Poisson kernel closer and hence
the probability that the excursion ($h$-process)
goes out that far is $O((r/R)^2)$.

\appendix
\section{Summary of notation}
\begin{center}
    \begin{tabular}{  l  | p{7cm} }
    
    {{ \bf Notation}} & {{\bf Short description}} \\ 
    \hline
    $N$ &  Large integer; $N^{-1}$ defines the mesh size. \\ 
    $(D,a',b')$ &   An analytic simply connected domain with distinct boundary points $a',b'$. \\ 
         $\phi(z)$ & Some fixed conformal map $D \to \Half$ with $\phi(a')=0, \, \phi(b') = \infty$. \\ 
    $(A,a,b)$ &  $A$: A discrete domain in $\ZZ^2$ with boundary edges $a$ and $b$.  \\  
    $\Square_z$ & Square with axis-parallel sides of side-length $1$, centered at $z$.\\
     $D_A$ &  ``Union of squares'' domain built from $A$: $D_A = \text{int} \cup_{x \in A} \Square_x$.   \\   
     $F(z)$ & Some fixed choice of conformal map from $(D_A,a,b)$ to $\Half$ with $F(a) = F_D(a) = 0$. \\  
     $S_{A,a,b}(z)$  &  $\sin\left[ \arg F(z) \right]$.\\ 
     $r_A(z)$ &  The conformal radius of $D_A$ seen from $z$. \\ 
     $G_{D_A}(z; a,b)$ &  $\tilde c \, r_A(z)^{-3/4} S_A(z)^3$, SLE$_2$ Green's function for $(D_A,a,b)$. \\ 
     $\eta, \, \check{\eta}$ &  LERW on $\ZZ^2$ and $\frac{1}{N}\ZZ^2$, respectively. \\ 
    $(A_n, a_n,b)$ &  Sequence of LERW domains with mesoscopic capacity increments. \\ 
     $\check{D}=\check{D}_A$ & The scaled domain $N^{-1} D_A$, which approximates $D$. \\   
     $\check{\phi}(z)$ & $\check{\phi}(z):=F(Nz)$ \\ 
       
     $g_n^\LERW(z), \, F_n^{\LERW}(z)$ & Uniformizing map  $g_n^\LERW(z): F(D_{A_n}) \to \Half$ and $F_n^{\LERW}(z) = (g_n^\LERW(z) \circ F)(z)-U_n$. \\ 
     $\gamma, \, \check{\gamma}, \, \hat{\gamma}$ &  SLE$_2$ in $\Half$; SLE$_2$ in $\check{D}$; SLE$_2$ in $D_A$. \\ 
     $g_t^\SLE(z), \, F_t^\SLE(z)$ & Uniformizing map $g_t^\SLE(z): \Half \setminus \gamma_t \to \Half$ and $F_t^{\SLE} = (g_t^{\SLE} \circ F)(z)-W_t$, where $W$ is the Brownian motion generating $(g_t^\SLE)$.\\
     &\\
     \hline
    \end{tabular}
\end{center}

\end{document}